\documentclass[twoside,11pt]{article}

\usepackage{blindtext}
\usepackage{graphicx}
\usepackage{amsmath,amssymb,euscript ,yfonts,psfrag,latexsym,dsfont,graphicx,bbm,color,amstext,wasysym,epsfig,subfig,parskip,textcomp}
 
\usepackage{xspace}
\usepackage{geometry} 
\usepackage[overload]{empheq}
\usepackage{dsfont} 
\usepackage{stmaryrd}
\usepackage{tikz} 
\usetikzlibrary{arrows}

\DeclareMathOperator*{\argmax}{arg\,max}

\DeclareMathOperator*{\argmin}{arg\,min}

\usepackage[abbrvbib, preprint]{jmlr2e}

\usepackage{lastpage}
\jmlrheading{23}{2022}{1-\pageref{LastPage}}{1/21; Revised 5/22}{9/22}{21-0000}{Stéphanovitch}

\ShortHeadings{IPMs on manifolds: interpolation inequalities and optimal inference}{Stéphanovitch}
\firstpageno{1}

\begin{document}

\title{Integral Probability Metrics on submanifolds:\\ interpolation inequalities and optimal inference}

\author{\name Arthur Stéphanovitch \email stephanovitch@dma.ens.fr\\
       \addr D\'epartement de Math\'ematiques et Applications\\
       Ecole Normale Sup\'erieure, Université PSL, CNRS\\
       F-75005 Paris, France}

\editor{My editor}

\maketitle

\begin{abstract}We study interpolation inequalities between Hölder Integral Probability Metrics (IPMs) in the case where the measures have densities on closed submanifolds. Precisely, it is shown that if two probability measures $\mu$ and $\mu^\star$ have $\beta$-smooth densities with respect to the volume measure of some submanifolds $\mathcal{M}$ and $\mathcal{M}^\star$ respectively, then the Hölder IPMs $d_{\mathcal{H}^\gamma_1}$ of smoothness $\gamma\geq 1$ and $d_{\mathcal{H}^\eta_1}$ of smoothness $\eta>\gamma$, satisfy $d_{ \mathcal{H}_1^{\gamma}}(\mu,\mu^\star)\lesssim  d_{ \mathcal{H}_1^{\eta}}(\mu,\mu^\star)^\frac{\beta+\gamma}{\beta+\eta}$, up to logarithmic factors. We provide an application of this result to high-dimensional inference. These functional inequalities turn out to be a key tool for density estimation on unknown submanifold. In particular, it allows to build the first estimator attaining optimal rates of estimation for all the distances $d_{\mathcal{H}_1^\gamma}$, $\gamma \in [1,\infty)$ simultaneously.
\end{abstract}

\begin{keywords}
  integral probability metrics, interpolation inequality, manifold data, distribution estimation, minimax rate.
\end{keywords}

\section{Introduction}
The development of methods for estimating probability measures from data has been a major focus of contemporary statistics and machine learning \citep{Bishop2006Pat}. This pursuit has given rise to thriving research fields such as generative models \citep{ruthotto2021introduction}, reinforcement learning \citep{szepesvari2022algorithms} and geometrical inference \citep{chazal2011geometric}. In order to evaluate the precision of the estimation, many methods use Integral Probability Metrics (IPMs) \citep{IPMsMuller} as the distances to compare probability measures. It consists in choosing a class of functions $\mathcal{F}$ and looking at the distance between two probability measures $\mu,\mu^\star$ defined by 
\begin{equation}\label{eq:genralIPM}
    d_{\mathcal{F}}(\mu,\mu^\star):=\sup \limits_{f\in \mathcal{F}} \left|\int f(x)d\mu(x) -\int f(x)d\mu^\star(x)  \right|.
\end{equation}
Notable examples of IPMs include the total
variation distance \citep{verdu2014total} ($\mathcal{F}$ is the class of functions taking value in $[-1,1]$), the Wasserstein
distance \citep{villani2009optimal} ($\mathcal{F}$ is the class of 1-Lipschitz functions) and Maximum Mean Discrepancies \citep{smola2006maximum} ($\mathcal{F}$ is the unit
ball of a RKHS). In this paper we focus on the Hölder IPMs that have been used in generative models (\cite{arjovsky2017wasserstein}, \cite{chakraborty2024statistical}) and geometric learning  \citep{tang2023minimax}.
For $\eta >0$, $\mathcal{X},\mathcal{Y}$ two subsets of Euclidean spaces and $f=(f_1,...,f_p)\in C^{\lfloor \eta \rfloor}(\mathcal{X},\mathcal{Y})$ the set of $\lfloor \eta \rfloor:=\max \{k\in \mathbb{N}_0 |\ k\leq \eta\}$  times differentiable functions, denote $\partial^\nu f_i = \frac{\partial^{|\nu|}f_i}{\partial x_1^{\nu_1}...\partial x_d^{\nu_d}}$ the partial differential operator  for any multi-index $\nu = (\nu_1,...,\nu_d)\in \mathbb{N}_0^d$ with $|\nu|:=\nu_1+...+\nu_d\leq \lfloor \eta \rfloor$. Write $\|f_i\|_{\eta-\lfloor \eta \rfloor}=\sup \limits_{x\neq y} \frac{f_i(x)-f_i(y)}{\min\{1,\|x-y\|^{\eta - \lfloor \eta \rfloor}\}}$ and let 
\begin{align*}
\mathcal{H}^\eta_K(\mathcal{X},\mathcal{Y})=\Big\{& f  \in C^{\lfloor \eta \rfloor}(\mathcal{X},\mathcal{Y}) \ \big| \max \limits_{i} \sum \limits_{|\nu|\leq \lfloor \eta \rfloor} \|\partial^\nu f_i\|_{L^\infty(\mathcal{X},\mathcal{Y})} + \sum \limits_{|\nu|  = \lfloor \eta \rfloor} \|\partial^\nu f_i\|_{\eta-\lfloor \eta \rfloor}   \leq K\Big\}
\end{align*}
denote the ball of radius $K$ of the Hölder space $\mathcal{H}^\eta(\mathcal{X},\mathcal{Y})$, the set of functions $f:\mathcal{X}\rightarrow \mathcal{Y}$ of regularity $\eta$. In this work we are interested in comparing the Holder IPMs
\begin{equation}\label{eq:HolderIPM}
    d_{\mathcal{H}^\eta_1}(\mu,\mu^\star):=\sup \limits_{f\in \mathcal{H}^\eta_1}\int f(x)d\mu(x) -\int f(x)d\mu^\star(x)  
\end{equation}
of different regularities $\eta$. The popularity of the Hölder IPM comes from its equivalence with the Wasserstein metric in the case $\eta=1$. This metric enjoys several useful properties including the metrization of weak convergence of measures \cite{villani2021topics}. Unfortunately, the Wasserstein distance has been shown to be computationally expensive \citep{peyre2019computational}. Then, the need of finding smaller classes than the Lipschitz class in the IPM \eqref{eq:genralIPM} arises. Ideally one wants to have a class $\mathcal{F}$ that is smaller than the Lipschitz class but that still controls the Wasserstein distance. Classical interpolation inequalities \citep{lunardi2018interpolation} allow to show the connection between the different IPMs: if $\mu$ and $\mu^\star$ have densities with respect to the Lebesgue measure and that belong to $\mathcal{H}^\beta_1$ for $\beta>0$, then for all $0<\gamma<\eta$,  we have 
\begin{equation}\label{eq:classcialinterpineq}
        d_{ \mathcal{H}_1^{\gamma}}(\mu,\mu^\star)\lesssim  d_{ \mathcal{H}_1^{\eta}}(\mu,\mu^\star)^\frac{\beta+\gamma}{\beta+\eta},
\end{equation}
up to logarithmic terms (see Theorem \ref{eq:ineqinfulldim}). These inequalities are key in the estimation of smooth densities as they allow to show that if an estimator $\hat{\mu}$ of $\mu^\star$ attains optimal rates of estimation for the $ d_{ \mathcal{H}_1^{\eta}}$ distance, then it also attains optimal rates for the $ d_{ \mathcal{H}_1^{\gamma}}$ distance for all $\gamma\in (0,\eta)$ \citep{stephanovitch2023wasserstein}.

Recent use of high dimensional data have shed light on the fact that the optimal rates for the estimation of probability measures, degenerate exponentially with respect to the dimension of the data \citep{koppen2000curse}. This is the so called "curse of dimensionallity" and has led researcher to forsake density estimation to rather focus on measures with low intrinsic dimensional structure. In this context, a natural setting is to consider measures having densities with respect to the volume measure of a submanifold \citep{divol2022measure}. However, classical interpolation inequalities like \eqref{eq:classcialinterpineq} do not apply anymore in this setting. In order to obtain efficient methods taking into account the low dimensional structure of the measures, it is then crucial to understand how these inequalities can be generalized.

\cite{stephanovitch2023wasserstein} have recently shown a generalization of these inequalities in the case of push forward measures. Supposing that $\mu$ and $\mu^\star$ have $\beta$-regular densities with respect to the volume measure of submanifolds $\mathcal{M}$ and $\mathcal{M}^\star$ respectively and that $\mu$ and $\mu^\star$ are push forward measures of $\beta+1$-regular maps from the $d$-dimensional torus $\mathbb{T}^d$, it is shown that inequality \eqref{eq:classcialinterpineq} still stands in the case $\gamma \in [1,\beta+1]$ and $\eta=\beta+1$. In this paper, this result is generalized to measures having $\beta$-regular densities with respect to the volume measure of any closed submanifolds (without the push forward assumption) and to any $0<\gamma<\eta$. This generalization allows measure estimation methods like \cite{stephanovitch2023wasserstein} to be generalized to the manifold setting. Precisely, we build the first estimator of density on unknown manifold that attains optimal rates for all the distances $d_{\mathcal{H}_1^\gamma}$, $\gamma \in [1,\infty)$ simultaneously. Its construction is more simpler than previous estimators like \cite{tang2023minimax} as its use of high regularity IPMs implicitly regularizes the data.

To motivate our result from a functional point of view, it is detailed in Section \ref{sec:classicalineq} how the generalization of inequality  
\eqref{eq:classcialinterpineq} is the extension of some interpolation inequalities between Besov norms of $L^2$ functions, to Besov norms of signed measures having low dimensional smooth structure.

The remainder of this paper is organised as follows. In Section \ref{sec:setting}, we set the notation, properly define our setting, and present the main results of the paper. In Section \ref{sec:result}, we give a walk-through of the proof by stating the key intermediate results that allow to obtain Theorem \ref{theo:theineq}. In Section \ref{sec:application} we present a direct application of Theorem \ref{theo:theineq} to the estimation of density on unknown manifold. In Section \ref{sec:example}, we detail a simple example to get some intuition on Theorem \ref{theo:theineq}. Technical points of the proofs are gathered in the supplementary material.

\section{Preliminaries and main results}\label{sec:setting}
\subsection{Setting}
Let us start by defining the precise setting of our work. The results focus on $\beta$-regular densities with $\beta\geq 1$ with respect to the volume measure of closed (i.e. compact without boundary) manifolds that are of dimension $d\in \mathbb{N}^\star$ and that are immersed in the ambient space $\mathbb{R}^p$ with $p>d$ an integer. 

For $\mathcal{M}$ a closed $d$-dimensional submanifold embedded in $\mathbb{R}^p$ and $x\in \mathcal{M}$, note $$U_x= B^p(x,K^{-1})\cap \mathcal{M}\ \text{ and }\ \varphi_x:U_x\rightarrow \mathcal{T}_x(\mathcal{M}), \ \varphi_x=\pi_{\mathcal{T}_x(\mathcal{M})}-x,$$ 
for $\pi_{\mathcal{T}_x(\mathcal{M})}$ the orthogonal projection onto the tangent space $\mathcal{T}_x(\mathcal{M})$ of $\mathcal{M}$ at the point $x$ and $B^p(x,K^{-1})$ the $p$-dimensional ball of radius $K^{-1}>0$ around $x$. Let us fix a certain $K>1$ and define the notion of regularity for manifold used in this work.

\begin{definition}\label{def:manifoldcond}
    A closed $d$-dimensional submanifold $\mathcal{M}$ is said to verify the $(\beta+1,K)$-manifold condition if $\mathcal{M}\subset B^p(0,K)$ and for all $x\in \mathcal{M}$, $\varphi_x:U_x\rightarrow \mathcal{T}_x(\mathcal{M})$ is a diffeomorphism and verifies
$$\varphi_x^{-1} \in \mathcal{H}^{\beta+1}_K(\varphi_x(U_x),U_x).$$
\end{definition}

Definition \ref{def:manifoldcond} is equivalent to a reach condition \citep{federer1959} for $\beta+1$-smooth submanifolds. The reach $r_{\mathcal{M}}$ of a manifold $\mathcal{M}$ corresponds to the largest $\epsilon\geq0$ such that its orthogonal projection $\pi_\mathcal{M}$ is well defined from $\mathcal{M}^\epsilon$ to $\mathcal{M}$ with $\mathcal{M}^\epsilon=\{x\in \mathbb{R}^p|\ d(x,\mathcal{M})< \epsilon\}$ the set of points having their distance to the manifold $d(x,\mathcal{M})$ smaller than $\epsilon$. Definition \ref{def:manifoldcond} is very general as any $\beta+1$-regular closed submanifold with reach uniformly bounded from below, verifies the $(\beta+1,K)$-manifold condition for a sufficiently large $K$ \citep{divol2022measure}. Let us now define the notion of regularity for densities used in the paper.

\begin{definition}\label{def:densitycond}
A probability measure $\mu$ supported on a submanifold $\mathcal{M}$ is said to verify the $(\beta,K)$-density condition if 
\begin{itemize}
    \item[i)] $\mu$ admits a density $f_{\mu}\in \mathcal{H}^{\beta}_K(\mathcal{M},\mathbb{R})$ with respect to the volume measure on $\mathcal{M}$,
    \item[ii)] $f_{\mu}$ is bounded below by $K^{-1}$.
\end{itemize}
\end{definition}
 
Supposing that the densities are bounded below is a classical assumption in the setting of density estimation on manifolds (\cite{divol2022measure}, \cite{tang2023minimax}). Here it is needed in order to obtain smooth optimal transport maps between some push forward measures in Proposition \ref{prop:keydecomp} using Caffarelli's regularity \citep{villani2009optimal}.   For a measure $\mu$ supported on a set $\mathcal{X}$ and a function $g$ defined on $\mathcal{X}$ and taking values in a set $\mathcal{Y}$, we write $g_{\# \mu}$ for the push forward measure of $\mu$ by $g$ i.e. $\forall f \in \mathcal{H}_1^0(\mathcal{X},\mathcal{Y})$, 
$$
\int_{\mathcal{Y}}f(x)dg_{\# \mu}(x)=\int_{\mathcal{X}}f(g(u))d\mu(u).
$$

\subsection{Main results}
Throughout, the quantities $C,C_\delta,C_{\delta,\eta}...$ represent constants that can vary from line to line and that only depend on $p,d,\beta,K$ and there underscore. Using Definitions \ref{def:manifoldcond} and \ref{def:densitycond}, let us state the main result of the paper.

\textbf{Theorem \ref{theo:theineq}:}
 Let $\mathcal{M},\mathcal{M}^\star$ be two submanifolds satisfying the $(\beta+1,K)$-manifold condition (Definition \ref{def:manifoldcond}) and $\mu,\mu^\star$ two probability measures satisfying the $(\beta,K)$-density condition (Definition \ref{def:densitycond}) on $\mathcal{M}$ and $\mathcal{M}^\star$ respectively. Then for all $1\leq \gamma\leq\eta$, we have
    \begin{align*}
        d_{\mathcal{H}^\gamma_1}(\mu,\mu^\star)\leq C\log\left(1+d_{\mathcal{H}^\eta_1}(\mu,\mu^\star)^{-1}\right)^{C_2} d_{\mathcal{H}^\eta_1}(\mu,\mu^\star)^{\frac{\beta+\gamma}{\beta+\eta}}. 
        \end{align*}

This result is sharp up to the logarithmic factor, meaning that the exponent $\frac{\beta+\gamma}{\beta+\eta}$ is optimal. The optimality of the exponent can be justified by three different arguments, the first one being that it is the same exponent as in the classical case (see Theorem \ref{eq:ineqinfulldim}). The second one is the example of Section \ref{sec:example}, where we explicit a sequence of measures such that $d_{\mathcal{H}^\gamma_1}(\mu_k,\mu^\star)\rightarrow 0$ when $k$ tends to the infinity and $d_{\mathcal{H}^\gamma_1}(\mu_k,\mu^\star)\geq C d_{\mathcal{H}^\eta_1}(\mu_k,\mu^\star)^{\frac{\beta+\gamma}{\beta+\eta}}$. Finally, the last one is more original as it is a statistical argument. Theorem \ref{theo:theineq} can be used to show that a certain estimator $\hat{\mu}$ of $\mu^\star$ using i.i.d data, attains optimal rates of convergence up to logarithmic factors \citep{stephanovitch2023wasserstein}. If the exponent on the $d_{\mathcal{H}^\eta_1}$ distance was larger than $\frac{\beta+\gamma}{\beta+\eta}$, this estimator would attain better rates than the optimal ones.

The logarithmic factor in Theorem \ref{theo:theineq} arises from the use of wavelet theory to describe Hölder spaces. Throughout the paper, we highlight the efficiency of the wavelet tool to describe smoothness as it allows to greatly simplify some geometric proofs. However, Hölder spaces with integer regularity exponent, can not be exactly described by wavelets \cite{haroske2006envelopes}. This leads to the appearance of the term $\log\left(d_{\mathcal{H}^\eta_1}(\mu,\mu^\star)^{-1}\right)^{C_2}$ in the inequality, as it allows to have injection of Besov into Hölder spaces. We do not keep track of the extra exponent $C_2$ to avoid unnecessary heaviness, as the obtained exponent is  very likely suboptimal. 

Theorem \ref{theo:theineq} is an improvement on Theorem 5.1 in \cite{stephanovitch2023wasserstein} which requires the additional assumptions that the measures are push forward measures from the torus and that $\eta=\beta+1$. Precisely,  it is supposed that there exist $g,g^\star \in \mathcal{H}^{\beta+1}_K(\mathbb{R}^d \slash \mathbb{Z}^d,\mathbb{R}^p)$ embeddings, such that $\mu=g_{\# \lambda^d}$ and $\mu^\star=g^\star_{\# \lambda^d}$. In particular, this model can only describe manifolds $\mathcal{M}$ and $\mathcal{M}^\star$ that are smooth deformations of the torus. On the other hand, our extension of the result for any $\eta\geq \gamma$, shows that the inequality behaves like in the classical case and the fact that the measures are on submanifolds (with possibly a complex topology) does not impact the IPMs. Theorem \ref{theo:theineq} shows that although the measures do not have any regularity from the point of view of the ambient space $\mathbb{R}^p$, their low dimensional regularity still intervenes the same way in their discrimination through IPMs. We also extend the result to the simpler case $\gamma \in (0,1)$ in Proposition \ref{prop:ineqgammaleqone}, where in this case the inequality is naturally different.

The motivation behind Theorem \ref{theo:theineq} is to  extend the results of \cite{stephanovitch2023wasserstein} to closed submanifolds that are not necessarily smooth deformations of the torus. Specifically, it is built in Section \ref{sec:application} an estimator of density on unknown submanifold attaining optimal rates.

Let us write $\mathcal{F}$ for the set of probability measures $\mu$ such that there exists a submanifold $\mathcal{M}$ satisfying the $(\beta+1,K)$-manifold condition and that $\mu$ satisfies the $(\beta,K)$-density condition on $\mathcal{M}$. Additionally, we write $g(n)=\tilde{O}( f(n))$ for $g(n)\leq C \log(n)^{C_2} f(n)$.

\textbf{Theorem \ref{theo:estimminimax} :} Let $n\in \mathbb{N}_{>0}$, $\mu^\star \in \mathcal{F}$ and $(X_1,...,X_n)$ an i.i.d. sample of law $\mu^\star$. Then for $\mu^\star_n:= \frac{1}{n}\sum_{i=1}^n \delta_{X_i}$, the estimator 
\begin{equation}\label{eq:lememeesti}
\hat{\mu} \in \argmin_{\mu \in \mathcal{F}} d_{\mathcal{H}^{d/2}_1}(\mu,\mu^\star_n),
\end{equation}
satisfies for all $\gamma\geq 1$,
  $$\sup_{\mu^\star\in \mathcal{F}}\ \mathbb{E}_{X_i}[d_{\mathcal{H}^{\gamma}_1}(\hat{\mu},\mu^\star)] =\tilde{O}\left( \inf_{\hat{\theta}\in \Theta}\ \sup_{\mu^\star\in \mathcal{F}}\ \mathbb{E}_{X_i}[d_{\mathcal{H}^{\gamma}_1}(\hat{\theta},\mu^\star)]\right),$$
where $\Theta$ denotes the set of all possible estimators  of $\mu^\star$ based on n sample.\vspace{0.2cm}

To the best of our knowledge, the estimator \eqref{eq:lememeesti} is the first one attaining optimal rates  for all the distances $d_{\mathcal{H}_1^\gamma}$, $\gamma \in [1,\infty)$ simultaneously. This estimator uses high regularity IPMs in the minimized loss which implicitly regularizes the data. Its construction is much simpler than previous optimal estimators (like \cite{tang2023minimax} and \cite{liang2021generative}) because it does not need to compute a complex regularization of the empirical measure as the regularization is made through the use of high regularity IPMs.

\subsection{Besov spaces and additional notations}
Throughout the paper, we use the connection between Hölder and Besov spaces extensively \citep{Triebel:1603865}. Let us define the Besov spaces trough their wavelet characterization. 
Let $\psi,\phi\in \mathcal{H}^{\beta+2}(\mathbb{R},\mathbb{R})$ be a compactly supported \emph{scaling} and \emph{wavelet} function respectively (see Daubechies wavelets \citep{daubechies1988orthonormal}). For ease of notation, the functions $\psi,\phi$ will be written $\psi_0,\psi_1$ respectively. Then for $j\in \mathbb{N},l \in \{1,...,2^p-1\}, w \in \mathbb{Z}^p$, the family of functions 
$$\psi_{0w}(x) = \prod \limits_{i=1}^p \psi_{0}(x_i-w_i) \ \text{ , } \ \psi_{jlw}(x) = 2^{jp/2}\prod \limits_{i=1}^p \psi_{l_i}(2^{j}x_i-w_i)
$$
form an orthonormal basis of $L^2(\mathbb{R}^p,\mathbb{R})$ (with $l_i$ the $i$-th digit of the base-2-decomposition of $l$).
Let $q_1,q_2\geq 1,s>0,b\geq 0$ be such that $\beta+2>s$. The Besov space $\mathcal{B}^{s,b}_{q_1,q_2}(\mathbb{R}^p,\mathbb{R})$ consists of functions $f$ that admit a wavelet expansion in $L^2$:
$$
f(x)=\sum \limits_{w\in \mathbb{Z}^p} \alpha_f(w)\psi_{0w}(x) + \sum \limits_{j=0}^\infty \sum \limits_{l=1}^{2^p-1}\sum \limits_{w\in \mathbb{Z}^p} \alpha_f(j,l,w)\psi_{jlw}(x)
$$
equipped with the norm
\begin{align*}
\|f\|_{\mathcal{B}^{s,b}_{q_1,q_2}}= &\Biggl(\left(\sum \limits_{w\in \mathbb{Z}^p} |\alpha_f(w)|^{q_1}\right)^{q_2/q_1}\\
& +\sum \limits_{j=0}^\infty 2^{jq_2(s+p/2-p/q_1)}(1+j)^{bq_2} \sum \limits_{l=1}^{2^p-1} \Big(\sum \limits_{w\in \mathbb{Z}^p} |\alpha_f(j,l,w)|^{q_1}\Big)^{q_2/q_1}\Biggl)^{1/q_2}.
\end{align*}
with the usual modification for $q_1,q_2=\infty$.
Note that for $b=0$, $\mathcal{B}^{s,0}_{q_1,q_2}$ coincides with the classical Besov space $\mathcal{B}^{s}_{q_1,q_2}$ \citep{giné_nickl_2015}.
For simplicity of notation, we write
$$
f(x)= \sum \limits_{j=0}^\infty \sum \limits_{l=1}^{2^p}\sum \limits_{w\in \mathbb{Z}^p} \alpha_f(j,l,w)\psi_{jlw}(x)
$$
with the convention that $\psi_{02^pw}=\psi_{0w}$ and for all $j\geq 1$, $\psi_{j2^pw}=0$. The Besov spaces can be generalized for any $s\in \mathbb{R}$ as a subspace of the space of tempered distribution $\mathcal{S}^{'}(\mathbb{R}^p)$. Indeed for $f \in \mathcal{S}^{'}(\mathbb{R}^p)$, writing $\alpha_f(j,l,w)=\langle f,\psi_{jlw}\rangle$, the Besov space for $s\in \mathbb{R}$ is defined as
$$\mathcal{B}^{s,b}_{q_1,q_2}=\{f\in \mathcal{S}^{'}(\mathbb{R}^p) | \|f\|_{\mathcal{B}^{s,b}_{q_1,q_2}}<\infty\}.$$

The same way, we write $\mathcal{B}^{s,b}_{q_1,q_2}(K)=\{f\in \mathcal{S}^{'}(\mathbb{R}^p) | \|f\|_{\mathcal{B}^{s,b}_{q_1,q_2}}\leq K\}.$
 In the following we will use intensively the connection between Hölder and Besov spaces.
\begin{lemma}\label{lemma:inclusions} (Proposition 4.3.23 \cite{giné_nickl_2015}, (4.63)  \cite{haroske2006envelopes}) If $\alpha>0$ is a non integer, then 
$$\mathcal{H}^\alpha(\mathbb{R}^p,\mathbb{R})=\mathcal{B}^\alpha_{\infty,\infty}(\mathbb{R}^p,\mathbb{R})$$
with equivalent norms. If $\alpha\geq 0$ is an integer, then $\forall \epsilon>0$
$$\mathcal{B}^{\alpha,1+\epsilon}_{\infty,\infty}(\mathbb{R}^p,\mathbb{R}) \xhookrightarrow{} \mathcal{H}^\alpha(\mathbb{R}^p,\mathbb{R})\xhookrightarrow{} \mathcal{B}^\alpha_{\infty,\infty}(\mathbb{R}^p,\mathbb{R}),$$
where we write $A\xhookrightarrow{} B$ if the function space $A$ compactly injects in the function space $B$.
\end{lemma}

This lemma states that Besov infinity and Hölder spaces coincides except when the regularity exponent is an integer. In particular, it implies that Theorem \ref{theo:theineq} extends to Besov infinity IPMs up to additional logarithmic terms for integer regularity exponents.

We write $\langle \cdot, \cdot \rangle$ the dot product on $\mathbb{R}^p$, $\|x\|$ the Euclidean norm of a vector $x$. For $a,b\in \mathbb{R}$, $a\wedge b$ and $a\vee b$ denote the minimum and maximum value between $a$ and $b$ respectively. We write $\text{Lip}_1$ for the set of 1-Lipschitz functions. The support of a function $f$ is denoted by $supp(f)$. We  denote by $\text{Id}$ the identity application from a Euclidean space to itself.

For any map $f:\mathbb{R}^k\rightarrow \mathbb{R}^l$ we denote by $\nabla f$ the differential of $f$ and by $\|\nabla f(x)\|$ its operator norm at the point $x$. We denote $(\nabla f(x))^\top$ its transpose matrix and if $k\leq l$, $\lambda_{\min}((\nabla f(x) )^\top \nabla f(x))$ corresponds to the smallest eigenvalue of the matrix $(\nabla f(x) )^\top \nabla f(x)$. We write $\int_{\mathcal{M}}f(x)d\lambda_{\mathcal{M}}(x)$ for the integration of a function $f:\mathcal{M}\rightarrow \mathbb{R}$ with respect to the volume measure on $\mathcal{M}$ i.e. the $d$-dimensional Hausdorff measure. The Hausdorff distance between two submanifold $\mathcal{M},\mathcal{M}^\star$ is denoted by $\mathbb{H}(\mathcal{M},\mathcal{M}^\star)$.

\section{Walk-through of the construction of Theorem \ref{theo:theineq}}\label{sec:result}
In this section the main sub-results allowing to prove  Theorem \ref{theo:theineq} are detailed. We first prove the classical inequality in the full dimensional setting. Then, it is explained how it can be generalized to the submanifold case.
\subsection{The classical inequality}\label{sec:classicalineq}
The classical setting corresponds to the case where the two measures $\mu,\mu^\star$ each have a squared integrable density with respect to the Lebesgue measure of the ambient space $\mathbb{R}^p$. This setting is called classical as the proof boils down to applying Hölder's inequality on the wavelet coefficients of the functions. This is a natural technique one could use for any interpolation between Banach spaces that are defined by the growth of coefficients in a given basis (see for example \cite{wang2021jackson}). 

To prove it, let us first define an operator that allows to change the regularity of a function by modifying its wavelet coefficients.

\begin{definition}
For a tempered distribution $f \in \mathcal{B}^{s,b}_{\infty,\infty}(\mathbb{R}^p,\mathbb{R})$, define the tempered distribution $\Gamma^{\gamma,c}(f)\in \mathcal{B}^{s+\gamma,b+c}_{\infty,\infty}(\mathbb{R}^p,\mathbb{R})$ by its wavelets coefficients 
$$\langle\Gamma^{\gamma,c}(f),\psi_{jlz}\rangle= 2^{j\gamma}(1+j)^c\langle f,\psi_{jlz}\rangle.$$
\end{definition}
The tempered distribution $\Gamma^{\gamma,c}(f)$ is a "regularization" of $f$ if $\gamma<0$ or $\gamma=0$ and $c<0$. For ease of notation, we will write $\Gamma^\gamma$ instead of $\Gamma^{\gamma,0}$.

Now take two densities $f,f^\star \in L^2(\mathbb{R}^p,\mathbb{R})$ and a potential $h\in  \mathcal{B}^{s}_{\infty,\infty}(\mathbb{R}^p,\mathbb{R})$ to compare them. We call $h$ a potential in reference to Kantorovich potentials in optimal transport \citep{santambrogio2015optimal}. The family of wavelets functions $(\psi_{jlw})_{jlw}$ being an orthonormal basis of $L^2$, we can write for $\tau>0$ and $q\in (0,1), q^\star=(1-1/q)^{-1}$,
\begin{align*}
    \int_{\mathbb{R}^p} & h(x)(f(x)-f^\star(x))  d\lambda^p(x)  
=  \sum \limits_{j=0}^\infty \sum \limits_{l=1}^{2^p} \sum \limits_{w \in \mathbb{Z}^p}\alpha_{h}(j,l,w)(\alpha_{f}(j,l,w)-\alpha_{f^\star}(j,l,w))\\
 \leq & \sum \limits_{j=0}^\infty \sum \limits_{l=1}^{2^p} \sum \limits_{w \in \mathbb{Z}^p}2^{-j\frac{\tau}{q}+j\frac{\tau}{q}}|\alpha_{h}(j,l,w)(\alpha_{f}(j,l,w)-\alpha_{f^\star}(j,l,w))|^{1/q+1/q^\star}.
 \end{align*}
 Then, applying Hölder's inequality we get
  \begin{align}\label{align:prrmierhodl}
  \int_{\mathbb{R}^p} & h(x)(f(x)-f^\star(x))  d\lambda^p(x)\nonumber\\   
 \leq & \left(\sum \limits_{j=0}^\infty \sum \limits_{l=1}^{2^p} \sum \limits_{w \in \mathbb{Z}^p}2^{-j\tau}|\alpha_{h}(j,l,w)(\alpha_{f}(j,l,w)-\alpha_{f^\star}(j,l,w))|\right)^{\frac{1}{q}}\nonumber\\
 & \times \left(\sum \limits_{j=0}^\infty \sum \limits_{l=1}^{2^p} \sum \limits_{w \in \mathbb{Z}^p}2^{\tau\frac{q^\star}{q}}|\alpha_{h}(j,l,w)(\alpha_{f}(j,l,w)-\alpha_{f^\star}(j,l,w))|\right)^{\frac{1}{q^\star}}.
\end{align}
Let us take $\tilde{\Gamma}^{-\tau}(h)\in \mathcal{B}^{s+\tau}_{\infty,\infty}$ and $\tilde{\Gamma}^{\frac{q^\star}{q}\tau}(h)\in \mathcal{B}^{s-\frac{q^\star}{q}\tau}_{\infty,\infty}$ defined through their wavelet coefficients:
$$
\alpha_{\tilde{\Gamma}^{-\tau}(h)}(j,l,w)=S(j,l,w)\alpha_{\Gamma^{-\tau}(h)}(j,l,w)
$$ 
and 
$$
\alpha_{\tilde{\Gamma}^{\frac{q^\star}{q}\tau}(h)}(j,l,w)=S(j,l,w)\alpha_{\Gamma^{\frac{q^\star}{q}\tau}(h)}(j,l,w)
$$
with $S(j,l,w) \in \{-1,1\}$ the sign of $\alpha_{h}(j,l,w)(\alpha_{f}(j,l,w)-\alpha_{f^\star}(j,l,w))$. Then, from \eqref{align:prrmierhodl} we obtain
\begin{align*}
\int_{\mathbb{R}^p} & h(x)(f(x)-f^\star(x))  d\lambda^p(x)\\
     \leq & \left(\int_{\mathbb{R}^p} \tilde{\Gamma}^{-\tau}(h)(x)(f(x)-f^\star(x))  d\lambda^p(x)\right)^{\frac{1}{q}} \left( \int_{\mathbb{R}^p} \tilde{\Gamma}^{\frac{q^\star}{q}\tau}(h)(x)(f(x)-f^\star(x))  d\lambda^p(x)\right)^{\frac{1}{q^{\star}}},
\end{align*}
which directly gives an interpolation inequality between Besov IPMs:
\begin{equation}\label{eq:ipmdebase}
d_{\mathcal{B}^{s}_{\infty,\infty}}(f,f^\star)\leq      d_{\mathcal{B}^{s+\tau}_{\infty,\infty}}(f,f^\star)^{\frac{1}{q}}\ d_{\mathcal{B}^{s-\frac{q^\star}{q}\tau}_{\infty,\infty}}(f,f^\star)^{\frac{1}{q^{\star}}}.
\end{equation}
Finding the smallest $q\in (0,1)$ such that the quantity $d_{\mathcal{B}^{s-\frac{q^\star}{q}\tau}_{\infty,\infty}}(f,f^\star)$ is finite, we obtain the following result.
\begin{proposition}\label{propo:ineqfulldimbes}
   Let $f,f^\star \in \mathcal{B}^{\beta,2}_{\infty,\infty}(\mathbb{R}^p,\mathbb{R},K)$ compactly supported in $B^p(0,K)$. Then for all $\alpha\geq \gamma>0$, we have
    $$\sup \limits_{h \in \mathcal{B}^{\gamma}_{\infty,\infty}(1)}\int_{\mathbb{R}^p}h(x)(f(x)-f^\star(x))d\lambda^p(x)\leq C \sup \limits_{h \in \mathcal{B}^{\alpha}_{\infty,\infty}(1)}\left(\int_{\mathbb{R}^p}h(x)(f(x)-f^\star(x))d\lambda^p(x)\right)^\frac{\beta+\gamma}{\beta+\alpha}.$$
\end{proposition}

The proof of Proposition \ref{propo:ineqfulldimbes} can be found in Section \ref{sec:addiproofs}.
Note that this result can be seen as a particular case of classical interpolation inequalities between Besov spaces \citep{hajaiej2010sufficient}. Indeed, noticing that for $g\in L^2$ we have
$$\sup \limits_{h \in \mathcal{B}^{\gamma}_{\infty,\infty}(1)}\int_{\mathbb{R}^p}h(x)g(x)d\lambda^p(x)=\|g\|_{\mathcal{B}_{1,1}^{-\gamma}},$$
Proposition \ref{propo:ineqfulldimbes} is the application to the function $g=f-f^\star$ of the inequality 
\begin{equation}\label{eq:besovinterpineq}
\|g\|_{\mathcal{B}_{1,1}^{-\gamma}}\leq \|g\|_{\mathcal{B}_{1,1}^{\beta,2}}^\frac{\alpha-\gamma}{\beta+\alpha}\ \|g\|_{\mathcal{B}_{1,1}^{-\alpha}}^\frac{\beta+\gamma}{\beta+\alpha},\end{equation}
which can be shown just like Proposition \ref{propo:ineqfulldimbes}.

Finally, by paying a logarithmic factor, we obtain that Proposition \ref{propo:ineqfulldimbes} translates to Hölder regularity.
\begin{theorem}\label{eq:ineqinfulldim} For $f,f^\star\in \mathcal{H}^\beta_K(\mathbb{R}^p,\mathbb{R})$ with compact support in $B^p(0,K)$, we have
$$d_{\mathcal{H}^{\gamma}_{1}}(f,f^\star)\leq C\log(1+d_{\mathcal{H}^{\alpha}_{1}}(f,f^\star)^{-1})^2 d_{\mathcal{H}^{\alpha}_{1}}(f,f^\star)^\frac{\beta+\gamma}{\beta+\alpha}.$$
\end{theorem}
The proof of Theorem \ref{eq:ineqinfulldim} can be found in Section \ref{sec:prevuedecethe}. The goal of this paper is to generalize this inequality to the case where $f$ and $f^\star$ are densities with respect to the volume measures of two different closed submanifolds. The main difficulty of this generalization is to exploit the smoothness of the densities although, from the point of view of the ambient space $\mathbb{R}^p$, they do not have any. In particular, in the case where all the regularity exponents are non integers, Theorem \ref{theo:theineq} can be seen as a generalization of the interpolation inequality between Besov norms on functions \eqref{eq:besovinterpineq} to Besov norms on the signed measures $\mu-\mu^\star$. Indeed, although the tempered distribution $\mu-\mu^\star$ does not belong to any Besov space with positive regularity exponent, Theorem \ref{theo:theineq} states that if $\mu$ and $\mu^\star$ have low dimensional $\beta$-regularity then the inequality 
$$\|\mu-\mu^\star\|_{\mathcal{B}_{1,1}^{-\gamma}}\lesssim \|\mu-\mu^\star\|_{\mathcal{B}_{1,1}^{-\alpha}}^\frac{\beta+\gamma}{\beta+\alpha}$$
stands up to logarithmic factors.

\subsection{Proof outline of the main result (Theorem \ref{theo:theineq})}
In this Section, the principal steps of the proof of Theorem \ref{theo:theineq} are presented. We start by stating some properties resulting of the manifold assumption and use them to show that the integration with respect to a smooth density on $\mathcal{M}$, can be parameterized through the integration on $\mathbb{R}^d$ (Proposition \ref{prop:keydecomp}). Then, we show that if two measures supported on two manifolds respectively, are close enough for a distance $d_{\mathcal{H}^{\alpha}_{1}}$, there exists a smooth diffeomorphism between their support having some keys properties (Theorem \ref{theo:existencecompmap}). Finally, this diffeomorphism is used to split the IPMs into two distances: a distance between the support of the manifolds and a distance between the densities. This decomposition allows to conclude the proof of Theorem \ref{theo:theineq}.

\subsubsection{Integration on manifolds via smooth transport maps}\label{sec:integ}
For a manifold $\mathcal{M}$ satisfying the $(\beta+1,K)$-manifold condition, there exists an atlas\\ $\Big(A_i,\varphi_{i}\Big)_{i\in \{1,...,m\}}$ of $\mathcal{M}$ such that:
\begin{itemize}
    \item for all $i\in \{1,...,m\}$, $A_i=\varphi_{i}^{-1}(B^d(0,\tau))$ with $\tau=\frac{1}{4K}$,
    \item for all $i\in \{1,...,m\}$, $\varphi_{i}\in \mathcal{H}^{\beta+1}_K(A_i,B^d(0,\tau))$ and $\varphi_{i}^{-1}\in \mathcal{H}^{\beta+1}_K(B^d(0,\tau),A_i)$,
    \item for all $x\in \mathcal{M}$, there exists $i\in \{1,...,m\}$ such that $B^p(x,\tau/2)\cap \mathcal{M}\subset A_i$.
\end{itemize}
The existence of this atlas results from basic properties of the $(\beta+1,K)$-manifold condition that are derived in Appendix Section \ref{app:A1}. Using this parametrization we can show that the integration with respect to a smooth density on $\mathcal{M}$ can be parameterized through the integration on $\mathbb{R}^d$.
 
\begin{proposition}\label{prop:keydecomp}
    Let $\mathcal{M}$ satisfying the $(\beta+1,K)$-manifold condition and $\mu$ a probability measure satisfying the $(\beta,K)$-density condition on $\mathcal{M}$. Then there exists  a collection of constants $C_i>0$, maps $\phi_i\in\mathcal{H}^{\beta+1}_{C}(\mathbb{R}^d,\mathbb{R}^p)$ with $\inf_{u\in B^d(0,\tau)}\lambda_{\min}((\nabla\phi_i(u))^\top \nabla \phi_i(u))\geq C^{-1}$ and weight functions $\zeta_i\in\mathcal{H}^{\beta+1}_{C}(\mathbb{R}^d,\mathbb{R})$ with $supp(\zeta_i)\subset B^d(0,\tau)$, such that for all bounded continuous function $h:\mathbb{R}^p\rightarrow \mathbb{R}$ we have 
    \begin{equation*}
        \int_\mathcal{M} h(x)d\mu(x)= \sum \limits_{i=1}^m \frac{1}{C_i}\int_{\mathbb{R}^d}h(\phi_i(u))\zeta_i(u)d\lambda^d(u).
    \end{equation*}
\end{proposition}

The proof of Proposition \ref{prop:keydecomp} can be found in Section \ref{sec:prop:keydecomp}. This is a key result that will enable to exploit the $\beta$-regularity of the densities defined on manifolds, using the wavelet decomposition on $\mathbb{R}^d$ of the  $\beta+1$-regular transport maps $\phi_i$.

\subsubsection{Existence of a smooth diffeomorphism between the manifolds}\label{sec:exdiff}

We shall distinguish two cases: 
\begin{itemize}
    \item the case where the distance $d_{\mathcal{H}^\eta_1}(\mu,\mu^\star)$ is small enough so that there exists a smooth diffeomorphism between $\mathcal{M}$ and $\mathcal{M}^\star$,
    \item the case where the distance $d_{\mathcal{H}^\eta_1}(\mu,\mu^\star)$ is not small enough.
\end{itemize}
In the second case, one has $d_{\mathcal{H}^\eta_1}(\mu,\mu^\star)\geq C^{-1}$ so it is easy to get $d_{\mathcal{H}^\gamma_1}(\mu,\mu^\star)\leq C d_{\mathcal{H}^\eta_1}(\mu,\mu^\star)$. Therefore, let us focus on the first case and gather the properties that the diffeomorphism needs to respect.

\begin{definition}\label{defi:compatibility} For $\mathcal{M},\mathcal{M}^\star$ satisfying the $(\beta+1,K)$-manifold condition and $t\in (0,r^\star/2)$ for $r^\star$ the reach of $\mathcal{M}^\star$, such that $\mathbb{H}(\mathcal{M},\mathcal{M}^\star)<t$. We say that a map $T:\mathcal{M}^{\star t}\rightarrow \mathcal{M}^\star$ is $(\mathcal{M},\mathcal{M}^\star)_{K_T}$-compatible with a radius $t$ if:
\begin{itemize}
    \item[i)] $T \in \mathcal{H}^{\beta+1}_{K_T}(\mathcal{M},\mathcal{M}^\star)$ and the restriction $T_{|_{\mathcal{M}}}:\mathcal{M}\rightarrow \mathcal{M}^\star$ is a diffeomorphism such that for $T_{|_{\mathcal{M}}}^{-1}:\mathcal{M}^\star\rightarrow \mathcal{M}$ the inverse application such that $T_{|_{\mathcal{M}}}^{-1}\circ T_{|_{\mathcal{M}}} = \text{Id}_{|_{\mathcal{M}}}$, we have $T_{|_{\mathcal{M}}}^{-1}\in \mathcal{H}^{\beta+1}_{K_T}(\mathcal{M}^\star,\mathcal{M})$.
    \item[ii)] For all $x\in \mathcal{M}^\star$, 
 $T^{-1}(\{x\})-x\subset E_x$ a $p-d$ dimensional subspace of $\mathbb{R}^p$ and 
 $$(T^{-1}(\{x\})-x)\cap B^p(x,t)=B_{E_x}(0,t).$$
    \item[iii)] For any $f\in \mathcal{H}^{\eta}_1(\mathcal{M}^{\star t},\mathbb{R})$ with $\eta\in [1,\beta+1]$, the map $x\mapsto \int_{T^{-1}(\{x\})}f(y)d\lambda^{p-d}_{E_x}(y)$ belongs to $\mathcal{H}^{\eta}_{K_T}(\mathcal{M}^\star,\mathbb{R})$.
    \item[iv)]  The map $x\mapsto \text{ap}_d(\nabla T(x))$ (Definition \ref{def:approx}) is bounded below by $K_T^{-1}$.
\end{itemize}
\end{definition}

Fulfilling compatibility may seem to be restrictive, but it is shown in Theorem \ref{theo:existencecompmap}, that as long as the $d_{\mathcal{H}^{\alpha}}$ distance between the measures is relatively small for any $\alpha>0$, there exists a $(\mathcal{M},\mathcal{M}^\star)$-compatible map. Let us first state that the distance $d_{\mathcal{H}^{\alpha}_1}$ controls the Wasserstein and Hausdorff distances.

\begin{lemma}\label{lemma:hausdorff}
    Let $\mathcal{M},\mathcal{M}^\star$ satisfying the $(0,K)$-manifold condition and $\mu,\mu^\star$ two probability measures satisfying the $(0,K)$-regularity density condition on $\mathcal{M}$ and $\mathcal{M}^\star$ respectively. Then for all $\alpha\geq 1$, there exists a constant $C_\alpha>0$ such that $\forall t>0$, if
    $$
     d_{\mathcal{H}^{\alpha}_1}(\mu,\mu^\star)\leq t^{(d+1)(2\alpha-1)},
     $$
     then 
     $$
     W_1(\mu,\mu^\star) 
     \leq C_\alpha t^{d+1} \quad \text{ and } \quad \mathbb{H}(\mathcal{M},\mathcal{M}^\star)< C_\alpha t.
     $$ 
\end{lemma}

The proof of Lemma \ref{lemma:hausdorff} can be found in Section \ref{sec:lemma:hausdorff}. We are now in position to state the existence of a $(\mathcal{M},\mathcal{M}^\star)$-compatible map.

\begin{theorem}\label{theo:existencecompmap}
 Let $\mathcal{M},\mathcal{M}^\star$ satisfying the $(\beta+1,K)$-manifold condition and $\mu,\mu^\star$ probability measures on $\mathcal{M}$ and $\mathcal{M}^\star$ respectively satisfying the $(0,K)$-density condition. For all $\alpha>0$, there exists a constant $C_\alpha>0$ such that if 
        $$
    d_{\mathcal{H}^{\alpha}_1}(\mu,\mu^\star)\leq C_{\alpha}^{-1},
     $$
 then there exists a $(\mathcal{M},\mathcal{M}^\star)_{K_T}$-compatible map with radius $t=C_{\alpha}^{-1}$ and  $K_T=C_{\alpha}$. 
\end{theorem}

The proof of Theorem \ref{theo:existencecompmap} can be found in Section \ref{sec:theo:existencecompmap}. We see that imposing that the densities are close for any IPM, implies that the manifolds are close in the sense that there exists a smooth diffeomorphism between them. This diffeomorphism is a key tool that is going to allow to decompose the IPMs into two distinct distances.

\subsubsection{Decomposition of the IPMs into a manifold distance and a density distance}\label{sec:decomp}

To show the bound on the $\mathcal{H}^\gamma_1$ IPM, we  split it into two terms: one characterizing the distance between the supports, the other characterizing the distance between the densities.\\
For $h\in \mathcal{H}^{\gamma}_1(\mathbb{R}^p,\mathbb{R})$, $T$ being $(\mathcal{M},\mathcal{M}^\star)$-compatible and $f_T$ the density of $T_{\# \mu}$ with respect to the volume measure on  $\mathcal{M}^\star$, we write
\begin{align}\label{split}
&\int_{\mathcal{M}}h(x)f(x)d\lambda_{\mathcal{M}}(x) - \int_{\mathcal{M}^\star}h(x)f^\star(x)d\lambda_{\mathcal{M}^\star}(x)\nonumber\\
= & \int_{\mathcal{M}}(h(x)-h(T(x)))f(x)d\lambda_{\mathcal{M}}(x) + \int_{\mathcal{M}^\star}h(x)(f_T(x)-f^\star(x))d\lambda_{\mathcal{M}^\star}(x).
\end{align} 
The cost is split this way to exploit the regularity of $f$ and $f^\star$. Indeed, the second term of \eqref{split} discriminates measures on the same support $\mathcal{M}^\star$. We are going to show that this case is equivalent to the classical inequality (Theorem \ref{eq:ineqinfulldim}) in full dimension. For the first term of \eqref{split}, we will use Proposition \ref{prop:keydecomp} to write the integration against the density $f$ on $\mathcal{M}$, as the integration on $\mathbb{R}^d$ via $\beta+1$-regular transport maps. This will enable to exploit the regularity of the density through the wavelet coefficients of the transport maps. 

Let us take care of the first term of \eqref{split} in the case $h\in \mathcal{H}^{1}_1(\mathbb{R}^p,\mathbb{R})$.

\begin{lemma}\label{lemma:firstterm}
       Let $\mathcal{M},\mathcal{M}^\star$ satisfying the $(\beta+1,K)$-manifold  condition and $\mu,\mu^\star$ two probability measures satisfying the  $(\beta,K)$-density condition on $\mathcal{M}$ and $\mathcal{M}^\star$ respectively. For all $\alpha>0$, there exists a constant $C_\alpha>0$ such that if 
        $$
    d_{\mathcal{H}^{\alpha}_1}(\mu,\mu^\star)\leq C_{\alpha}^{-1},
     $$
then for $T$ the  $(\mathcal{M},\mathcal{M}^\star)$-compatible map given by Theorem \ref{theo:existencecompmap}, we have for all $\epsilon\in(0,1)$ and $\eta \in (1,\beta+1]$,
    \begin{align*}
        d_{\mathcal{H}^{1}_1}(\mu,T_{\# \mu})\leq C\log(\epsilon^{-1})^4 \sup \limits_{\substack{h \in \mathcal{H}^{\eta}_1\\ h_{|\mathcal{M}^\star}=0}}\left(\int_{\mathcal{M}}h(x)d\mu(x)\right)^{\frac{\beta+1}{\beta+\eta}} + \epsilon. 
    \end{align*}
\end{lemma}

The proof of Lemma \ref{lemma:firstterm} can be found in Section \ref{sec:lemma:firstterm}. This result highlights the fact that the first term of \eqref{split} focuses on the distance between the manifolds, as it bounds it with a supremum over functions being equal to $0$ on $\mathcal{M}^\star$. Although Lemma \ref{lemma:firstterm} bounds the distance $d_{\mathcal{H}^{1}_1}(\mu,T_{\# \mu})$ only in the case $\gamma=1$, it is easy to generalize it to the case $\gamma\geq 1$ (see the detailed proof of Theorem \ref{theo:theineq} in Section \ref{sec:theo:theineq} for explicit arguments) using the classical interpolation inequality (Corollary \ref{coro:ineq without reg}). Let us now bound the second term of \eqref{split}.

\begin{lemma}\label{lemma:secondterm}
Let $\mathcal{M},\mathcal{M}^\star$ satisfying the $(\beta+1,K)$-manifold  condition and $\mu,\mu^\star$ two probability measures satisfying the  $(\beta,K)$-density condition on $\mathcal{M}$ and $\mathcal{M}^\star$ respectively. For all $\alpha>0$, there exists a constant $C_\alpha>0$ such that if 
        $$
    d_{\mathcal{H}^{\alpha}_1}(\mu,\mu^\star)\leq C_{\alpha}^{-1},
     $$
then for $T$ the  $(\mathcal{M},\mathcal{M}^\star)$-compatible map given by Theorem \ref{theo:existencecompmap}, we have for all $\epsilon\in(0,1)$ and $0< \gamma\leq\eta\leq \beta+1$,
    \begin{align*}
        d_{\mathcal{H}^{\gamma}_1}(T_{\# \mu},\mu^{\star})\leq C\log(\epsilon^{-1})^2 d_{\mathcal{H}^{\eta}_1}(T_{\# \mu},\mu^{\star})^{\frac{\beta+\gamma}{\beta+\eta}} + \epsilon. 
        \end{align*}
\end{lemma} 

The proof of Lemma \ref{lemma:secondterm} can be found in Section \ref{sec:lemma:secondterm}. An immediate corollary of this result is that if the two measures $\mu,\mu^\star$ live on the same manifold $\mathcal{M}^\star$ (so that $T_{\# \mu}=\mu$), then $$d_{\mathcal{H}^{\gamma}_1}(\mu,\mu^{\star})\leq C\log(\epsilon^{-1})^2 d_{\mathcal{H}^{\eta}_1}(\mu,\mu^{\star})^{\frac{\beta+\gamma}{\beta+\eta}} + \epsilon.$$ 
This is due to the reach condition on $\mathcal{M}^\star$, which  allows to thicken the manifold and apply the classical interpolation inequality as if we were in full dimension (see the proof of Lemma \ref{lemma:secondterm} for details). Note that contrary to Lemma \ref{lemma:firstterm}, Lemma \ref{lemma:secondterm} allows to have $\gamma\in (0,1)$. This supports the fact that the discrimination of densities with supports on the same manifolds is equivalent to the discrimination in full dimension. We will see in Proposition \ref{prop:ineqgammaleqone} that this is not true for densities having supports on different manifolds as the optimal potential behaves differently in this case.

Putting Lemmas \ref{lemma:firstterm} and \ref{lemma:secondterm} together, we finally obtain the interpolation inequality for densities having supports on different submanifolds.

\begin{theorem}\label{theo:theineq}
Let $\mathcal{M},\mathcal{M}^\star$ satisfying the $(\beta+1,K)$-manifold  condition and $\mu,\mu^\star$ two probability measures satisfying the  $(\beta,K)$-density condition on $\mathcal{M}$ and $\mathcal{M}^\star$ respectively. Then for all $1\leq \gamma\leq\eta$, we have
    \begin{align*}
        d_{\mathcal{H}^\gamma_1}(\mu,\mu^\star)\leq C\log\left(1+d_{\mathcal{H}^\eta_1}(\mu,\mu^\star)^{-1}\right)^{C_2} d_{\mathcal{H}^\eta_1}(\mu,\mu^\star)^{\frac{\beta+\gamma}{\beta+\eta}}. 
        \end{align*}
\end{theorem} 
The detailed proof of Theorem \ref{theo:theineq} can be found in Section \ref{sec:theo:theineq}. This result is the extension of Theorem \ref{eq:ineqinfulldim} to the submanifold case for $\gamma\geq 1$. In contrast to the the full dimension setting, the case $\gamma\in (0,1)$ behaves differently as stated in the following proposition.

\begin{proposition}\label{prop:ineqgammaleqone}
Let $\mathcal{M},\mathcal{M}^\star$ satisfying the $(\beta+1,K)$-manifold  condition and $\mu,\mu^\star$ two probability measures satisfying the  $(\beta,K)$-density condition on $\mathcal{M}$ and $\mathcal{M}^\star$ respectively. Then for all $0<\gamma\leq\eta\leq 1$, we have
    \begin{align*}
        d_{\mathcal{H}^\gamma_1}(\mu,\mu^\star)\leq Cd_{\mathcal{H}^\eta_1}(\mu,\mu^\star)^{\frac{\gamma}{\eta}}. 
        \end{align*}
\end{proposition} 
The proof of Proposition \ref{prop:ineqgammaleqone} can be found in Section \ref{sec:prop:ineqgammaleqone}. This inequality is sharp as the equality is attained in numerous cases. For example, if $\mathcal{M}$ is the $p-1$ sphere of radius $1$, $\mathcal{M}^\star$ is the $p-1$ sphere of radius $1+\epsilon$ and $\mu,\mu^\star$ are rescaled volume measure on  $\mathcal{M},\mathcal{M}^\star$ respectively. In this case we have that the $d_{\mathcal{H}^\gamma_1}$ and $d_{\mathcal{H}^\eta_1}$ distances behave like $\epsilon^\gamma$ and $\epsilon^\eta$ respectively. The intuition behind this result is that for $\gamma \in (0,1)$, the optimal potential $h$ within $\mathcal{H}_1^\gamma$ (i.e. $d_{\mathcal{H}^\gamma_1}(\mu,\mu^\star)=\int h d\mu-\int hd\mu^\star$) behaves like $d(\cdot,\mathcal{M}^\star)^\gamma$.

Putting Theorem \ref{theo:theineq} and Proposition \ref{prop:ineqgammaleqone} together, we obtain a general result for all $\gamma>0$.

\begin{corollary}\label{coro:finalres}
Let $\mathcal{M},\mathcal{M}^\star$ satisfying the $(\beta+1,K)$-manifold  condition and $\mu,\mu^\star$ two probability measures satisfying the  $(\beta,K)$-density condition on $\mathcal{M}$ and $\mathcal{M}^\star$ respectively. Then for all $0< \gamma <\eta$, we have
    \begin{align*}
        d_{\mathcal{H}^\gamma_1}(\mu,\mu^\star)\leq C\log(1+d_{\mathcal{H}^\eta_1}(\mu,\mu^\star)^{-1})^{C_2} d_{\mathcal{H}^\eta_1}(\mu,\mu^\star)^{\delta},
        \end{align*}
        with
    \begin{equation} 
    \label{Gstar2}
    \delta = \left\{
    \begin{array}{ll}
     \frac{\beta+\gamma}{\beta+\eta} &   \quad \text{if } \gamma\geq 1 \smallskip\\
        \frac{\beta\gamma+\gamma}{\beta+\eta} & \quad \text{if }  \gamma\leq 1 \leq \eta \smallskip\\
     \frac{\gamma}{\eta} & \quad \text{if } \eta \leq 1. \smallskip\\
    \end{array}
    \right.
    \end{equation} 
\end{corollary}
When $0<\gamma< 1< \eta$, the exponent $\delta$ is obtained by applying first Proposition \ref{prop:ineqgammaleqone} on $d_{\mathcal{H}^\gamma_1}$ and $d_{\mathcal{H}^1_1}$, then applying Theorem \ref{theo:theineq} on $d_{\mathcal{H}^1_1}$ and $d_{\mathcal{H}^\eta_1}$. The exponent $\delta$ is still optimal as outlined in \eqref{eq:proofviaex} using the example of Section \ref{sec:example}.

In the assumptions of Corollary \ref{coro:finalres}, we suppose that both $\mu$ and $\mu^\star$ verify the point ii) of the density condition which is $f_\mu,\
f_{\mu^\star}\geq K^{-1}$. However, it is only necessary that one of the measures verifies this condition as stated in the next result.

\begin{proposition}\label{prop:lambdamin}
    Let $\mathcal{M},\mathcal{M}^\star$ satisfying the $(2,K)$-manifold condition and $\mu,\mu^\star$ two probability measures with densities with respect to the volume measure on $\mathcal{M},\mathcal{M}^\star$ respectively. Suppose that $\mu^\star$ verifies the $(0,K)$-density condition and that $f_\mu\in \mathcal{H}_K^1(\mathcal{M},\mathbb{R})$. Then for all $\eta>0$, there exist constants $C,C_\eta>0$ such that  if $ \inf \limits_{x\in \mathcal{M}} f_\mu(x)\leq C_\eta^{-1}$, then 
    $$
d_{\mathcal{H}^\eta_1}(\mu,\mu^\star)\geq C^{-1}.
$$
\end{proposition}

The proof of Proposition \ref{prop:lambdamin} can be found in Section \ref{sec:prop:lambdamin}. Using this result we have that if $\mu$ does not verify point ii) of the density condition for $K=C_\eta$, we have 
$$d_{\mathcal{H}^\gamma_1}(\mu,\mu^\star)\leq C \leq Cd_{\mathcal{H}^\eta_1}(\mu,\mu^\star).$$
Therefore, we could only suppose that one of the measures verifies point ii) of the density condition in Theorem \ref{theo:theineq}. We did not write it this way for simplicity of read.

\section{A direct application: optimal estimator of density on unknown submanifold}\label{sec:application}
Suppose that we observe an i.i.d.  sample $X_1,...X_n\in \mathbb{R}^p$ from  a probability measure $\mu^\star$ satisfying the $(\beta,K)$ density condition on an unknown $\mathcal{M}^\star$ satisfying the $(\beta+1,K)$ manifold condition. A thriving field of Machine Learning/Statistics focuses on building an estimator $\hat{\mu}(X_1,...,X_n)$ of $\mu^\star$ that achieves optimal rates of convergence (see for example \cite{tang2023minimax},  \cite{divol2022measure}, \cite{schreuder2021statistical}, \cite{de2022convergence}). In particular, fixing a $\gamma\in (0,\infty)$, \cite{tang2023minimax} build a theoretical estimator that attains the optimal rate of estimation (up to a logarithmic factor) for the distance $d_{\mathcal{H}_1^\gamma}$. In this section we show that Theorem \ref{theo:theineq} allows to build a much simpler estimator that attains optimal rates for all the distances $d_{\mathcal{H}_1^\gamma}$, $\gamma \in [1,\infty)$ simultaneously.

\subsection{The estimator}
Let us write $\mathcal{F}$ for the set of probability measures $\mu$ such that there exists a submanifold $\mathcal{M}$ satisfying the $(\beta+1,K)$-manifold condition and that $\mu$ satisfies the $(\beta,K)$-density condition on $\mathcal{M}$. Define the estimator
\begin{equation}\label{eq:estimator}
\hat{\mu} \in \argmin_{\mu \in \mathcal{F}} d_{\mathcal{H}^{d/2}_1}(\mu,\mu^\star_n),
\end{equation}
with $\mu^\star_n:= \frac{1}{n}\sum_{i=1}^n \delta_{X_i}$ the empirical measure from the data. Let us first bound the expected error $\mathbb{E}_{X_1,....,X_n\sim \mu^\star}[d_{\mathcal{H}^{d/2}_1}(\hat{\mu},\mu^\star)]$ of our estimator. To this end, define for all $\mu \in \mathcal{F}$
$$h_\mu \in \argmax_{h\in \mathcal{H}^{d/2}_1} \int h(x)d\mu(x)-\int h(x)d\mu^\star(x),$$
an optimal potential between $\mu$ and $\mu^\star$ and define
$$h_\mu^n \in \argmax_{h\in \mathcal{H}^{d/2}_1} \int h(x)d\mu(x)-\frac{1}{n}\sum_{i=1}^n h(X_i),$$
an optimal potential between $\mu$ and $\mu^\star_n$. We have
\begin{align*}
    \mathbb{E}[d_{\mathcal{H}^{d/2}_1}(\hat{\mu},\mu^\star)] =& \mathbb{E}[\int h_{\hat{\mu}}(x)d\hat{\mu}(x)-\int h_{\hat{\mu}}(x)d\mu^\star(x)]\\
     = & \mathbb{E}[\int h_{\hat{\mu}}(x)d\hat{\mu}(x)-\frac{1}{n}\sum_{i=1}^n h_{\hat{\mu}}(X_i)+\frac{1}{n}\sum_{i=1}^n h_{\hat{\mu}}(X_i)-\int h_{\hat{\mu}}(x)d\mu^\star(x)]\\
     \leq & \mathbb{E}[\int h_{\hat{\mu}}^n(x)d\hat{\mu}(x)-\frac{1}{n}\sum_{i=1}^n h_{\hat{\mu}}^n(X_i)]+\mathbb{E}[d_{\mathcal{H}^{d/2}_1}(\mu^\star_n,\mu^\star)].
\end{align*}
Furthermore, by definition of $\hat{\mu}$ we have 
\begin{align*}
    \mathbb{E}[\int h_{\hat{\mu}}^n(x)d\hat{\mu}(x)-\frac{1}{n}\sum_{i=1}^n h_{\hat{\mu}}^n(X_i)]& \leq \mathbb{E}[\int h_{\mu ^\star}^n(x)d\mu^\star(x)-\frac{1}{n}\sum_{i=1}^n h_{\mu^\star}^n(X_i)]\leq \mathbb{E}[d_{\mathcal{H}^{d/2}_1}(\mu^\star_n,\mu^\star)].
\end{align*}
Then, using Lemma 13 from \cite{tang2023minimax}
 we have that 
 $$\mathbb{E}[d_{\mathcal{H}^{d/2}_1}(\mu^\star_n,\mu^\star)]\leq C\log(n)n^{-1/2},$$
 which gives us using the previous derivations
 $$\mathbb{E}[d_{\mathcal{H}^{d/2}_1}(\hat{\mu},\mu^\star)] \leq C\log(n)n^{-1/2}.$$
 Now, using Theorem \ref{theo:theineq} we have that for all $\gamma \in [1,d/2]$,
 $$\mathbb{E}[d_{\mathcal{H}^{\gamma}_1}(\hat{\mu},\mu^\star)] \leq C \mathbb{E}[\log(d_{\mathcal{H}^{d/2}_1}(\hat{\mu},\mu^\star)^{-1})^{C_2}d_{\mathcal{H}^{d/2}_1}(\hat{\mu},\mu^\star)^{\frac{\beta+\gamma}{\beta+d/2}}],$$
 so using Jensen's inequality, we finally get for all $\gamma \in [1,\infty),$
  $$\mathbb{E}[d_{\mathcal{H}^{\gamma}_1}(\hat{\mu},\mu^\star)] \leq C\log(n)^{C_2}(n^{-\frac{\beta+\gamma}{2\beta+d}}\vee n^{-1/2})$$
  and this rate has been proven to be optimal (up to the logarithmic factor) in \cite{tang2023minimax}. 

In all, we have shown the following minimax optimality for $\hat{\mu}$, where we write $g(n)=\tilde{O}( f(n))$ if  the exist $C,C_2>0$ such that $g(n)\leq C \log(n)^{C_2} f(n)$.

\begin{theorem}\label{theo:estimminimax}
Let $n\in \mathbb{N}_{>0}$, $\mu^\star \in \mathcal{F}$ and $(X_1,...,X_n)$ an i.i.d. sample of law $\mu^\star$. Then the estimator $\hat{\mu}$ from \eqref{eq:estimator}
satisfies for all $\gamma\geq 1$
  $$\sup_{\mu^\star\in \mathcal{F}}\ \mathbb{E}_{X_i}[d_{\mathcal{H}^{\gamma}_1}(\hat{\mu},\mu^\star)] =\tilde{O}\left( \inf_{\hat{\theta}\in \Theta}\ \sup_{\mu^\star\in \mathcal{F}}\ \mathbb{E}_{X_i}[d_{\mathcal{H}^{\gamma}_1}(\hat{\theta},\mu^\star)]\right),$$
where $\Theta$ denotes the set of all possible estimators  of $\mu^\star$ based on n sample.
    
\end{theorem}
  
\subsection{Discussion}
If we were to replace the IPM $d_{\mathcal{H}^{d/2}_1}$ by the IPM $d_{\mathcal{H}^{\gamma}_1}$ with $1\leq \gamma<d/2$ in the estimator \eqref{eq:estimator}, we would only attain the rate $O(n^{-\frac{\gamma}{d}}\vee n^{-\frac{1}{2}})$ which is not optimal for $\gamma \in [1,d/2)$. The intuition is that computing $\hat{\mu}$ with a high regularity IPM, implicitly regularizes the measure $\mu_n^\star$ which allows to not have to compute a complex regularization of $\mu_n^\star$ like in \cite{liang2021generative} or \cite{tang2023minimax}. This highlights the need of a sharp inequality like Theorem \ref{theo:theineq} allowing to compare the IPMs.

In contrast to \cite{tang2023minimax}, the estimator $\hat{\mu}$ \eqref{eq:estimator} does not attain optimal rates for the distance $d_{\mathcal{H}_1^\gamma}$ with $\gamma \in (0,1)$. The intuition is that optimal potentials within $\mathcal{H}_1^{d/2}$ and $\mathcal{H}_1^\gamma$ tend to be very different when $\gamma$ is small so the estimator $\hat{\mu}$ is not adapted to IPMs of low regularity. This raises the question of whether there could exist an estimator being optimal for every $\gamma \in (0,\infty)$ simultaneously.

Like in \cite{tang2023minimax}, the estimator \eqref{eq:estimator} is theoretical meaning it is not computable in practice. Nevertheless, the minimum over the class $\mathcal{F}$ and the supremum over the class $\mathcal{H}^{d/2}_1$ could be approximated using stochastic gradient descent on neural network classes like in \cite{stephanovitch2023wasserstein}. The computational aspect of this estimator will be investigated in a future work.

\section{An example to get some intuition on Lemma \ref{lemma:firstterm}}\label{sec:example}
\subsection{Preamble}
In this section, a detailed analysis of a simple example is provided in order to better understand Lemma \ref{lemma:firstterm}. The main result of this paper (Theorem \ref{theo:theineq}) is a direct implication of Lemmas \ref{lemma:firstterm} and \ref{lemma:secondterm}. We only illustrate Lemma \ref{lemma:firstterm} as Lemma \ref{lemma:secondterm} focuses on probability measures having densities with respect to the same manifold and we showed that this case is equivalent to the full dimensional case. In contrast, Lemma \ref{lemma:firstterm} focuses on probability measures having densities with respect to two different manifolds.

To get some intuition on Lemma \ref{lemma:firstterm}, we shall construct some submanifolds $\mathcal{M}^\star,\mathcal{M}_n$, measures $\mu_n$ on $\mathcal{M}_n$ and construct optimal potentials $h^\star_n\in \mathcal{H}^{\eta}_1$ with $h_{n|\mathcal{M}^\star}^\star=0$ such that  
\begin{equation}\label{eq:uneautre}
d_{\mathcal{H}_1^1}(\mu_n,T_{\# \mu_n})\leq C\left(\int_{\mathcal{M}_n}h_n^\star(x)d\mu_n(x)\right)^{\frac{\beta+1}{\beta+\eta}},
\end{equation}

for $T$ being $(\mathcal{M}_n,\mathcal{M}^\star)$-compatible. The goals of this section are the following.
\begin{itemize}
    \item Describe the shape of the potential $h^\star_n$.
    \item Provide a concrete example where $d_{\mathcal{H}_1^1}(\mu_n,T_{\# \mu_n})\geq C^{-1} d_{\mathcal{H}_1^\eta}(\mu_n,T_{\# \mu_n})^{\frac{\beta+1}{\beta+\eta}}.$
\end{itemize}
For simplicity we  focus on $\eta \in [1,\beta+1]$ being an integer. In order to well illustrate the result, the example needs to verify the following two points.
\begin{itemize}
    \item[i)] It has to verify $\max \limits_{h \in \mathcal{H}^{\eta}_1, h_{|\mathcal{M}^\star}=0}\int_{\mathcal{M}_n}h(x)d\mu_n(x)\rightarrow 0$ when $n\rightarrow \infty$ as otherwise we would have \begin{align*}d_{\mathcal{H}_1^1}(\mu_n,T_{\# \mu_n}) & \leq K = K \left(\max \limits_{\substack{h \in \mathcal{H}^{\eta}_1\\ h_{|\mathcal{M}^\star}=0}}\int_{\mathcal{M}_n}h(x)d\mu_n(x)\right)^{-1}\max \limits_{\substack{h \in \mathcal{H}^{\eta}_1\\ h_{|\mathcal{M}^\star}=0}}\int_{\mathcal{M}_n}h(x)d\mu_n(x)\\
    & \leq C \max \limits_{\substack{h \in \mathcal{H}^{\eta}_1\\ h_{|\mathcal{M}^\star}=0}}\int_{\mathcal{M}_n}h(x)d\mu_n(x),
    \end{align*}
    so the result would be immediate.
    \item[ii)] The manifold $\mathcal{M}_n$ has to oscillate around $\mathcal{M}^\star$ in the sense that the function $x\rightarrow \|x-T(x)\|$ needs to have strong irregularities and therefore not belonging to $\mathcal{H}^{\eta}_C$ for $\eta>1$. For example take $\mathcal{M}^\star$ the sphere of radius 1, $\mathcal{M}_n$ the sphere of radius $1+1/n$ and $\mu_n$ the uniform measure on $\mathcal{M}_n$. Then for $T:\mathcal{M}^{\star 1/4}\rightarrow \mathcal{M}^\star$ the projection onto $\mathcal{M}^\star$, the map $x\mapsto \|x-T(x)\|$ belongs to $ \mathcal{H}^{\eta}_C$ so the result is trivial. This would also be true for any $T$ being $(\mathcal{M}_n,\mathcal{M}^\star)$-compatible as $(\mathcal{M}_n,\mathcal{M}^\star)$-compatible maps act like projections (see the proof of Theorem \ref{theo:existencecompmap}).
\end{itemize}

\subsection{The example}
\subsubsection{Definition of the manifolds and measures}
Let $n,\beta\in \mathbb{N}$ and $g_n,g^\star:[0,1]\rightarrow \mathbb{R}^2$ defined by 
\begin{equation}\label{eq:ex}
g^\star(t):=(\cos(2\pi t),\sin(2\pi t)) \text{ and } g_n(t):=(1+(2\pi n)^{-(\beta+1)}\sin(2\pi n t))g^\star(t).
    \end{equation}
Define the two submanifolds $\mathcal{M}_n:=g_n([0,1]),\mathcal{M}^\star:=g^\star([0,1])$ and $\mu_n$ the uniform measure on $\mathcal{M}_n$. The manifolds $\mathcal{M}_n$ and $\mathcal{M}^\star$ are represented in Figure \ref{fig:ex}. We choose the $(\mathcal{M}_n,\mathcal{M}^\star)$-compatible map  $T$ to be the projection onto $\mathcal{M}^\star$ for simplicity of the derivations.

\definecolor{ffqqqq}{rgb}{1,0,0}
\begin{figure}[!h]
\centering
\scalebox{2.5}{\begin{tikzpicture}[line cap=round,line join=round,x=1cm,y=1cm]
\clip(-5,-1.1) rectangle (1.5,1.2);
\draw[line width=0.3pt, smooth,samples=1000,domain=0:6.283185307179586] plot[parametric] function{cos(t)-3.5,sin(t)};
\draw[line width=0.3pt,color=ffqqqq, smooth,samples=1000,domain=0:6.283185307179586] plot[parametric] function{(1+0.05*cos(20*t))*cos(t)-3.5,(1+0.05*cos(20*t))*sin(t)};
\draw[line width=0.3pt, smooth,samples=1000,domain=0:6.283185307179586] plot[parametric] function{cos(t)-0.5,sin(t)};
\draw[line width=0.3pt,color=ffqqqq, smooth,samples=1000,domain=0:6.283185307179586] plot[parametric] function{(1+0.025*cos(50*t))*cos(t)-0.5,(1+0.025*cos(50*t))*sin(t)};
\end{tikzpicture}}
\caption{$\mathcal{M}_n$ and $\mathcal{M}^\star$ defined in \eqref{eq:ex} for $\beta=0$, $n=20$ (left) and $n=50$ (right).}
\label{fig:ex}
\end{figure}

We have that $\mathcal{M}_n$ verifies the $(\beta+1,C)$-manifold condition and for all $\alpha>\beta+1$, we have $\|g_n\|_{\mathcal{H}^\alpha}\geq C^{-1}n^{\alpha-(\beta+1)}$. Therefore this example is well suited to represent the case of $\beta+1$ regular manifolds. We could have taken a sequence of manifolds $\mathcal{M}^\star_n$ parameterized by functions $g^\star_n$ that also oscillate around $\mathcal{M}_n$, but it would not have given more insights. Indeed, for $T_n$ a $(\mathcal{M}_n,\mathcal{M}^\star_n)$-compatible map, we have 
$$d_{\mathcal{H}_1^1}(\mu_n,T_{n\# \mu_n})\leq d_{\mathcal{H}_1^1}(\mu_n,T_{\# \mu_n}) + d_{\mathcal{H}_1^1}(T_{\# \mu_n},T\circ T_{n\# \mu_n})+d_{\mathcal{H}_1^1}(T\circ T_{n\# \mu_n},T_{n\# \mu_n}).$$
The term $d_{\mathcal{H}_1^1}(T_{\# \mu_n},T\circ 
 T_{n\# \mu_n})$ can be treated by Lemma \ref{lemma:secondterm} as it is a distance between measures having support on the same manifold. The other two terms are equivalent to our example at they are distances between a measure on an oscillating manifold ($\mathcal{M}_n$ and $T_n(\mathcal{M}_n)$ respectively) and its projection by $T$ onto $\mathcal{M}^\star$.  Therefore this case is equivalent to our example.

\subsubsection{Expected shape of optimal potentials}

Let us take a closer look at the quantities involved in $d_{\mathcal{H}_1^1}(\mu_n,T_{\# \mu_n})$. We have 
\begin{align*}
    d_{\mathcal{H}_1^1}(\mu_n,T_{\# \mu_n}) & = \max \limits_{h \in \mathcal{H}^{1}_1}\int_{\mathcal{M}_n}(h(x)-h(T(x))f_{\mu_n}(x)d\lambda_{\mathcal{M}_n}(x)\\
    & \leq W_1(\mu_n,T_{\# \mu_n})\\
    & = \int_{\mathcal{M}_n}\|x-T(x)\|f_{\mu_n}(x)d\lambda_{\mathcal{M}_n}(x)\\
    & = \int_{\mathcal{M}^\star}\|y-T_{|\mathcal{M}_n}^{-1}(y)\|f_{\mu_n}(T_{|\mathcal{M}_n}^{-1}(y))\text{ap}_d(\nabla T(T_{|\mathcal{M}_n}^{-1}(y)))^{-1}d\lambda_{\mathcal{M}^\star}(y)\\
    & \leq C \int_{\mathcal{M}^\star}\|y-T_{|\mathcal{M}_n}^{-1}(y)\|d\lambda_{\mathcal{M}^\star}(y)\\
    & \leq C \int_0^1n^{-(\beta+1)}|\sin(2\pi n t)|dt\\
    & \leq C n^{-(\beta+1)}.
\end{align*}
Note that using the same derivations, we can show that $ d_{\mathcal{H}_1^1}(\mu_n,T_{\# \mu_n})\geq C^{-1} n^{-(\beta+1)}$ so $d_{\mathcal{H}_1^1}(\mu_n,T_{\# \mu_n})$ behaves like $O(n^{-(\beta+1)})$.

To build an optimal potential in $\mathcal{H}^\eta_1$, let us look at how to regularize the optimal Lipschitz potential $$L_n^\star:=\|x-T(x)\|\in \text{Lip}_1$$
such that it belongs to $\mathcal{H}^\eta$ for $\eta\geq 1$. A simple way could be to use Jensen's inequality:
\begin{align}\label{align:Jensen}
\int_{\mathcal{M}_n}\|x-T(x)\|f_{\mu_n}(x)d\lambda_{\mathcal{M}_n}(x)\leq \left(\int_{\mathcal{M}_n}\|x-T(x)\|^2f_{\mu_n}(x)d\lambda_{\mathcal{M}_n}(x)\right)^{\frac{1}{2}}.
\end{align}
We have that $\overline{L_n^\star}(x):=\|x-T(x)\|^2$ belongs to $\mathcal{H}^{\beta+1}$ but is not an optimal potential as its cost is too low. Indeed, we obtain an exponent $1/2$ in \eqref{align:Jensen}, which is strictly smaller than $\frac{\beta+1}{2\beta+1}$, the one of Lemma \ref{lemma:firstterm}.
The reason why $\overline{L_n^\star}$ is not optimal is that it does not use the fact that $\mathcal{M}_n$ is of regularity $\beta+1$. Let us study what should be the shape of an optimal potential. For $h\in \mathcal{H}^{\eta}$, we have
\begin{align}\label{align:gradientiskey}
 \int_{\mathcal{M}_n}& (h(x)-h(T(x)))f_{\mu_n}(x)d\lambda_{\mathcal{M}_n}(x)\nonumber\\
 & =\int_{\mathcal{M}_n}\langle \int_0^1 \nabla h(T(x)+t(x-T(x)))dt,x-T(x)\rangle f_{\mu_n}(x)d\lambda_{\mathcal{M}_n}(x)\nonumber\\
 & = \int_{\mathcal{M}_n}\Big(\langle \nabla h(T(x)) ,x-T(x)\rangle f_{\mu_n}(x) +O(\|x-T(x)\|^2)\Big)d\lambda_{\mathcal{M}_n}(x)\nonumber\\
 & \leq C \int_{\mathcal{M}_n}\langle \nabla h(T(x)) ,x-T(x)\rangle d\lambda_{\mathcal{M}_n}(x)+Cn^{-2(\beta+1)}\nonumber\\
& \leq C \int_0^1\langle \nabla h(g^\star(t)) ,g^\star(t)n^{-(\beta+1)}\sin(2\pi nt)\rangle dt+Cn^{-2(\beta+1)}, \end{align}
with $g^\star$ defined in \eqref{eq:ex}. We see that in order to determine a good potential $h\in \mathcal{H}^{\eta}$, it is enough to determine the function $H:[0,1]\rightarrow \mathbb{R}^2$ defined by 
\begin{equation}
    H(t):= \nabla h(g^\star(t)).
\end{equation}

When $n$ is very large, $\mathcal{M}_n$ and $\mathcal{M}^\star$ behave locally like the images of
\begin{equation}\label{eq:localcase}
\overline{g}^\star(t)=(t,0) \text{ and } \overline{g}_n(t)=(t,n^{-(\beta+1)}\sin(2\pi nt)),
\end{equation}
as represented in Figure \ref{fig:exzoom}. Let us first study the local case where we are looking for an optimal potential between $\overline{\mathcal{M}}_n$ and $\overline{\mathcal{M}}^\star$ defined as the images of $\overline{g}^\star$ and $\overline{g}_n$ in \eqref{eq:localcase}. Recall from \eqref{align:gradientiskey} that we want to maximize with respect to $H\in \mathcal{H}^{\eta-1}_1([0,1],\mathbb{R}^2)$ the quantity
$$\int_0^1\langle H(t) ,g^\star(t)n^{-(\beta+1)}\sin(2\pi nt)\rangle dt,$$
which in the local case \eqref{eq:localcase} corresponds to
\begin{equation}\label{eq:costex}
n^{-(\beta+1)}\int_0^1 H_1(t) \sin(2\pi nt) dt,
\end{equation}
for $H_1(t):=\langle H(t),g^\star(t)\rangle$.

To maximise \eqref{eq:costex}, we are going to take $H_1$ having the same sign as the function $x\mapsto \sin(2\pi nx)$. Therefore we will take $H_1(t)=0$ when $\sin(2\pi nt)=0$ (i.e. $t\in \{0,\frac{1}{2n},...,\frac{2n-1}{2n}\}$), and for $t \in [\frac{k}{2n},\frac{k+1}{2n}]$, $H_1(t)$ will be increasing in the direction of $\frac{\sin(\frac{k+1/2}{2n})}{|\sin(\frac{k+1/2}{2n})|}$ as $t$ gets closer to the middle point $\frac{k+1/2}{2n}$. This is illustrated in Figure \ref{fig:exzoom}.

\definecolor{ccqqqq}{rgb}{0.8,0,0}
\definecolor{qqqqcc}{rgb}{0,0,0.8}
\begin{figure}[h]
\centering
\scalebox{2.5}{\begin{tikzpicture}[line cap=round,line join=round,>=triangle 45,x=1cm,y=1cm]
\clip(-.2,-0.7) rectangle (4.9,0.5);
\draw[line width=0.3pt, smooth,samples=1000,domain=-10:10] plot[parametric] function{t,0};
\draw (4.1,0.06323731138545823) node[anchor=north west] {\tiny{$\mathcal{M}^\star$}};
\draw (-0.2,0.3) node[anchor=north west] {\tiny{$t$}};
\draw [color=ccqqqq](4.1,0.5) node[anchor=north west] {\tiny{$\mathcal{M}$}};
\draw [color=qqqqcc](1.1,0.4) node[anchor=north west] {\tiny{$H_1(t)$}};
\draw[line width=0.3pt,color=ccqqqq, smooth,samples=1000,domain=-10:10] plot[parametric] function{t,0.5*cos((0.9*t+1.01))};
\draw[-stealth,line width=0.4pt,color=qqqqcc] (.5-0.1,0) -- (0.5-0.1,0.1);
\draw[-stealth,line width=0.4pt,color=qqqqcc] (1.0-0.1,0) -- (1.0-0.1,-0.1);
\draw[-stealth,line width=0.4pt,color=qqqqcc] (1.5-0.1,0) -- (1.5-0.1,-0.15);
\draw[-stealth,line width=0.4pt,color=qqqqcc] (2-0.1,0) -- (2-0.1,-0.4);
\draw[-stealth,line width=0.4pt,color=qqqqcc] (2.5-0.1,0) -- (2.5-0.1,-0.55);
\draw[-stealth,line width=0.4pt,color=qqqqcc] (3-0.1,0) -- (3-0.1,-0.44);
\draw[-stealth,line width=0.4pt,color=qqqqcc] (3.5-0.1,0) -- (3.5-0.1,-.15);
\draw[-stealth,line width=0.4pt,color=qqqqcc] (4-0.1,0) -- (4-0.1,-.1);
\draw[-stealth,line width=0.4pt,color=qqqqcc] (4.45-0.1,0) -- (4.45-0.1,.1);
\draw[-stealth,line width=0.4pt,color=qqqqcc] (4.9-0.1,0) -- (4.9-0.1,.2);
\end{tikzpicture}}
\caption{Zoom on a part of $\mathcal{M}_n,\mathcal{M}^\star$ defined in \eqref{eq:ex}.}
\label{fig:exzoom}
\end{figure}
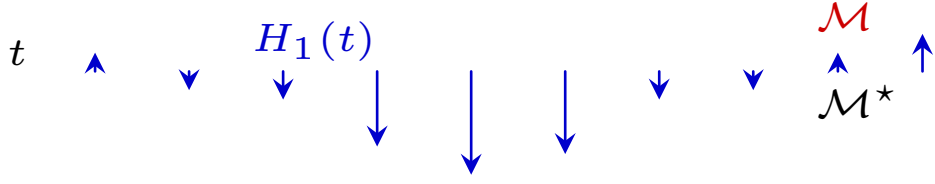

The idea behind this construction is to be able to utilize that $\mathcal{M}_n$ is of regularity $\beta+1$. It implies that we can quantify the distance between two points of intersections of the manifolds $\mathcal{M}_n$ and $\mathcal{M}^\star$ with respect to the distance between the manifolds. Indeed for two critical points $\frac{k}{2n},\frac{k+1}{2n}$, we have
\begin{align}
\frac{k+1}{2n}-\frac{k}{2n} & =\frac{1}{2n}=O(n^{-(\beta+1)})^{\frac{1}{\beta+1}}=\left(\sup \limits_{t\in [\frac{k}{2n},\frac{k+1}{2n}]} n^{-\beta+1}|\sin(2\pi nt)|\right)^{\frac{1}{\beta+1}}\nonumber\\
& = \left(\sup \limits_{x \in \Theta_n(k)} \|x-T(x)\|\right)^{\frac{1}{\beta+1}},
\end{align}
with 
$\Theta_n(k)=\{x \in \mathcal{M}_n|\ \exists \lambda\in [0,1], T(x)=g^\star(\lambda \frac{k}{2n}+(1-\lambda)\frac{k+1}{2n})\}$. Therefore as $H_1$ is increasing between critical points $\frac{k}{2n},\frac{k+1}{2n}$, we will be able to relate the value $H_1(\frac{k+1/2}{2n})$ to the distance $ \left(\sup \limits_{x \in \Theta_n(k)} \|x-T(x)\|\right)^{\frac{1}{\beta+1}}$.

\subsubsection{Explicit construction of an optimal potential}
Let us build $H_1\in \mathcal{H}^{\eta-1}_1$ that maximizes (up to a multiplicative constant) the objective \eqref{eq:costex}. Let $\lambda\in \mathcal{H}^{\eta}([0,1],[0,1])$ such that for all $k\in \{0,...,\eta-1\}$, $\nabla^k \lambda(0)=\nabla^k \lambda(1)=0$ and for all $x\in [1/4,3/4]$, $\lambda(x)\geq 1/2$. We define the optimal potential by 
$$H_1(t):=\frac{(-1)^{\lfloor 2nt \rfloor}}{n^{\eta-1}}\lambda\big(2n(t-\frac{\lfloor 2nt \rfloor}{2n})\big).$$

The cost \eqref{eq:costex} of $H_1$ can be lower bounded as follow.
\begin{align}\label{align:calculdansex}
    n^{-(\beta+1)}\int_0^1 H_1(t) \sin(2\pi nt) dt & = n^{-(\beta+1)}2n\int_0^{\frac{1}{2n}} H_1(t) \sin(2\pi nt) dt \nonumber \\
    & \geq 2n^{-\beta}\int_{\frac{1}{8n}}^{\frac{3}{8n}} H_1(t) \sin(2\pi nt) dt\nonumber\\
    & \geq 2n^{-\beta}\int_{\frac{1}{8n}}^{\frac{3}{8n}} \frac{1}{2n^{\eta-1}} \sin(2\pi nt) dt\nonumber\\
    & \geq C^{-1}   n^{-\beta-\eta+1}\int_{\frac{1}{8n}}^{\frac{3}{8n}}\sin(2\pi nt) dt\nonumber\\
    & \geq C^{-1}  n^{-(\beta+\eta)}.
\end{align}
Furthermore, we have 
\begin{align}\label{align:calculdansex}
 n^{-(\beta+\eta)} &=\left(n^{-(\beta+1)}\right)^{\frac{\beta+\eta}{\beta+1}}\nonumber\\
    & \geq C^{-1}   \left(\int_0^1n^{-(\beta+1)}|\sin(2\pi n t)|d\lambda_{\mathcal{M}_n}(x)\right)^{\frac{\beta+\eta}{\beta+1}}\nonumber\\
    & \geq C^{-1}  \left(\int_{\mathcal{M}_n}\|x-T(x)\|f_{\mu_n}(x)d\lambda_{\mathcal{M}_n}(x)\right)^{\frac{\beta+\eta}{\beta+1}}\nonumber\\
    &=  C^{-1}  d_{\mathcal{H}_1^1}(\mu_n,T_{\# \mu_n})^{\frac{\beta+\eta}{\beta+1}}\nonumber,
\end{align}
so in particular $H_1$ verifies that
$$d_{\mathcal{H}_1^1}(\mu_n,T_{\# \mu_n})\leq C \left(n^{-(\beta+1)}\int_0^1 H_1(t) \sin(2\pi nt) dt\right)^{\frac{\beta+1}{\beta+\eta}}.$$

Let us now show that this potential is optimal (up to a multiplicative constant) by giving an upper bound on \eqref{eq:costex}.  For $f:\mathbb{R}/\mathbb{Z}\rightarrow \mathbb{C}$, let $(c_f(k))_{k\in \mathbb{Z}}$ denote its Fourier coefficients and let 
$$
\mathcal{W}^{\eta-1,2}_1=\{f\in L^2(\mathbb{R}/\mathbb{Z},\mathbb{R}) |\ \sum_{k\in \mathbb{Z}} |c_f(k)|^2|k|^{2(\eta-1)}\leq 1\}
$$
be the unit ball of the $L^2$ Sobolev space of regularity $\eta-1$. As $\mathcal{H}^{\eta-1}\xhookrightarrow{}\mathcal{W}^{\eta-1,2}$, we have 
\begin{align*}
 \sup_{f\in \mathcal{H}^{\eta-1}_1}\int_0^1 f(t) \sin(2\pi nt) dt & \leq C \sup_{f\in \mathcal{W}^{\eta-1,2}_1}\int_0^1 f(t) \sin(2\pi nt) dt\\
& =C \sup_{f\in \mathcal{W}^{\eta-1,2}_1} \sum_{k\in \mathbb{Z}}c_f(k) c_{\sin(2\pi n\cdot)}(k)\\
& =C \sup_{f\in \mathcal{W}^{\eta-1,2}} -\frac{i}{2}c_f(n) +\frac{i}{2}c_f(-n)\\
& \leq C n^{-(\eta-1)}.
\end{align*}
Therefore,
\begin{align*}
   n^{-(\beta+1)}\sup_{f\in \mathcal{H}^{\eta-1}_1}\int_0^1 f(t) \sin(2\pi nt) dt & \leq C n^{-(\beta+\eta)} \leq C_2 n^{-(\beta+1)}\int_0^1 H_1(t) \sin(2\pi nt) dt, 
\end{align*}
so $H_1$ is indeed optimal.

Coming back to the non local case, for $\Theta:\mathcal{M}^\star\rightarrow \mathbb{R}^2$ such that $\Theta(x)$ is the unitary normal vector pointing outward the manifold $\mathcal{M}^\star$ at the point $x$, defining
$$h(x)=H_1\circ g^{\star-1}\circ T(x)\langle \Theta\circ T(x),x-T(x)\rangle,$$
we have $h\in \mathcal{H}^{\eta}_C(\mathcal{M}^{\star n^{-(\beta+1)}},\mathbb{R})$ so using Proposition \ref{prop:extensionH} it can be extended into a map belonging to $\mathcal{H}^{\eta}_C(\mathbb{R}^2,\mathbb{R})$. Furthermore, we have
\begin{align*}
\int_{\mathcal{M}_n}(h(x)-h(T(x)))f_{\mu_n}(x)d\lambda_{\mathcal{M}_n}(x)& \geq C \int_0^1\langle H_1(t)g^\star(t)) ,g^\star(t)n^{-(\beta+1)}\sin(2\pi nt)\rangle dt\\
& \geq C d_{\mathcal{H}_1^1}(\mu_n,T_{\# \mu_n})^{\frac{\beta+\eta}{\beta+1}},
\end{align*}
so $h$ is a potential satisfying \eqref{eq:uneautre} for the non local case. Finally, for all $h^{'}\in \mathcal{H}^{\eta}_1(\mathbb{R}^2,\mathbb{R})$ with $h^{'}_{|\mathcal{M}_{g^\star}=0}$, we have
\begin{align*}
\int_{\mathcal{M}_n}h^{'}(x)f_{\mu_n}(x)d\lambda_{\mathcal{M}_n}(x)& \leq C  \int_0^1\langle \nabla h^{'}(g^\star(t)) ,g^\star(t)n^{-(\beta+1)}\sin(2\pi nt)\rangle dt +Cn^{-2(\beta+1)}\\
& \leq  n^{-(\beta+1)}\sup_{f\in \mathcal{H}^{\eta-1}_1}\int_0^1 f(t) \sin(2\pi nt) dt +Cn^{-2(\beta+1)}\\
& \leq Cn^{-(\beta+\eta)}\\
& \leq C \int_{\mathcal{M}_n}h(x)f_{\mu_n}(x)d\lambda_{\mathcal{M}_n}(x).
\end{align*}
We conclude that $h$ is an optimal potential for the non local case which gives in particular the sought-after inequality
\begin{align*}
    d_{\mathcal{H}_1^\eta}(\mu_n,T_{\# \mu_n}) & \leq C \int_{\mathcal{M}_n}h(x)f_{\mu_n}(x)d\lambda_{\mathcal{M}_n}(x)\\
    & \leq Cn^{-(\beta+\eta)}\\
    & \leq C d_{\mathcal{H}_1^1}(\mu_n,T_{\# \mu_n})^{\frac{\beta+\eta}{\beta+1}}.
\end{align*}

Finally note that this example also allows to show that the exponent $\delta$ of Corollary \ref{coro:finalres} is optimal. Indeed, for $\gamma\in(0,1)$ and $\eta\geq 1$ we have
\begin{align}\label{eq:proofviaex}
d_{\mathcal{H}^\gamma_1}(\mu_n,T_{\# \mu_n})&=O(n^{-\gamma(\beta+1)})
=O(n^{-\beta+\eta})^{\frac{\beta\gamma+\gamma}{\beta+\eta}}=d_{\mathcal{H}^\eta_1}(\mu_n,T_{\# \mu_n})^{\frac{\beta\gamma+\gamma}{\beta+\eta}}.
\end{align}

\subsubsection{Further remarks on the example}
We see that in the case $\eta=\beta+1$, the exponent $\frac{\beta+1}{2\beta+1}$ is an improvement over the exponent $1/2$ given by Jensen's inequality \eqref{align:Jensen}. In contrary to $x\mapsto \|x-T(x)\|^2$, the potential $h$ uses the fact that $\mathcal{M}_n$ is of regularity $\beta+1$. Instead of being equal to $0$ on $\mathcal{M}^\star$, $\nabla h$ increases between the intersections points of $\mathcal{M}^\star$ and $\mathcal{M}_n$ which corresponds to the irregularities of the function $x\mapsto \|x-T(x)\|$. As the distance between these points is proportional to the distance between $\mathcal{M}^\star$ and $\mathcal{M}_n$ to the power $\frac{1}{\beta+1}$, we have that $\|\nabla h(T(x))\|$ is proportional to $\|x-T(x)\|^{\frac{\eta-1}{\beta+1}}=\|x-T(x)\|^{\frac{\beta}{\beta+1}}$ for $\eta=\beta+1$. This explains the improvement over the potential $x\mapsto \|x-T(x)\|^2$ as the norm of its gradient is proportional to $\|x-T(x)\|$.

The construction of $h$ is similar to the one used in the proof of Lemma \ref{lemma:firstterm}. In both cases we take $$h(x)=\langle  F(x),x-T(x)\rangle,$$
with $F$ a regularization of the map $x\mapsto \frac{x-T(x)}{\|x-T(x)\|}$ belonging to $\mathcal{H}^{\eta-1}_C$ and such that $\|\nabla^{ \lfloor \eta \rfloor}F(x)\|\leq \|x-T(x)\|^{-1+\eta-\lfloor \eta \rfloor}$. The difference lies in the regularization of the map  $x\mapsto \frac{x-T(x)}{\|x-T(x)\|}$ where here we did a geometrical construction which is already quite complex even in the simple setting of this example. We believe that Lemma \ref{lemma:firstterm} could be proven using a construction of this type but it might be very complicated to describe it in the general case. In the proof of Lemma \ref{lemma:firstterm}, we simply regularize the function $x\mapsto \frac{x-T(x)}{\|x-T(x)\|}$ by lowering the size of its wavelet coefficients. This trick allows us to obtain a good potential of regularity $\eta-1$ without really knowing its shape. In particular, this outlines the efficiency of the wavelet tool which allows to easily modify the regularity of functions while controlling their potential action.

\vskip 0.2in
\bibliography{sample}


\newpage

\appendix
\section{Proofs of the properties of $\beta+1$-smooth submanifolds (Section \ref{sec:integ})}\label{app:A}
In this Section, we give the detail of the proof of the integration on manifolds via smooth transport maps (Proposition \ref{prop:keydecomp}) and list some properties of manifolds satisfying the $(2,K)$-manifold condition. 
\subsection{Properties of $\beta+1$-smooth submanifolds}\label{app:A1}
Recalling the notation of Definition \ref{def:manifoldcond}, we obtain in the following result that the images of the charts contain the same ball around the origin.
\begin{proposition}\label{prop:firstinclusion}
If $\mathcal{M}$ verifies the $(2,K)$-manifold condition, then  for all $x\in \mathcal{M}$,
$$B^d(0,\frac{1}{4K})\in \varphi_x(\overline{B}^p(x,\sqrt{\frac{3}{4}}K^{-1})\cap M).$$
\end{proposition}

\begin{proof}
Let us first notice that for all $z\in \overline{B}^p(x,\sqrt{\frac{3}{4}}K^{-1})\cap \mathcal{M}$ we have 
\begin{align*}
    \|z-\pi_{\mathcal{T}_x(\mathcal{M})}(z)\|=&  \|\varphi_x^{-1}\big(\pi_{\mathcal{T}_x(\mathcal{M})}(z)-x\big)-\pi_{\mathcal{T}_x(\mathcal{M})}(z)\|\\
 =& \|\varphi_x^{-1}(0)+\nabla\varphi_x^{-1}(0)(\pi_{\mathcal{T}_x(\mathcal{M})}(z)-x)\\
& +\int_0^1 \nabla^2\varphi_x^{-1}\big(t(\pi_{\mathcal{T}_x(\mathcal{M})}(z)-x)\big)(\pi_{\mathcal{T}_x(\mathcal{M})}(z)-x)dt-\pi_{\mathcal{T}_x(\mathcal{M})}(z)\| \\
= & \|\int_0^1 \nabla^2\varphi_x^{-1}\big(t(\pi_{\mathcal{T}_x(\mathcal{M})}(z)-x)\big)(\pi_{\mathcal{T}_x(\mathcal{M})}(z)-x)dt\|\\
\leq & K \|\pi_{\mathcal{T}_x(\mathcal{M})}(z)-x\|^2\\
\leq &  K \|z-x\|^2.
\end{align*}

Then we obtain 
\begin{align*}
\|\varphi_x(z)-\varphi_x(x)\|^2&=\|x-z\|^2-\|z-\pi_{\mathcal{T}_x(\mathcal{M})}(z)\|^2\\
& \geq \|x-z\|^2-K^2 \|x-z\|^4\geq \frac{1}{4} \|x-z\|^2,
\end{align*}
so if $\|z-x\|\geq \frac{1}{2K}$ we have 
\begin{equation}\label{eq:firstproof}
\|\varphi_x(z)-\varphi_x(x)\|\geq \frac{1}{4K}.
\end{equation}

Suppose now the result is not true and take $y\in B^d(0,\frac{1}{4K})\cap \varphi_x(\overline{B}^p(x,\sqrt{\frac{3}{4}}K^{-1}))^c$ with $A^c$ denoting the complementary of the set $A$. Define $\tilde{y}=t_yy$ for
$$
t_y= \argmax \{t\in [0,1)\ | \ ty \in \varphi_x(\overline{B}^p(x,\sqrt{\frac{3}{4}}K^{-1}))\}. 
$$

As $\|\tilde{y}-\varphi_x(x)\|=\|\tilde{y}\|\leq \|y\|< \frac{1}{4K}$, we have by \eqref{eq:firstproof} that $\varphi_x^{-1}(\tilde{y})\in B^p(x,\frac{1}{2K}).$ But as $\varphi_x$ is a local diffeomorphism, there exists $\epsilon>0$ such that $B^d(\tilde{y},\epsilon)\in \varphi_x(\overline{B}^p(x,\sqrt{\frac{3}{4}}K^{-1}))$ which is in contradiction with the definition of $\tilde{y}$.
\end{proof}

This result enables to build a covering of $\mathcal{M}$ with images of the ball $B^d(0,\frac{1}{4K})$ by smooth diffeomorphisms. Write $\tau=\frac{1}{4K}$ and let $(x_i)_{i=1,...,m}$ be a finite sequence such that $\mathcal{M}\subset \bigcup \limits_{i=1}^m B^p(x_i,\frac{\tau}{2})$.

\begin{corollary}\label{corollary:newatlas}
If $\mathcal{M}$ verifies the $(2,K)$-manifold condition then $\mathcal{M}\subset \bigcup \limits_{i=1}^m\varphi_{x_i}^{-1}(B^d(0,\tau))$ and  for all $x\in \mathcal{M}$, there exists $i\in \{1,...,m\}$ such that $B^p(x,\frac{\tau}{2})\cap \mathcal{M}\subset \varphi_{x_i}^{-1}(B^d(0,\tau))$.
\end{corollary}

\begin{proof}
    Let $x\in B^p(x_i,\tau/2)\cap \mathcal{M}$, as the projection $\pi_{\mathcal{T}_{x_i}(\mathcal{M})}$ is $1$-Lipschitz, we have $\varphi_{x_i}(x)\in B^d(0,\tau/2)$. Furthermore
    $$
    B^p(x,\frac{\tau}{2})\cap\mathcal{M}\subset B^p(x_i,\tau)\cap\mathcal{M}\subset \varphi_{x_i}^{-1}(B^d(0,\tau)).$$
\end{proof}
For ease of notation, we will write $\varphi_i$ instead of $\varphi_{x_i}$. Let us now bound below the norm of the differential of the $\varphi_i$'s.

\begin{proposition}\label{prop:localdiffeo}
If $\mathcal{M}$ verifies the $(2,K)$-manifold condition of Definition \ref{def:manifoldcond}, then the charts $\varphi_i:B^d(0,\tau)\rightarrow \mathcal{M}$ verify that for all $u\in B^d(0,\tau)$ and $v\in \mathbb{R}^d$, $$\|\nabla \varphi_i(u)\frac{v}{\|v\|}\|\geq 1/2.$$
\end{proposition}
\begin{proof}
We have that $\nabla \varphi_i(x_i)=\text{Id}$ so for all $u\in B^d(0,\tau)$ and $v\in \mathbb{R}^d$ we have 
$$\|\nabla \varphi_i^{-1}(u)v\|=\|\nabla \varphi_i^{-1}(0)v+\int_0^1 \nabla^2 \varphi_i^{-1}(tu)uvdt\|\geq \|v\|-K\|u\| \|v\|\geq (1-\frac{1}{4})\|v\|$$
recalling that $\tau=\frac{1}{4K}$.
\end{proof}
In the following we will use the characterization of the reach for submanifolds.

\begin{lemma}\label{lemma:reachformanifolds}
(\cite{federer1959}, Theorem 4.18) The reach $r\in [0,\infty)$ of a submanifold $\mathcal{M}$ verifies 
$$
r=\inf \limits_{q\neq p\in \mathcal{M}} \frac{\|q-p\|^2}{2d(q-p,\mathcal{T}_p(\mathcal{M}))}.$$
\end{lemma}
Using this result, let us prove a lower bound on the reach resulting of the $(2,K)$-manifold condition.

\begin{proposition}\label{prop:reachfromregularity}
    Let $\mathcal{M}$ satisfying the $(2,K)$-manifold condition, then $r$ the reach of $\mathcal{M}$ verifies $r\geq \frac{1}{2K} $.
\end{proposition}

\begin{proof}
Let $x,y\in \mathcal{M}$ such that $ \|x-y\|\leq K^{-1}$, for $\varphi_y$ the orthogonal projection onto the tangent space $\mathcal{T}_y(\mathcal{M})$, we have 
\begin{align*}\label{align:reachbound}
   d(&x-y,\mathcal{T}_{y}(\mathcal{M})) \\
   & =  d(\varphi^{-1}_y(\varphi_y(x))-\varphi_y^{-1}(0),\mathcal{T}_{y}(\mathcal{M}))\\
    & \leq d(\nabla \varphi_y^{-1}(0)\varphi_y(x),\mathcal{T}_{y}(\mathcal{M}))+d(\int_0^1 \nabla^2 \varphi_y^{-1}(t\varphi_y(x))\varphi_y(x)^2(1-t)dt,\mathcal{T}_{y}(\mathcal{M}))\\
    & =d(\int_0^1 \nabla^2 \varphi_y^{-1}(t\varphi_y(x))\varphi_y(x)^2(1-t)dt,\mathcal{T}_{y}(\mathcal{M}))\\
    & \leq \frac{K}{2}\|\varphi_y(x)\|^2 \leq \frac{K}{2}\|x-y\|^2.
\end{align*}
Then $\|x-y\|\leq K^{-1}$ imply that $$\frac{\|x-y\|^2}{2d(x-y,\mathcal{T}_{y}(\mathcal{M}))}\geq \frac{\|x-y\|^2}{2\frac{K}{2}\|x-y\|^2}= K^{-1}.$$ 
Now if $\|x-y\|\geq K^{-1}$ then 
 $$\frac{\|x-y\|^2}{2d(x-y,\mathcal{T}_{y}(\mathcal{M}))}\geq \frac{\|x-y\|^2}{2\|x-y\|}\geq \frac{1}{2K}.$$ 
\end{proof}
\subsection{Proof of Proposition \ref{prop:keydecomp}}\label{sec:prop:keydecomp}
Let us first define the notion of approximation of unity subordinated to an atlas.

\begin{definition}
 A collection of $C^{\infty}$ functions $(\rho_i)_{i=1,...,m}$ on $\mathcal{M}$ is said to be an approximation of unity subordinated to $(\varphi_i^{-1}(B^d(0,\tau)))_i$ if for all $i\in \{1,...,m\}$,
\begin{itemize}
    \item  $0\leq \rho_i\leq 1$ and $\sum \limits_{i=1}^m \rho_i(x)=1$ for $x\in \mathcal{M}$,
    \item $supp(\rho_i)\subset \varphi_i^{-1}(B^d(0,\tau))$.
\end{itemize}
\end{definition}
It is well known that there exists an approximation of unity subordinated to any atlas on a closed manifold \citep{jost2008riemannian}. Using this approximation of unity we can now give the proof of Proposition \ref{prop:keydecomp}.

\begin{proof}
As $\mu$ has a density $f_{\mu}$ with respect to the volume measure on $\mathcal{M}$, then by change of variable we have that the measure $\varphi_{i\# \mu}$ admits a density $f^i_{\mu}$ with respect to the Lebesgue measure on $B^d(0,\tau)$ given by
$$f^i_{\mu}(u)=f_{\mu}(\varphi_{i}^{-1}(u))\text{det}((\nabla \varphi_{i}^{-1}(u))^\top\nabla \varphi_{i}^{-1}(u))^{1/2}.$$
Furthermore, as $\mu$ verifies the $(\beta,K)$-regularity density condition, we have by the Faa di Bruno formula that $f^i_{\mu}\in \mathcal{H}^\beta_{C}(B^d(0,\tau),\mathbb{R})$ and $f^i_{\mu}$ is bounded below by $\frac{1}{2K}$. Write $\lambda^d_i$ the Lebesgue measure on $B^d(0,\tau)$ normalized by $C_i^{-1}=\int f^i_{\mu} d\lambda^d/\int_{B^d(0,\tau)} d\lambda^d$. Then by Caffarelli's regularity (Theorem 12.50, \cite{villani2009optimal}) we have that there exists $T_i\in \mathcal{H}^{\beta+1}_{C}(B^d(0,\tau),B^d(0,\tau))$ such that $T_{i\# \lambda^d_i}=\varphi_{i\# \mu}$. Furthermore, $T_i$ is solution to the equation
$$\text{det}(\nabla T_i)^{-1}=f^i_{\mu}\circ T_i$$
so in particular $\text{det}(\nabla T_i)\geq C^{-1}$. As $\lambda_{\max}(\nabla T_i)\leq C$ we deduce that $\lambda_{\min}(\nabla T_i)\geq C^{-1}$. Therefore we have that $\lambda_{\min}(\nabla (\varphi_i^{-1}\circ T_i))\geq C^{-1}$. Using Theorem \ref{whitney}, we can extend $\varphi_i^{-1}\circ T_i$ to a map $\phi_i \in\mathcal{H}^{\beta+1}_{C}(\mathbb{R}^d,\mathbb{R}^p)$.

Let $h:\mathbb{R}^p\rightarrow \mathbb{R}$ a bounded continuous function, then for $\rho_i$ an approximation of unity subordinated to $(\varphi_i^{-1}(B^d(0,\tau)))_i$, we have 
\begin{align*}
    \mathbb{E}_{\mu}[h(X)] & =\sum \limits_{i=1}^m\mathbb{E}_{\mu}[h(X)\rho_i(X)]=\sum \limits_{i=1}^m\int_{B^d(0,\tau)}h(\varphi_i^{-1}(u))\rho_i(\varphi_i^{-1}(u))d\varphi_{i\# \mu}(u)\\
    & = \sum \limits_{i=1}^m\frac{1}{C_i}\int_{B^d(0,\tau)}h(\varphi_i^{-1}\circ T_i(u))\rho_i(\varphi_i^{-1} \circ T_i(u))d\lambda^d_{i}(u)\\
    & = \sum \limits_{i=1}^m\frac{1}{C_i}\int_{\mathbb{R}^d}h(\phi_i(u))\rho_i(\phi_i(u))\mathds{1}_{\{u\in B^d(0,\tau)\}}d\lambda^d_{i}(u).\\
\end{align*}
Let $\zeta_i:\mathbb{R}^d\rightarrow \mathbb{R}$ defined by 
$$
\zeta_i(u)=\rho_i( \phi_i(u))\mathds{1}_{\{u\in B^d(0,\tau)\}}.
$$
Then as $\rho_i$ is subordinated to $(\varphi_i^{-1}(B^d(0,\tau)))_i$, we have that $\zeta_i\in \mathcal{H}^{\beta+1}_C(\mathbb{R}^d, \mathbb{R})$. Finally,
$$\mathbb{E}_{\mu}[h(X)]=\sum \limits_{i=1}^m\frac{1}{C_i}\int_{\mathbb{R}^d}h( \phi_i(u))\zeta_i(u)d\lambda^d_{i}(u),$$
which concludes the proof.
\end{proof}

\section{Proof of the existence of a diffeomorphism between the manifolds (Section \ref{sec:exdiff})}

Let us recall two results from \cite{stephanovitch2023wasserstein} :

\begin{proposition}\label{prop:Hölder}
Let $h_1\in \mathcal{B}^{s_1,b_1}_{\infty,\infty}(\mathbb{R}^m,\mathbb{R}^p,1)$ , $h_2\in \mathcal{B}^{s_2,b_2}_{\infty,\infty}(\mathbb{R}^m,\mathbb{R}^p,1)$ for $s_1,s_2,b_1,b_2\in \mathbb{R}$. Then for $\tau\in \mathbb{R}$, $t,r\in[0,1]$, $q>1$ and $1/q+1/q^\star=1$ we have
$$
\Big\langle h_1,h_2\Big\rangle_{L^2}\leq \Big\langle \tilde{\Gamma}^{t\tau}(h_1),\Gamma^{(1-t)\tau}(h_2)\Big\rangle_{L^2}^{\frac{1}{q}}\Big\langle \tilde{\Gamma}^{-r\frac{q^\star}{q}\tau}(h_1),\Gamma^{-(1-r)\frac{q^\star}{q}\tau}(h_2)\Big\rangle_{L^2}^{\frac{1}{q^{\star}}}
$$
for $\tilde{\Gamma}^{t\tau}(h_1)\in \mathcal{B}^{s_1+t\tau,b_1}_{\infty,\infty}$ and $\tilde{\Gamma}^{-r\frac{q^\star}{q}\tau}(h_1)\in \mathcal{B}^{s_1-r\frac{q^\star}{q},b_1}_{\infty,\infty}$such that
$$
\langle \tilde{\Gamma}^{t\tau}(h_1)_i,\psi_{jlz}\rangle=S(j,l,w)_i\langle \Gamma^{t\tau}(h_1)_i,\psi_{jlz}\rangle
$$ 
and 
$$
\langle \tilde{\Gamma}^{-t\frac{q^\star}{q}\tau}(h_1)_i,\psi_{jlz}\rangle=S(j,l,w)_i\langle \Gamma^{-r\frac{q^\star}{q}\tau}(h_1)_i,\psi_{jlz}\rangle
$$
with $S(j,l,w)_i \in \{-1,1\}$ the sign of $\langle h_{1_i},\psi_{jlz}\rangle \langle h_{2_i},\psi_{jlz}\rangle$, $i=1,...,p$.
\end{proposition}

This proposition is a key result that we will use numerous times along the proof of Theorem \ref{theo:theineq}. It is the classical interpolation inequality between Besov spaces of Section \ref{sec:classicalineq} in a more general form. An important Corollary is its applications to Hölder IPMs.

\begin{corollary}\label{coro:ineq without reg}
   Let $\mu,\mu^\star$ two probability measures with compact support in $\mathbb{R}^p$. Then for any $\theta,\theta_1,\theta_2>0$ such that $\theta_1<\theta<\theta_2$,  we have for all $\epsilon>0$
    $$d_{ \mathcal{H}_1^{\theta}}(\mu,\mu^\star)\leq C\log(\epsilon^{-1})^{2\delta_{\theta_1,\theta_2}}d_{ \mathcal{H}_1^{\theta_1}}(\mu,\mu^\star)^\frac{\theta_2-\theta}{\theta_2-\theta_1}d_{ \mathcal{H}_1^{\theta_2}}(\mu,\mu^\star)^\frac{\theta-\theta_1}{\theta_2-\theta_1}+\epsilon\delta_{\theta_1,\theta_2},$$
    with $\delta_{\theta_1,\theta_2}=1$ if $\theta_1$ or $\theta_2$ is an integer and $\delta_{\theta_1,\theta_2}=0$ otherwise.
\end{corollary}

Let us define the notion of approximate Jacobian (Definition 2.10 in \cite{federer1959}).
\begin{definition}\label{def:approx}
    Let $\mathcal{M}$ be a $d$-dimensional submanifold of $\mathbb{R}^p$ and a map $T:\mathbb{R}^p\rightarrow \mathcal{M}$ of regularity $C^1$. For $x\in \mathbb{R}^p$, using the matrix of $\nabla T(x)$ with respect to orthonormal basis
of $\mathbb{R}^p$ and $\mathcal{T}_{T(x)}(\mathcal{M})$, the approximate Jacobian of $T$, noted $\text{ap}_d(\nabla T(x))$, is equal to the square root of the sum of the squares
of the determinants of the $d$ by $d$  minors of this matrix.
    \end{definition}

We will use multiple times the next proposition.
\begin{proposition}\label{approx}
(Theorem 3.2.22, \cite{federer2014geometric}) For $t>0$, $T:\mathcal{M}^{\star t}\rightarrow \mathcal{M}^\star$ of regularity $C^1$ and $D:\mathcal{M}^{\star t}\rightarrow \mathbb{R}$, we have
$$
\int_{\mathcal{M}^{\star t}}  D(x)\text{ap}_d(\nabla T(x))d\lambda^{p}(x)=\int_{\mathcal{M}^\star} \int_{T^{-1}(\{z \})} D(x)d\lambda^{p-d}(x)d\lambda_{\mathcal{M}^\star}(z),
$$
\end{proposition}
for $\lambda_{\mathcal{M}^\star}$ the volume measure on the submanifold $\mathcal{M}^\star$.

\subsection{Proof of Lemma \ref{lemma:hausdorff}}\label{sec:lemma:hausdorff}
\begin{proof}
If $\alpha$ is not an integer, using Corollary \ref{coro:ineq without reg} with $\theta=1$, $\theta_1=1/2$ and $\theta_2=\alpha$, we have 
    \begin{align*}      
    W_1(\mu,\mu^\star) & \leq 2K d_{\mathcal{H}^{1}_1}(\mu,\mu^\star)\\
    & \leq C d_{\mathcal{H}^{1/2}_1}(\mu,\mu^\star)^\frac{\alpha-1}{\alpha-1/2}d_{\mathcal{H}^{\alpha}_1}(\mu,\mu^\star)^\frac{1}{2\alpha-1}\\
     & \leq C d_{\mathcal{H}^{\alpha}_1}(\mu,\mu^\star)^\frac{1}{2\alpha-1}
    \end{align*}
If now $\alpha$ is an integer, then using Corollary \ref{coro:ineq without reg} with $\theta=1$, $\theta_1=\frac{\alpha-3/2}{2(\alpha-1)}$ and $\theta_2=\alpha+1/2$, we get the same result.
    Then by hypothesis we have 
    \begin{align*}       
    W_1(\mu,\mu^\star)  
     \leq Ct^{d+1}.
    \end{align*}
Let $\delta>0$ and suppose that $\mathbb{H}(\mathcal{M},\mathcal{M}^\star) \geq\delta$. For $x\in \mathcal{M}$ such that $d(x,\mathcal{M}^\star)\geq \delta$, write $\nu=\mu(\cdot\cap B^p(x,\delta/2))$. Then has $\mu$ has a density with respect to the volume measure of $\mathcal{M}$ bounded below by $K^{-1}$, we have that
$$\nu(B^p(x,\delta/2)\cap \mathcal{M})\geq K^{-1}\lambda_{\mathcal{M}}(\mathcal{M}\cap B^p(x,\delta/2)\geq C\delta^d.$$

Using the dual formulation of the $1$-Wasserstein distance, we have that 
\begin{align*}
    W_1(\mu,\mu^\star)&  = \sup \limits_{D \in \text{Lip}(1)}\mathbb{E}_{\substack{X\sim \mu^\star \\ Y\sim \mu}}[D(X)-D(Y)]\\
    & \geq \int_{B^p(x,\delta/2)} d(y,\mathcal{M}^\star)d\mu(y)\geq \frac{\delta}{2}\mu(B^p(x,\delta/2)\cap \mathcal{M})\\
    & \geq C \delta^{d+1}.
\end{align*}    
Therefore as $W_1(\mu,\mu^\star)  
     \leq Ct^{d+1}$, we can deduce that $$\mathbb{H}(\mathcal{M},\mathcal{M}^\star)\leq C t.$$
\end{proof}

\subsection{Proof of Theorem \ref{theo:existencecompmap}}\label{sec:theo:existencecompmap}
Let us first use Proposition 3.1 in \cite{fefferman2015reconstruction} to show that there exists a $C^\infty$ manifold close to $\mathcal{M}^\star$.

As every charts $\phi_i$ (given by Proposition \ref{prop:keydecomp}) of $\mathcal{M}^\star$ belongs to $\mathcal{H}^2_K(B^d(0,\tau),\mathcal{M})$, we have in particular that for all $x\in \mathcal{M}^\star$ and $0<s<\tau$, 
$$
\mathbb{H}(B^p(x,s)\cap \mathcal{M}^\star,T_x(\mathcal{M}))\leq Ks^2.
$$
Then applying Proposition 3.1 in \cite{fefferman2015reconstruction} we have that there exists a $C^\infty$ closed manifold $\mathcal{M}_s$ such that $
\mathbb{H}(\mathcal{M}^\star,\mathcal{M}_s)\leq 5s^2
$ with charts $\varphi_{x}^s\in \mathcal{H}^\eta_{C_{\eta}}(B^d(0,s), \mathcal{M}_s)$ $\forall \eta>0$, and its reach $r_s$ verifies $r_s\geq C^{-1}s$. We are going to use the orthogonal projection $\pi_s$ onto the submanifold $\mathcal{M}_s$ to build our map.

From \cite{Leobacher} we know that the canonical projection $\pi$ onto a submanifold $\mathcal{M}$ with strictly positive reach $r$, verifies
\begin{equation}\label{grad}
\nabla \pi(x)=\left(\text{Id}_{\mathcal{T}_{\pi(x)}(M)}-\|x-\pi(x)\|L_{\pi(x),v} \right)^{-1}P_{\mathcal{T}_{\pi(x)}(M)}.
\end{equation}

for $L_{\pi(x),v}$ the shape operator in the direction $v=\frac{x-\pi(x)}{\|x-\pi(x)\|}$. From Corollary 3 in \cite{Leobacher}, we also have that 
\begin{equation}\label{eq:shapeop}
\|\left(\text{Id}_{\mathcal{T}_{\pi(x)}(\mathcal{M})}-\|x-\pi(x)\|L_{\pi(x),v} \right)^{-1}\|\leq (1-\|x-\pi(x)\|/r)^{-1}.
\end{equation}
Furthermore if $\mathcal{M}$ is of regularity $\eta$, there exists a constant $C_{\eta}>0$ depending on the $\mathcal{H}^{\eta}$ norm of the charts such that
\begin{equation}\label{gradborn}
\|\pi_{|_{\mathcal{M}^{r/2}}}\|_{\mathcal{H}^{\eta-1}}\leq C_{\eta}.
\end{equation}
Using these results, let us show that the restriction of
$\pi_s$ to the manifold $\mathcal{M}^\star$ is a diffeomorphism.

\begin{lemma}\label{lemma:neardiffeo} Let $\mathcal{M}_s,\mathcal{M}^\star$ satisfying the $(\beta+1,K)$-manifold condition such that $\mathbb{H}(\mathcal{M}_s,\mathcal{M}^\star)\leq 5s^2$. Supposing $s>0$ small enough,
    the map $\pi_{s|_{\mathcal{M}^\star}}:\mathcal{M}^\star\rightarrow \mathcal{M}_{s}$, i.e. the restriction of the projection $\pi_s$ to the manifold $\mathcal{M}^\star$, is a diffeomorphism.
\end{lemma}

\begin{proof}
Let us first show that $\pi_{s|_{\mathcal{M}^\star}}$ is a local diffeomorphism, by showing that if there existed $x\in \mathcal{M}$ such that $\lambda_{\min}(\nabla (\pi_{s}\circ \varphi_x^{-1})(0)^\top \nabla (\pi_{s}\circ \varphi_x^{-1})(0))$ was too small, then the Hausdorff distance between $\mathcal{M}_{s}$ and $\mathcal{M}^\star$ would be too large.\\
Let $x\in \mathcal{M}^\star$ and $h\in(0,\frac{\tau}{2K})$. For $v\in \mathbb{R}^d$ with $\|v\|=1$ such that $\langle \pi_{s}(\varphi_x^{-1}(0))-\varphi_x^{-1}(0), \nabla(\pi_{s}\circ \varphi_x^{-1} - \varphi_x^{-1})(0)v\rangle\geq0$, using a Taylor expansion we have
\begin{align*}
    \| \pi_{s}&  (\varphi_x^{-1}(hv))-\varphi_x^{-1}(hv)\|\\
    \geq & \|\pi_{s}(\varphi_x^{-1}(0))-\varphi_x^{-1}(0)+ \nabla (\pi_{s}\circ \varphi_x^{-1} - \varphi_x^{-1})(0)hv\| - Ch^2\\
     = &\Big(\|\pi_{s}(\varphi_x^{-1}(0))-\varphi_x^{-1}(0)\|^2 +2\langle \pi_{s}(\varphi_x^{-1}(0))-\varphi_x^{-1}(0), \nabla (\pi_{s}\circ \varphi_x^{-1} - \varphi_x^{-1})(0)hv\rangle\\
   & +  \|\nabla (\pi_{s}\circ \varphi_x^{-1} - \varphi_x^{-1})(0)hv \|^2\Big)^{1/2} - Ch^2\\
     \geq & h\|\nabla (\pi_{s}\circ \varphi_x^{-1} - \varphi_x^{-1})(0)v \| - Ch^2\\
     \geq & h(\|\nabla \varphi_x^{-1}(0)v \|-\|\nabla (\pi_{s}\circ \varphi_x^{-1})(0)v \|) - Ch^2\\
     \geq & h(1/2-\|\nabla (\pi_{s}\circ \varphi_x^{-1})(0)v \|) - Ch^2.
\end{align*}
As $\mathbb{H}(\mathcal{M}_{s},\mathcal{M}^\star)\leq 5s^2$ we deduce with $h=\frac{\tau(1/2-\|\nabla (\pi_{s}\circ \varphi_x^{-1})(0)v \|)}{2C}$ that 
$$\frac{\tau^2(1/2-\|\nabla (\pi_{s}\circ \varphi_x^{-1})(0)v \|)^2}{4C}\leq 5s^2$$
so $$1/2-\sqrt{20C}\frac{s}{\tau}\leq\|\nabla (\pi_{s}\circ \varphi_x^{-1})(0)v \|$$ which gives for $s$ small enough
\begin{align}\label{eq:lambdamin}
    \lambda_{\min}(\nabla (\pi_{s}\circ \varphi_x^{-1})(0) ^\top \nabla (\pi_{s}\circ \varphi_x^{-1})(0))^{1/2} \geq \frac{1}{4}.
\end{align}

In particular $\pi_s\circ \varphi_x^{-1}$ is an immersion and therefore $\pi_{s|_{\mathcal{M}^\star}}:\mathcal{M}^\star\rightarrow \mathcal{M}_{s}$ is a local diffeomorphism.

Suppose there exists $x,y\in\mathcal{M}^\star$ different such that $\pi_s(x)=\pi_s(y)$. Then 
$$\|x-y\|\leq \|x-\pi_s(x)\|+\|y-\pi_s(y)\|\leq 2\mathbb{H}(\mathcal{M}^\star,\mathcal{M}_s)\leq 10s^2.$$
Then taking $10s^2\leq \frac{1}{K},$ we have that $y\in \varphi^{-1}_x(B(0,10s^2))$. But from \eqref{eq:lambdamin} and the fact that $\|\nabla^2 (\pi_{s}\circ \varphi_x^{-1})\|\leq C$ we have that for $s$ small enough, $\pi_{s}\circ \varphi_x^{-1}$ is a diffeomorphism on $B^d(0,10s^2)$ and therefore  $\pi_s(x)\neq\pi_s(y)$.

As $\mathcal{M}^\star$ is compact without boundary, $\pi_s(\mathcal{M}^\star)$ is a closed submanifold of  $\mathcal{M}_{s}$  so being of same dimension, they are equal. We can conclude that $\pi_{s|_{\mathcal{M}^\star}}$ is a diffeomorphism.
\end{proof}

We can now define the map $T:\mathcal{M}^{\star r_{s}/4}\rightarrow \mathcal{M}^\star$ by 
$$T(x)=(\pi_{s|_{\mathcal{M}^\star}})^{-1} \circ \pi_s (x)$$
which is well defined for $s$ small enough as $\mathbb{H}(\mathcal{M}_{s},\mathcal{M}^\star)\leq 5s^2< r_{s}/4$. Let us now state that $T$ is the map we are interested in.

\begin{proposition}
    The map $T:\mathcal{M}^{\star r_{s}/4}\rightarrow \mathcal{M}^\star$ defined by 
$T(x)=(\pi_{s|_{\mathcal{M}^\star}})^{-1} \circ \pi_s (x)$, is $(\mathcal{M},\mathcal{M}^\star)_{K_T}$-compatible with radius $t=C_{\alpha}^{-1}$and  $K_T=C_\alpha$.
\end{proposition}

\begin{proof}
We have 
\begin{align}\label{gradT}
    \nabla T(x)=  \nabla (\pi_{s|_{\mathcal{M}^\star}})^{-1} (\pi_s(x))\circ \nabla \pi_s(x).
\end{align}
From \eqref{eq:lambdamin} and the fact that for all $x\in \mathcal{M}^\star$, $\varphi_x^{-1}$ is $2$-Lipschitz on $B^d(0,\tau)$, we have for all $h\in \mathcal{T}_{x}(\mathcal{M}^\star)$, \begin{equation}\label{ineqpis}
\|\nabla \pi_{s|_{\mathcal{M}^\star}}h\|\geq \frac{1}{8}\|h\|.
\end{equation}
Then, from \eqref{gradT} and the Faa di Bruno formula, we deduce that $T\in \mathcal{H}^{\beta+1}_{C_{s}}(\mathcal{M}_{s}^{r_{s}/4},\mathcal{M}^\star)$.

Supposing $d_{\mathcal{H}^{\alpha}_1}(\mu,\mu^\star)\leq s^{2(d+1)(2\alpha-1)}$ we have from Lemma \ref{lemma:hausdorff} that $\mathbb{H}(\mathcal{M},\mathcal{M}^\star)\leq Cs^2$ so
$\mathbb{H}(\mathcal{M}_{s},\mathcal{M})\leq \mathbb{H}(\mathcal{M},\mathcal{M}^\star) +\mathbb{H}(\mathcal{M}_{s},\mathcal{M}^\star)\leq Cs^2 $. We show the same way as we did with $\mathcal{M}^\star$, that for all $x\in \mathcal{M}$, $\pi_{s}\circ\phi_x^{-1}$ is a local diffeomorphism for $\phi_x^{-1}$ the charts of $\mathcal{M}$ and that $\pi_s$ is injective on $\mathcal{M}$, so $\pi_{s|\mathcal{M}}$ is a diffeomorphism. Therefore, we deduce that $T_{|\mathcal{M}}$ is also a diffeomorphism.

Doing the same derivation as for $\mathcal{M}^\star$  to obtain \eqref{eq:lambdamin}, we get for $\mathcal{M}$  that
$$\lambda_{\min}(\nabla (\pi_{s}\circ \phi_x^{-1})(0)^\top \nabla (\pi_{s}\circ \phi_x^{-1})(0))^{1/2} \geq C^{-1}.
$$
Furthermore, as for all $y\in \mathcal{M}^\star$ and 
 $ h \in \mathcal{T}_{y}(\mathcal{M}^\star)$, we have
$$\|\nabla  \pi_{s|_{\mathcal{M}^\star}}(y)h\|\leq 2\|h\|,$$

we deduce that
\begin{equation}\label{eq:boundlambda}
    \lambda_{\min}\big(\nabla(T\circ \phi_x^{-1})(0)^\top \nabla(T\circ \phi_x^{-1})(0)\big)^{1/2}\geq C^{-1}.
\end{equation}
Then as $\phi_x^{-1}$ is $C$-Lipschitz, we have that
\begin{equation*} 
    \lambda_{\min}\big(\nabla T_{|_{\mathcal{M}}}(X)^\top \nabla T_{|_{\mathcal{M}}}(X)\big)^{1/2}\geq C^{-1},
\end{equation*}
so we deduce that $T_{|_{\mathcal{M}}}^{-1}\in \mathcal{H}^{\beta+1}_{C_s}(\mathcal{M}^\star,\mathcal{M})$. Therefore point $i)$ is verified. 

From Theorem 4.8 in \cite{federer1959} we have that $\pi_s:\mathcal{M}_s^{r_s}\rightarrow \mathcal{M}_s$ the projection onto $\mathcal{M}_s$ verifies point $ii)$ of Definition \ref{defi:compatibility} for a radius $t=r_s $, so as $\mathbb{H}(\mathcal{M}^\star,\mathcal{M}_{s})\leq Cs^2 <r_s$, we deduce that $T$ verifies $ii)$ for a radius $r_s/2$. 

Let us prove point $iii)$. Let  $f\in \mathcal{H}^{\beta+1}_{C_s}(\mathcal{M}_{s},(\mathbb{R}^p)^{p-d})$ such that $(f_1(x),...,f_{p-d}(x))$ is an orthonormal basis of $\mathcal{T}_x(\mathcal{M}_{s})^\top$. Let $R_x:\mathbb{R}^{p-d}\rightarrow \mathbb{R}^p$ a linear map such that $R_x(e_i)=f_i(x)$ for $(e_1,...,e_{p-d})$ the canonical basis of $\mathbb{R}^{p-d}$. Then for $h\in \mathcal{H}^{\eta}_1(\mathcal{M}^{\star t},\mathbb{R})$ and $x\in \mathcal{M}^\star$, we have
\begin{align*}
    \int_{T^{-1}(\{x\})}h(y)d\lambda^{p-d}_{E_x}(y) & = \int_{\mathcal{T}_{\pi_s(x)}(\mathcal{M}_{s})^\perp\cap B(x,t)}h(y)d\lambda^{p-d}_{\mathcal{T}_{\pi_s(x)}(\mathcal{M}_{s})^\perp}(y)\\
    & = \int_{\mathbb{R}^{p-d}\cap B(0,t)}h(x+R_x(z))d\lambda^{p-d}(z).
\end{align*}
 Then as $\forall z \in \mathbb{R}^{p-d}$, we have that the map $x\mapsto x +R_x(z)$ belongs to $\mathcal{H}^{\beta+1}_{C_s}(\mathcal{M}^\star,\mathbb{R}^p)$, so $x\mapsto h(x +R_x(z))$ belongs to $\mathcal{H}^{\eta}_{C_s}(\mathcal{M}^\star,\mathbb{R}^p)$ and therefore $T$ verifies $iii)$.

Let us prove $iv)$. Using \eqref{eq:lambdamin} and the fact that $\pi_s$ is $2$-Lipschitz on $\mathcal{M}^\star$, we have that for all $x\in \mathcal{M}^{\star t}$,
$$\inf \limits_{z\in \mathcal{T}_{T(x)}(\mathcal{M}^\star)} \|\nabla T(x)\frac{z}{\|z\|}\|\geq C^{-1}.$$ 
Therefore, we deduce that the approximate Jacobian of $T$ verifies  
$$
\text{ap}_d(\nabla T(x))\geq  C^{-1},
$$
so $T$ verifies the point $iv)$ for $s>0$ small enough.
\end{proof}

\section{Proofs of the interpolation inequalities in the manifold case (Section \ref{sec:decomp}) and additional results}

\subsection{Extension Lemma}
In this section we prove a general result allowing to extend maps that are defined on the neighborhood of a submanifold. This result is a corollary of the following version of Whitney's extension Theorem (Theorem A in \cite{fefferman2005sharp} and Theorem 4 in Chapter 6 of \cite{SingularStein}). 

\begin{theorem}\label{whitney} Given $\eta,p \geq 1$, there exists $k\in \mathbb{N}$ depending only on $p$ and $\beta$, for which the following holds. Let $f:E\rightarrow \mathbb{R}$ a closed subset of $\mathbb{R}^p$. Suppose that for any $k$ distinct points $x_1,...,x_k\in E$, there exists $\lfloor \eta \rfloor $ degree polynomials $P_1,...,P_k$ on $\mathbb{R}^p$, satisfying
\begin{itemize}
\item[a)] $P_i(x_i)=f(x_i)$ for $i=1,...,k$;
\item[b)] $\|\nabla^\alpha P_i(x_i)\|\leq M $ for $i=1,...,k$ and $|\alpha|\leq \lfloor \eta \rfloor$ and 
\item[c)] $\|\nabla^\alpha( P_i-P_j)(x_i)\|\leq M\|x_i-x_j\|^{\eta-\alpha} $ for $i,j=0,...,k$ and $|\alpha|\leq \lfloor \eta \rfloor$ with $M$
 independent of $x_1,...,x_k$
\end{itemize}
Then $f$ extends to  $\mathcal{H}_{C_M}^{\eta}(\mathbb{R}^p,\mathbb{R})$.
\end{theorem}
Using this theorem we obtain the following extension result.

\begin{proposition}\label{prop:extensionH}
Let $E\subset \mathbb{R}^p$ a closed set and $\mathcal{M}^{\star}$ a closed submanifold covered by an atlas $(\varphi_{i}^{-1}(B^d(0,K^{-1})))_{\{i=1,...m\}}$  with $\varphi_{i}^{-1}\in \mathcal{H}^1_K$ and with reach $r^\star\geq K^{-1}$. Suppose that 
\begin{itemize}
    \item[i)] $d(E,\mathcal{M}^\star)<r^\star/4$,
    \item[ii)] $\mathcal{M}^\star\subset E$,
    \item[iii)] there exists a map $\pi\in \mathcal{H}^0_C(E,\mathcal{M}^\star)$ such that for all $x \in E$ and $t\in [0,1]$, $x+t(\pi(x)-x)\in E$,
    \item[iv)] there exists a map $D\in \mathcal{H}^1_C(\mathcal{M}^{\star  r^\star/2},E)$ such that $D_{|E}=Id$.
\end{itemize}
    Then for all maps $H\in \mathcal{H}^{\eta}_{M}(E,\mathbb{R})$, $H$ can be extended into a map in $\mathcal{H}^{\eta}_{C_{M}}(\mathbb{R}^p,\mathbb{R})$.
\end{proposition}

\begin{proof} Let $x_1,...,x_k\in E$ and define for all $r\in \{1,...,k\}$
$$
P_r(x)=\sum \limits_{\alpha=0}^{\lfloor \eta \rfloor} \nabla^\alpha H(x_r)\frac{(x-x_r)^\alpha}{\alpha!}.
$$
with the notations $\nabla^\alpha H(x_r)$ the $\alpha$-differential of $H$ and $$ \nabla^\alpha H(x_r)(x-x_r)^\alpha=\nabla^\alpha H(x_r)\big(x-x_r,x-x_r,...,x-x_r).$$
Then the polynomials $P_r$ check conditions a) and b) of Theorem \ref{whitney} for $M=\|H\|_{\mathcal{H}^{\eta}(E,\mathbb{R})}$ . Let us show that they also verify condition c).

First, for $\alpha=\lfloor \eta \rfloor$, we have $$\|\nabla^{\lfloor \eta \rfloor} P_r(x_r)-\nabla^{\lfloor \eta \rfloor} P_j(x_r)\| =\|\nabla^{\lfloor \eta \rfloor} H(x_r)-\nabla^{\lfloor \eta \rfloor} H(x_j)\|\leq C\|x_r-x_j\|^{\eta-\lfloor \eta \rfloor}.$$
Let $\alpha \in \{0,...,\lfloor \eta \rfloor-1\}$ and take a path $\gamma\in \mathcal{H}^1([0,1],E)$ such that $\gamma(0)=x_j$ and $\gamma(1)=x_r$. Applying Taylor expansions recursively we obtain
\begin{align*}
\nabla^\alpha &P_r(x_r)  =  \nabla^\alpha H(x_r)= \nabla^\alpha H(x_j)+\int_0^1 \nabla^{\alpha+1}H(\gamma(t_1)) \dot{\gamma}(t_1)dt_1\\
 = &\nabla^\alpha H(x_j)+\int_0^1 \left(\nabla^{\alpha+1}H(x_j) + \int_0^1 \nabla^{\alpha+2}H(\gamma(t_2t_1))t_1\dot{\gamma}(t_2t_1))d_{t_2}\right)\dot{\gamma}(t_1)dt_1\\
  = & \nabla^\alpha H(x_j)\\
  & + \sum \limits_{l=1}^{\lfloor \eta \rfloor-\alpha-1} \int_0^1... \int_0^1 \nabla^{\alpha+l}H(x_j)\dot{\gamma}(t_lt_{l-1}...t_1)d_{t_l}\dot{\gamma}(t_{l-1}t_{l-2}...t_1)t_{l-1}d_{t_{l-1}}...\dot{\gamma}(t_1)t_1^{l-1}dt_1 \nonumber\\
  & + \int_0^1... \int_0^1 \nabla^{\lfloor \eta \rfloor}H(\gamma(t_{\lfloor \eta \rfloor-\alpha}...t_1)\dot{\gamma}(t_{\lfloor \eta \rfloor-\alpha}...t_1)d_{t_{\lfloor \eta \rfloor-\alpha}}\\
  & \qquad \times \dot{\gamma}(t_{\lfloor \eta \rfloor-\alpha-1}...t_1)t_{t_{\lfloor \eta \rfloor-\alpha-1}}d_{t_{\lfloor \eta \rfloor-\alpha-1}}...\dot{\gamma}(t_1)t_1^{\lfloor \eta \rfloor-\alpha-1}dt_1\\
 = &\nabla^\alpha H(x_j)\\
 & + \sum \limits_{l=1}^{\lfloor \eta \rfloor-\alpha} \int_0^1... \int_0^1 \nabla^{\alpha+l}H(x_j)\dot{\gamma}(t_lt_{l-1}...t_1)d_{t_l}\dot{\gamma}(t_{l-1}t_{l-2}...t_1)t_{l-1}d_{t_{l-1}}...\dot{\gamma}(t_1)t_1^{l-1}dt_1 \nonumber\\
& + \int_0^1... \int_0^1 \big(\nabla^{\lfloor \eta \rfloor}H(\gamma(t_{\lfloor \eta \rfloor-\alpha}...t_1)-\nabla^{\lfloor \eta \rfloor}H(x_j)\big)\dot{\gamma}(t_{\lfloor \eta \rfloor-\alpha}...t_1)d_{t_{\lfloor \eta \rfloor-\alpha}}...\dot{\gamma}(t_1)t_1^{\lfloor \eta \rfloor-\alpha-1}dt_1.
\end{align*} 
Using integration by part one can show that
\begin{align*}
\sum \limits_{l=1}^{\lfloor \eta \rfloor-\alpha} &\int_0^1... \int_0^1 \nabla^{\alpha+l}H(x_j)\dot{\gamma}(t_lt_{l-1}...t_1)d_{t_l}\dot{\gamma}(t_{l-1}t_{l-2}...t_1)t_{l-1}d_{t_{l-1}}...\dot{\gamma}(t_1)t_1^{l-1}dt_1\\
& =     \sum \limits_{l=1}^{\lfloor \eta \rfloor-\alpha} \nabla^{\alpha+l}H(x_j)\frac{1}{l!}(x_r-x_j)^l,
\end{align*}
so from the previous derivation we get
\begin{align*}
    &\|\nabla^\alpha   P_r(x_r) - \nabla^\alpha P_j(x_r)\|\\
    & = \| \int_0^1... \int_0^1 \big(\nabla^{\lfloor \eta \rfloor}H(\gamma(t_{\lfloor \eta \rfloor-\alpha}...t_1)-\nabla^{\lfloor \eta \rfloor}H(x_j)\big)\dot{\gamma}(t_{\lfloor \eta \rfloor-\alpha}...t_1)d_{t_{\lfloor \eta \rfloor-\alpha}}...\dot{\gamma}(t_1)t_1^{\lfloor \eta \rfloor-\alpha-1}dt_1\|\\
    & \leq  \int_0^1... \int_0^1 C\|\gamma(t_{\lfloor \eta \rfloor-\alpha}...t_1)-x_j\|^{\beta - \lfloor \beta\rfloor}\|\dot{\gamma}(t_{\lfloor \eta \rfloor-\alpha}...t_1)\|d_{t_{\lfloor \eta \rfloor-\alpha}}...\|\dot{\gamma}(t_1)\|t_1^{\lfloor \eta \rfloor-\alpha-1}dt_1.
\end{align*}
Suppose that $\|x_r-x_j\|\leq r^\star/4$ and for $D$ the map of assumption iv), define $$\gamma(s)=D(x_r+s(x_j-x_r))$$
which is well defined as for $s\in [0,1]$,
$$d(x_r+s(x_j-x_r),\mathcal{M}^{\star})\leq \|x_r+s(x_j-x_r)-x_r\| +d(x_r,\mathcal{M}^{\star}) <  r^\star/2.$$
We have $\|\dot{\gamma}(s)\|\leq C\|x_r-x_j\|$ so we deduce that

\begin{align*}
    \|\nabla^\alpha &  P_r(x_r) - \nabla^\alpha P_j(x_r)\|\\
    & \leq  \int_0^1... \int_0^1 C\|x_r-x_j\|^{\beta - \lfloor \beta\rfloor}  4C\|x_r-x_j\|d_{t_{\lfloor \eta \rfloor}}4C\|x_r-x_j\|d_{t_{\lfloor \beta \rfloor}}...4C\|x_r-x_j\|dt_\alpha\\
    & \leq C \|x_r-x_j\|^{\beta+1-\alpha}.
\end{align*}
Suppose now that $\|x_r-x_j\|\geq r^\star/4$ and define $\delta \in \mathcal{H}^{1}([0,1],\mathcal{M}^\star)$ a geodesic on $\mathcal{M}^\star$ from $x_j$ to $x_r$. Let
\begin{equation*}
\gamma(s) = \left\{
\begin{array}{ll}
  x_r + 3s(\pi(x_r)-x_r) & \text{if } s\in [0,1/3]   \medskip\\
\delta(3(t-1/3)) & \text{if } s\in (1/3,2/3)   \medskip\\ 
    \pi(x_j) + 3(s-2/3)(x_j-\pi(x_j)) & \text{if } s\in [2/3,1]
\end{array}
\right.
\end{equation*} 
be the path from $x_r$ to $x_j$ that starts by projecting $x_r$ onto $\mathcal{M}^{\star}$, then circulates along $\mathcal{M}^{\star}$ until it reaches the projection of $x_j$, and finally go to $x_j$ in a straight line.

For $s\in [0,1/3]\cup [2/3,1]$ we have 
$$\|\dot{\gamma}(s)\|\leq 3 \max(\|\pi(x_r)-x_r\|,\|\pi(x_j)-x_j\|)\leq C \leq C r^{\star-1} \|x_r-x_j\|.$$

As $\mathcal{M}^\star$ can be covered by an atlas $(\varphi_{i}^{-1}(B^d(0,\tau)))_{\{i=1,...m\}}$  with $\varphi_{i}^{-1}\in \mathcal{H}^1_K$, we deduce that 
$$\int_0^1 \|\dot{\delta}(t)\|dt\leq mK\tau.$$
Therefore, for $\delta$ with constant speed we have
$$\|\dot{\delta}(t)\|\leq mK\tau \leq C r^{\star-1} \|x_r-x_j\|.$$
We then get as in the case where $\|x_r-x_j\|\leq t/4$, that 
\begin{align*}
    \|\nabla^\alpha   P_r(x_r) - \nabla^\alpha P_j(x_r)\| \leq & C (r^{\star-1} \|x_r-x_j\|)^{\beta+1-\alpha}.
\end{align*}
Therefore using the sharp form of Whitney's extension (Theorem \ref{whitney}), $H$ can be extended to a map in $\mathcal{H}^{\eta}_{C_{M}}(\mathbb{R}^p,\mathbb{R})$.
\end{proof}

\subsection{Proof of Lemma \ref{lemma:firstterm}}\label{sec:lemma:firstterm}
Let us first recall a result from \cite{stephanovitch2023wasserstein} showing that we can gain some weak Besov regularity by paying a logarithmic term. For $1\leq p \leq \infty$, we let $L^p(\mathbb{R}^d,\mathbb{R},C)$ be the set of functions $f:\mathbb{R}^d\rightarrow \mathbb{R}$ with $p$-norm $\|f\|_{L^p}:=\left(\int \|f\|^p d\lambda^d\right)^{1/p}$ bounded by $C>0$.

\begin{proposition}\label{prop:logforweakregularity}
    Let $f\in L^1(\mathbb{R}^d,\mathbb{R},1)\cap L^2(\mathbb{R}^d,\mathbb{R})$ and $g\in \mathcal{H}^\gamma_1(\mathbb{R}^d,\mathbb{R},1)$ with $\gamma>0$. Then for all $\tau>0$, $\epsilon\in (0,1)$ we have
$$\int_{\mathbb{R}^d}f(x)g(x)d\lambda_{\mathbb{R}^d}(x)\leq C\log(\epsilon^{-1})^\tau\int_{\mathbb{R}^d}f(x)\tilde{\Gamma}^{0,-\tau}_\epsilon(g)(x)d\lambda_{\mathbb{R}^d}(x)+C\epsilon$$
for 
$$\tilde{\Gamma}^{0,-\tau}_\epsilon(g)(x)=\sum \limits_{j=0}^{\log(\epsilon^{-1})} \sum \limits_{l=1}^{2^{d}} \sum \limits_{w \in \mathbb{Z}^{d}}(1+j)^{-\tau}\alpha_g(j,l,w)\frac{\alpha_g(j,l,w)\alpha_f(j,l,w)}{|\alpha_g(j,l,w)\alpha_f(j,l,z)|}\psi_{j,l,z}(x).$$
\end{proposition}

Using Proposition \ref{prop:extensionH}, we can extend $T$ into a map belonging to $\mathcal{H}^{\beta+1}_C(\mathbb{R}^p,\mathbb{R}^p)$.
Let $h\in \mathcal{H}^1_1(\mathbb{R}^p,\mathbb{R})$, using Proposition \ref{prop:keydecomp} we have
\begin{align*}
 \int (h(x)-h(T(x))f_\mu(x)d\lambda_{\mathcal{M}}(x) & \leq  \int \|x-T(x)\|f_\mu(x)d\lambda_{\mathcal{M}}(x)\\
 & = \sum \limits_{i=1}^m \int_{\mathbb{R}^d}\frac{1}{C_i} \|\phi_i(u)-T(\phi_i(u))\|\zeta_i(u)d\lambda^d(u).
\end{align*}
Let $\xi\in (0,1)$ and for all $i\in \{1,...,m\}$, define
 $$X_i(u)=\phi_i(u)-T(\phi_i(u))$$
 and
 $L_i:\mathbb{R}^d\rightarrow \mathbb{R}^p$  by 
 \begin{align}\label{eq:L}
    L_i(u)=\left(\frac{X_i(u)}{\xi}\mathds{1}_{\{\|X_i(u)\|< \xi\}}+\frac{X_i(u)}{\|X_i(u)\|}\mathds{1}_{\{\|X_i(u)\|\geq \xi\}}\right)\mathds{1}_{\{u\in B^d(0,2\tau)\}}.
\end{align}

 We have
\begin{align*}\int_{\mathbb{R}^d} &\|\phi_i(u)-T(\phi_i(u))\|\zeta_i(u)d\lambda^d(u)\\  & =\int_{\mathbb{R}^d} \langle \frac{\phi_i(u)-T(\phi_i(u))}{\|\phi_i(u)-T(\phi_i(u))\|},\phi_i(u)-T(\phi_i(u))\rangle \zeta_i(u)d\lambda^d(u)\\
& \leq \int_{\mathbb{R}^d} \langle L_i(u),X_i(u)\rangle \zeta_i(u)d\lambda^d(u) + \xi.
\end{align*}
Let us write $K_X=\max \limits_{i\in \{1,...,m\}}\|X_i\|_{\mathcal{H}^{\beta+1}}\vee 1\leq C$ and $\delta_u^i=\|X_i(u)\|$ for $u\in B^d(0,2\tau)$.
We have that 
$L_i\in \mathcal{H}^{0}_{1}( B^d(0,2\tau),\mathbb{R}^p)$, let us fix $\eta\in (1,\beta+1]$ and show that $L_i$ has some additional Hölder regularity. 

\begin{proposition}\label{prop:HölderregL}
 For all $i\in \{1,...,m\}$ and $L_i:\mathbb{R}^d\rightarrow \mathbb{R}^p$ defined in \eqref{eq:L}, we have for all $u,v\in B^d(0,2\tau)$,
$$    \|L_i(u)-L_i(v)\|\leq 8K_X\min(\delta_u^i,\delta_v^i)^{-(1-(\eta-\lfloor \eta \rfloor ))}\|u-v\|^{1-(\eta-\lfloor \eta \rfloor)}.$$
\end{proposition}

\begin{proof}
Suppose first that $\|u-v\|\geq \min(\delta_u^i,\delta_v^i)/(4K_X)$. Then
\begin{align*}
    \|L_i(u)-L_i(v)\|&\leq 2=2\|u-v\|^{-(1-(\eta-\lfloor \eta \rfloor ))}\|u-v\|^{1-(\eta-\lfloor \eta \rfloor )}\\
    &\leq 8K_X\min(\delta_u^i,\delta_v^i)^{-(1-(\eta-\lfloor \eta \rfloor ))}\|u-v\|^{1-(\eta-\lfloor \eta \rfloor )}.
\end{align*}
Now if $\|u-v\|\leq \min(\delta_u^i,\delta_v^i)/(4K_X)$, then for $u_s=u+s(v-u)$ with $s\in [0,1]$, we have 
\begin{align*}
    \|X(u_s)\|& \geq  \|X_i(u)\|-sK_X\|u-v\|\\
    & \geq \delta_u^i- \frac{1}{4}\min(\delta_u^i,\delta_v^i)\\
    & \geq \frac{1}{2}\min(\delta_u^i,\delta_v^i).
\end{align*}
Furthermore, 
\begin{align*}
    \nabla L_i(u)= & \frac{\mathds{1}_{\{\|X_i(u)\|\geq \xi\}}}{\|X_i(u)\|}\Big(\nabla X_i(u)  - \frac{X_i(u)}{\|X_i(u)\|} \frac{(X_i(u))^\top}{\|X_i(u)\|}\nabla X_i(u) \Big)+ \frac{\mathds{1}_{\{\|X_i(u)\|<\xi\}}}{\xi}\nabla X_i(u),
\end{align*}
so $\|\nabla L_i(u)\|\leq K_X\|X_i(u)\|^{-1}= K_X(\delta_u^{i})^{-1}$.
Then
\begin{align*}
    \|L_i(u)-L_i(v)\| & =\|\int_0^1\nabla L_i(u+s(v-u))(v-u)ds\|\\
    & \leq \|v-u\|\int_0^1\|\nabla L_i(u+s(v-u))\|ds\leq 2K_X\min(\delta_u^i,\delta_v^i)^{-1}\|v-u\|\\
    & \leq 2K_X\min(\delta_u^i,\delta_v^i)^{-(1-(\eta-\lfloor \eta \rfloor ))}\|u-v\|^{1-(\eta-\lfloor \eta \rfloor)}.
\end{align*} 
\end{proof} 

We have that $\phi_{i|B^d(0,\tau)}$ the restriction of $\phi_i$ to $B^d(0,\tau)$, is a diffeomorphism on its image. For $x\in \phi_{_i}(B^d(0,\tau))$, we will write $\phi_i^{-1}(x)$ for the inverse application of $\phi_{i|B^d(0,\tau)}$.

Let us define $F_i:T^{-1}(\phi_{i}(B^d(0,\tau)))\rightarrow B^d(0,\tau)$ by
$$F_i:=\phi_i^{-1} \circ T_{|\mathcal{M}}^{-1} \circ T,$$

\begin{equation}\label{eq:A}
A:=\{x \in \mathcal{M}^{\star t} | \ \|x-T(x)\|\leq \|T_{|\mathcal{M}}^{-1}\circ T(x)-T(x)\|\}
\end{equation}
and
$H_i:A\rightarrow \mathbb{R}$ by
\begin{equation}\label{eq:H}
H_i(x):=\Big\langle \tilde{\Gamma}^{-\eta+1,-2}_\epsilon(L_i)\circ F_i(x)\ ,\ x-T(x)\Big\rangle\zeta_i(F_i(x))\mathds{1}_{\{T_{|\mathcal{M}}^{-1} \circ T(x)\in \phi_{i}(B^d(0,\tau))\}}.
\end{equation}
We have the following result on $L_i,X_i$ and $H_i$.

\begin{proposition}\label{prop:keyinterpmani} For all $i\in \{1,...,m\}$ and $\epsilon\in (0,1)$ we have
    $$\int_{\mathbb{R}^d} \langle L_i(u),X_i(u)\rangle \zeta_i(u)d\lambda^d(u)\leq C \log(\epsilon^{-1})^2\left(\int_\mathcal{M} H_i(x)d\mu(x)\right)^{\frac{\beta+1}{\beta+\eta}}+C\epsilon.$$
\end{proposition}

\begin{proof}
For $r\in \{1,...,p\}$, Let $$\tilde{\Gamma}^{0,-2}_\epsilon(L_{i_r})(u) =  \sum \limits_{j=0}^{\log(\lfloor \epsilon^{-1}\rfloor)} \sum \limits_{l=1}^{2^d} \sum \limits_{z\in \mathbb{Z}^d}(1+j)^{-2}\alpha_{L_{i_r}}(j,l,z)S(j,l,w)_{ir}\psi_{jlz}(u)$$
with $S(j,l,w)_{ir}=\frac{\alpha_{L_{i_r}}(j,l,z)\alpha_{X_{i}\zeta_i}(j,l,z)_r}{|\alpha_{L_{i_r}}(j,l,z)\alpha_{X_{i}\zeta_i}(j,l,z)_r|}$.
Then using Proposition \ref{prop:logforweakregularity}, we have
$$\int_{\mathbb{R}^d} \langle L_i(u),X_i(u)\rangle \zeta_i(u)d\lambda^d(u)\leq C\log(\epsilon^{-1})^2\int_{\mathbb{R}^d} \langle \tilde{\Gamma}^{0,-2}_\epsilon(L_i)(u),X_i(u)\rangle \zeta_i(u)d\lambda^d(u)+C\epsilon.$$

Applying Proposition \ref{prop:Hölder} for $h_1=\tilde{\Gamma}^{0,-2}_\epsilon(L_i)$, $h_2=X_i\zeta_i$, $s_1=0$, $b_1=-2$, $s_2=\beta+1$, $b_2=0$, $\tau=-\eta+1$, $t=1$, $r=0$ and $q=\frac{\beta+\eta}{\beta+1}$ we get
\begin{align*}
\int_{\mathbb{R}^d}& \langle \tilde{\Gamma}^{0,-2}_\epsilon(L_i)(u),X_i(u)\rangle \zeta_i(u)d\lambda^d(u)\\
& \leq \Big\langle \tilde{\Gamma}^{-\eta+1}(\tilde{\Gamma}^{0,-2}_\epsilon(L_i)),X_i\zeta_i\Big\rangle_{L^2}^{\frac{\beta+1}{\beta+\eta}}\Big\langle \tilde{\Gamma}^0(\tilde{\Gamma}^{0,-2}_\epsilon(L_i)),\Gamma^{\beta+1}(X_i\zeta_i)\Big\rangle_{L^2}^{\frac{\eta-1 }{\beta+\eta}}\\
& \leq \Big\langle \tilde{\Gamma}^{-\eta+1,-2}_\epsilon(L_i),X_i\zeta_i\Big\rangle_{L^2}^{\frac{\beta+1}{\beta+\eta}}\Big\langle \tilde{\Gamma}^{0,-2}_\epsilon(L_i),\Gamma^{\beta+1}(X_i\zeta_i)\Big\rangle_{L^2}^{\frac{\eta-1 }{\beta+\eta}}.
\end{align*}

We have that $\Gamma^{\beta+1}(X_i\zeta_i)\in \mathcal{B}^{0}_{\infty,\infty}(C)$ and $\tilde{\Gamma}^{0,-2}_\epsilon(L_i)\in \mathcal{B}^{0,2}_{\infty,\infty}(C)$ so
\begin{align*}
 \sum \limits_{r=1}^p &\sum \limits_{j=0}^\infty \sum \limits_{l=1}^{2^d} \sum \limits_{z\in \mathbb{Z}^d}\alpha_{\tilde{\Gamma}^{0,-2}_\epsilon(L_i)}(j,l,z)_r \alpha_{\Gamma^{\beta+1}(X_i\zeta_i)}(j,l,z)_r\\
 & \leq   \sum \limits_{r=1}^p \sum \limits_{j=0}^\infty \sum \limits_{l=1}^{2^d} \sum \limits_{z\in \mathbb{Z}^d}(1+j)^{-2}C2^{-jd}\mathds{1}_{\{z\in B^d(0,C)\}}\\
 & \leq C  \sum \limits_{r=1}^p \sum \limits_{j=0}^\infty (1+j)^{-2}\leq C.
\end{align*}
Therefore, we have 
$$\Big\langle \tilde{\Gamma}^{0,-2}_\epsilon(L_i),\Gamma^{\beta+1}(X_i\zeta_i)\Big\rangle_{L^2}^{\frac{\eta-1 }{\beta+\eta}}\leq C.$$

Furthermore, recalling the proof of Proposition \ref{prop:keydecomp}, we have that $(\phi_i)_{\# \lambda^d_{|B^d(0,\tau)}}=\mu_{|\phi_i(B^d(0,\tau))}$ so
\begin{align*}
\Big\langle \tilde{\Gamma}^{-\eta+1,-2}_\epsilon(L_i),X_i\zeta_i\Big\rangle_{L^2}& = \int_{\mathbb{R}^d} \langle \tilde{\Gamma}^{-\eta+1,-2}_\epsilon(L_i)(u),X_i(u)\rangle \zeta_i(u)d\lambda^d(u)\\
& = \int_{\mathbb{R}^d} \langle \tilde{\Gamma}^{-\eta+1,-2}_\epsilon(L_i)(\phi_i^{-1}(x)),x-T(x)\rangle \zeta_i(\phi_i^{-1}(x))d(\phi_i)_{\# \lambda^d_{|B^d(0,\tau)}}(x)\\
&  = \int_\mathcal{M} H_i(x)d\mu(x).
\end{align*}
We can then conclude that
    $$\int_{\mathbb{R}^d} \langle L_i(u),X_i(u)\rangle \zeta_i(u)d\lambda^d(u)\leq C \log(\epsilon^{-1})^2\left(\int_\mathcal{M} H_i(x)d\mu(x)\right)^{\frac{\beta+1}{\beta+\eta}}+C\epsilon.$$  
\end{proof}

Let us now show that the maps $H_i$ from \eqref{eq:H} are of regularity $\eta$. To that end, let us first show that $\nabla^{\lfloor \eta\rfloor}\tilde{\Gamma}^{-\eta+1,-2}_\epsilon(L_i)$ exists and has some Hölder regularity.

\begin{lemma}\label{lemma:le28}
    For all $i\in \{1,...,m\}$, $i_1,...,i_{\lfloor \eta \rfloor} \in \{1,...,d\}$ and $u\in B^d(0,\tau)$,  we have 
    $$\| \partial_{i_1,...,i_{\lfloor \eta \rfloor}}^{\lfloor \eta \rfloor}  \tilde{\Gamma}_\epsilon^{-\eta+1,-2}(L_i)(u)\|\leq CK_X(\delta_u^i)^{-(\lfloor \eta \rfloor+1-\eta)},$$
    (with $\delta_u^i=\|X_i(u)\|$) and for $u,v\in B^d(0,2\tau)$ we have
    \begin{align*}
\|\partial_{i_1,...,i_{\lfloor \eta \rfloor}}^{\lfloor \eta \rfloor} &\tilde{\Gamma}^{-\eta+1,-2}_\epsilon(L_i)(u)-\partial_{i_1,...,i_{\lfloor \eta \rfloor}}^{\lfloor \eta \rfloor} \tilde{\Gamma}^{-\eta+1,-2}_\epsilon(L_i)(v)\| \leq CK_X(\delta_u^{i} \wedge \delta_v^i)^{-1}\|u-v\|^{\eta-\lfloor \eta \rfloor}.
\end{align*}
\end{lemma}

\begin{proof}
 For $u\in B^d(0,\tau)$, we shall distinguish whether we have $\delta_u^i>\xi$ or not. 

Let us first consider the case $\delta_u^i>\xi$. Define $L_{ui}:B^d(u,\frac{\delta_u^i}{4K_X})\rightarrow \mathbb{R}^p$ being equal to $L_i$. We then extend $L_{ui}$ to $\mathbb{R}^d$ by 
\begin{equation}\label{eq:extensionprojec}
    \overline{L}_{ui}(v)=L_{ui}(\pi_{B_u}(v))(0\vee(1-3\|v-\pi_{B_u}(v)\|)),
\end{equation} 
for $\pi_{B_u}$ the projection on $B^d(u,\frac{\delta_u^i}{4K_X})$. For all $v\in B^d(u,\frac{\delta_u^i}{4K_X})$, we have $\delta_v^i\geq \delta_u^i/2$ so from Proposition \ref{prop:HölderregL} we get
$$
L_{ui}\in \mathcal{H}^{0}_{1}(B^d(u,\frac{\delta_u^i}{4K_X}),\mathbb{R}^p)\cap \mathcal{H}^{1-(\eta-\lfloor \eta \rfloor )}_{16K_X\left(\delta_u^{i}\right)^{-(1-(\eta-\lfloor \eta \rfloor ))}}(B^d(u,\frac{\delta_u^i}{4K_X}),\mathbb{R}^p).
$$
Therefore, we also have that
$$
\overline{L}_{ui} \in \mathcal{H}^{0}_{1}(\mathbb{R}^d,\mathbb{R}^p)\cap \mathcal{H}^{1-(\eta-\lfloor \eta \rfloor )}_{16K_X\left(\delta_u^{i}\right)^{-(1-(\eta-\lfloor \eta \rfloor ))}}(\mathbb{R}^d,\mathbb{R}^p).
$$
For $r\in \{1,...,p\}$, let us write $(\alpha_{\overline{L}_{ui}}(j,l,z)_r)_{(j,l,z)}$ the wavelet coefficients of $(\overline{L}_{ui})_r$. As $\text{support}(\psi_{jlz})\subset B^d(2^{-j}z,C2^{-j})$, we have that for any $(j,l,z)$ such that $j\geq \lfloor\log_2(CK_X(\delta_u^{i})^{-1})\rfloor+1$ and $\text{support}(\psi_{jlz})\cap B^d(u,\frac{\delta_u^i}{8K_X})\neq \varnothing$ then $\alpha_{\overline{L}_{ui}}(j,l,z)_i=\alpha_{L}(j,l,z)_i$. Let us note $$\xi_u^i=\lfloor\log_2(CK_X(\delta_u^{i})^{-1})\rfloor+1$$ and
$$
\tilde{\Gamma}^{-\eta+1,-2}_\epsilon(\overline{L}_{ui})_r(v)= \sum \limits_{j=0}^{\log_2(\epsilon^{-1})} \sum \limits_{l=1}^{2^d} \sum \limits_{z\in \mathbb{Z}^d}2^{-j(\eta-1)} (1+j)^{-2}S(j,l,z)_{ir}\alpha_{\overline{L}_{ui}}(j,l,z)_r\psi_{j,l,z}(v),
$$
the $(\eta-1,2)$ wavelet regularization of $\overline{L}_{u_i}$. In particular we have $$\tilde{\Gamma}^{-\eta+1,-2}_\epsilon(\overline{L}_{ui}) \in \mathcal{H}^{\eta-1}_{C}(\mathbb{R}^d,\mathbb{R}^p)\cap \mathcal{H}^{\lfloor \eta \rfloor}_{CK_X\left(\delta_u^{i}\right)^{-(1-(\eta-\lfloor \eta \rfloor ))}}(\mathbb{R}^d,\mathbb{R}^p).$$

Let $i_1,...,i_{\lfloor \eta \rfloor} \in \{1,...,d\}$, we have
\begin{align}\label{align:split}
    |  &\partial_{i_1,...,i_{\lfloor \eta \rfloor}}^{\lfloor \eta \rfloor}  \tilde{\Gamma}_\epsilon^{-\eta+1,-2}(L_i)_r(u)| \nonumber\\
    = &|\partial_{i_1,...,i_{\lfloor \eta \rfloor}}^{\lfloor \eta \rfloor}\sum \limits_{j=0}^{\log_2(\epsilon^{-1})} \sum \limits_{l=1}^{2^d} \sum \limits_{z\in \mathbb{Z}^d} \alpha_{\tilde{\Gamma}^{-\eta+1,-2}(L_i)}(j,l,z)_r  \psi_{j,l,z}(u)|\nonumber\\
     \leq & |\partial_{i_1,...,i_{\lfloor \eta \rfloor}}^{\lfloor \eta \rfloor}\sum \limits_{j=0}^{\xi_u^i\wedge \log_2(\epsilon^{-1})} \sum \limits_{l=1}^{2^d} \sum \limits_{z\in \mathbb{Z}^d}(\alpha_{\tilde{\Gamma}^{-\eta+1,-2}(L_i)}(j,l,z)_r-\alpha_{\tilde{\Gamma}^{-\eta+1,-2}(\overline{L}_{ui})}(j,l,z)_r)  \psi_{j,l,z}(u)|\nonumber\\
    & +|\partial_{i_1,...,i_{\lfloor \eta \rfloor}}^{\lfloor \eta \rfloor}\sum \limits_{j=0}^{\xi_u^i\wedge \log_2(\epsilon^{-1})} \sum \limits_{l=1}^{2^d} \sum \limits_{z\in \mathbb{Z}^d}\alpha_{\tilde{\Gamma}^{-\eta+1,-2}(\overline{L}_{ui})}(j,l,z)_r  \psi_{j,l,z}(u)\nonumber \\
    & +\partial_{i_1,...,i_{\lfloor \eta \rfloor}}^{\lfloor \eta \rfloor} \sum \limits_{j=\xi_u^i\wedge \log_2(\epsilon^{-1})+1}^{\log_2(\epsilon^{-1})} \sum \limits_{l=1}^{2^d} \sum \limits_{z\in \mathbb{Z}^d} \alpha_{\tilde{\Gamma}^{-\eta+1,-2}(L_i)}(j,l,z)_r  \psi_{j,l,z}(u)|.
\end{align}

For $k\in \{1,...,d\}$, let us write $\theta_k=\sum \limits_{r=1}^{\lfloor \eta \rfloor}\mathds{1}_{\{i_r=k\}}$. For the first term of \eqref{align:split} we have 
\begin{align*}
    |\partial_{i_1,...,i_{\lfloor \eta \rfloor}}^{\lfloor \eta \rfloor} & \sum \limits_{j=0}^{\xi_u^i\wedge \log_2(\epsilon^{-1})} \sum \limits_{l=1}^{2^d} \sum \limits_{z\in \mathbb{Z}^d}(\alpha_{\tilde{\Gamma}^{-\eta+1,-2}(L_i)}(j,l,z)_r-\alpha_{\tilde{\Gamma}^{-\eta+1,-2}(\overline{L}_{ui})}(j,l,z)_r)  \psi_{j,l,z}(u)|\\
     \leq & \sum \limits_{j=0}^{\xi_u^i\wedge \log_2(\epsilon^{-1})} \sum \limits_{l=1}^{2^d} \sum \limits_{z\in \mathbb{Z}^d}|(\alpha_{\tilde{\Gamma}^{-\eta+1,-2}(L_i)}(j,l,z)_r-\alpha_{\tilde{\Gamma}^{-\eta+1,-2}(\overline{L}_{ui})}(j,l,z)_r)\\
     & \qquad \times 2^{j(\lfloor \eta \rfloor+d/2)} \prod \limits_{k=0}^d \partial^{\theta_k} \psi_{l_k}(2^{j}u_k-z_k)|\\
    \leq & C \sum \limits_{j=0}^{\xi_u^i\wedge \log_2(\epsilon^{-1})} \sum \limits_{z\in \mathbb{Z}^d}2^{j(\lfloor \eta \rfloor+1-\eta)}|\prod \limits_{k=0}^d \partial^{\theta_k} \psi_{l_k}(2^{j}u_k-z_k)|\\
     \leq & C \sum \limits_{j=0}^{\xi_u^i\wedge \log_2(\epsilon^{-1})} \sum \limits_{z\in \mathbb{Z}^d}2^{j(\lfloor \eta \rfloor+1-\eta)}C\mathds{1}_{\{u\in supp(\psi_{jlz})\}}\\
      \leq & C \sum \limits_{j=0}^{\xi_u^i\wedge \log_2(\epsilon^{-1})} 2^{j(\lfloor \eta \rfloor+1-\eta)}\leq C2^{\xi_u^i(\lfloor \eta \rfloor+1-\eta)}\leq CK_X(\delta_u^i)^{-(\lfloor \eta \rfloor+1-\eta)}.
\end{align*}

If $\xi_u^i<\log_2(\epsilon^{-1})$, for the second term we have
    \begin{align*}
    |\partial_{i_1,...,i_{\lfloor \eta \rfloor}}^{\lfloor \eta \rfloor}\sum \limits_{j=0}^{\xi_u^i} & \sum \limits_{l=1}^{2^d} \sum \limits_{z\in \mathbb{Z}^d}\alpha_{\tilde{\Gamma}^{-\eta+1,-2}(\overline{L}_{ui})}(j,l,z)_r  \psi_{j,l,z}(u)\\
    & +\partial_{i_1,...,i_{\lfloor \eta \rfloor}}^{\lfloor \eta \rfloor} \sum \limits_{j=\xi_u^i+1}^{\log_2(\epsilon^{-1})} \sum \limits_{l=1}^{2^d} \sum \limits_{z\in \mathbb{Z}^d} \alpha_{\tilde{\Gamma}^{-\eta+1,-2}(L_i)}(j,l,z)_r  \psi_{j,l,z}(u)|\\
     = &|\partial_{i_1,...,i_{\lfloor \eta \rfloor}}^{\lfloor \eta \rfloor}\sum \limits_{j=0}^{\log_2(\epsilon^{-1})} \sum \limits_{l=1}^{2^d} \sum \limits_{z\in \mathbb{Z}^d}\alpha_{\tilde{\Gamma}^{-\eta+1,-2}(\overline{L}_{ui})}(j,l,z)_r  \psi_{j,l,z}(u)|\\
     \leq &\|\partial_{i_1,...,i_{\lfloor \eta \rfloor}}^{\lfloor \eta \rfloor}\tilde{\Gamma}^{-\eta+1,-2}(\overline{L}_{ui})_r\|_\infty \leq CK_X(\delta_u^i)^{-(\lfloor \eta \rfloor+1-\eta)}.
\end{align*}
On the other hand, if $\xi_u^i\geq \log_2(\epsilon^{-1})$ we have that the second term is equal to
    \begin{align*}
&|\partial_{i_1,...,i_{\lfloor \eta \rfloor}}^{\lfloor \eta \rfloor}\sum \limits_{j=0}^{\log_2(\epsilon^{-1})} \sum \limits_{l=1}^{2^d} \sum \limits_{z\in \mathbb{Z}^d}\alpha_{\tilde{\Gamma}^{-\eta+1,-2}(\overline{L}_{ui})}(j,l,z)_r  \psi_{j,l,z}(u)|\\
     \leq &\|\partial_{i_1,...,i_{\lfloor \eta \rfloor}}^{\lfloor \eta \rfloor}\tilde{\Gamma}^{-\eta+1,-2}(\overline{L}_{ui})_r\|_\infty \leq CK_X(\delta_u^i)^{-(\lfloor \eta \rfloor+1-\eta)}.
\end{align*}

Now that we have shown that
\begin{equation}\label{eq:lafinestproche}
| \partial_{i_1,...,i_{\lfloor \eta \rfloor}}^{\lfloor \eta \rfloor}  \tilde{\Gamma}_\epsilon^{-\eta+1,-2}(L_i)_r(u)|\leq CK_X(\delta_u^i)^{-(\lfloor \eta \rfloor+1-\eta)},
    \end{equation}
for $u\in B^d(0,\tau)$ satisfying $\delta_u^i>\xi$, in the case $\delta_u^i\leq \xi$ we have $L_i(u)=X_i(u)$ so it can be treated the same way. Defining $L_{ui}$ coinciding with $L_i$ on $B^d(u,\frac{\xi}{4K_X})$ and doing the same derivations, we also obtain $|\partial_{i_1,...,i_{\lfloor \eta \rfloor}}^{\lfloor \eta \rfloor}\tilde{\Gamma}_\epsilon^{-\eta+1,-2}(L_i)(u)_r|\leq CK_X(\delta_u^i)^{-(\lfloor \eta \rfloor+1-\eta)}.$

Let us now show that for $u,v\in B^d(0,2\tau)$ we have
\begin{align}
\label{Hölder}
|\partial_{i_1,...,i_{\lfloor \eta \rfloor}}^{\lfloor \eta \rfloor} &\tilde{\Gamma}^{-\eta+1,-2}_\epsilon(L_i)_r(u)-\partial_{i_1,...,i_{\lfloor \eta \rfloor}}^{\lfloor \eta \rfloor} \tilde{\Gamma}^{-\eta+1,-2}_\epsilon(L_i)_r(v)|\nonumber \\
& \leq CK_X((\delta_u^{i})^{-1} \vee (\delta_v^i)^{-1})\|u-v\|^{\eta-\lfloor \eta \rfloor}.
\end{align}
If $\|u-v\|< \frac{\delta_u^i\wedge \delta_v^i}{8K_X}$, define $L_{i(u,v)}$ coinciding with $L$ on $B^d(\argmax \limits_{w\in \{u,v\}} \delta_w,\frac{ \delta_u^i\vee \delta_v^i}{8K_X})$ and extend it into $\overline{L}_{i(u,v)}:B^d(0,\tau)\rightarrow \mathbb{R}^p$ by projection as we did with $L_{ui}$ in \eqref{eq:extensionprojec}.
Otherwise, if $\|u-v\|\geq \frac{\delta_u^i\wedge \delta_v^i}{8K_X}$, define $L_{i(u,v)}$ coinciding with $L$ on $B^d(u,\frac{\delta_u^i\wedge \delta_v^i}{32K_X})\cup B^d(v,\frac{\delta_u^i\wedge \delta_v^i}{32K_X})$.
Then, we have $L_{i(u,v)}\in \mathcal{H}^1_{CK_X(\delta_u^{i})^{-1}\vee (\delta_v^i)^{-1}}$ and we can extend it by projection on the balls as we did with $L_{ui}$ in \eqref{eq:extensionprojec}:
\begin{align*}
    \overline{L}_{i(u,v)}(w)=&L_{i(u,v)}(\pi_{B_u}(w))(0\vee(1-32K_X(\delta_u^{i})^{-1}\|w-\pi_{B_u}(w)\|))\\
    & +L_{i(u,v)}(\pi_{B_v}(w))(0\vee(1-32K_X(\delta_v^i)^{-1}\|w-\pi_{B_v}(w)\|)).
\end{align*}
We also have $\overline{L}_{i(u,v)}\in \mathcal{H}^1_{C(\delta_u^{i})^{-1}\vee (\delta_v^i)^{-1}}$. Then by splitting the wavelets coefficients of $\partial_{i_1,...,i_{\lfloor \eta \rfloor}}^{\lfloor \eta \rfloor} \tilde{\Gamma}^{-\eta+1,-2}_\epsilon(L_i)_r(u)-\partial_{i_1,...,i_{\lfloor \eta \rfloor}}^{\lfloor \eta \rfloor} \tilde{\Gamma}^{-\eta+1,-2}_\epsilon(L_i)_r(v)$ as in \eqref{align:split}, we obtain \eqref{Hölder}. 
\end{proof}

Using Lemma \ref{lemma:le28}, we obtain that the maps $H_i$ from \eqref{eq:H} are of regularity $\eta$.
\begin{proposition}\label{prop:keyinterpmani2} For all $i\in \{1,...,m\}$, the application $H_i$ belongs to $\mathcal{H}^{\eta}_{C}(A,\mathbb{R})$.
\end{proposition}

\begin{proof}
Define
$$
F_i:=\phi_i^{-1} \circ T_{|\mathcal{M}}^{-1} \circ T\ \text{ and }\ \sigma_i:=\mathds{1}_{\{T_{|\mathcal{M}}^{-1} \circ T(\cdot)\in \phi_{i|B^d(0,\tau)}\}} \zeta_i\circ F_i.
$$
We have $F_i\in \mathcal{H}^{\beta+1}_{C}(A,\mathbb{R}^d)$ and as $supp(\zeta_i)\subset B^d(0,\tau)$, the application  $\sigma_i$ belongs to $\mathcal{H}^{\beta+1}_{C}(A,\mathbb{R}^p)$. Furthermore, 
 we have $(\text{Id}-T) \in \mathcal{H}^{\beta+1}_{C}(A,\mathbb{R}^p)$ and $\tilde{\Gamma}^{-\eta+1,-2}_\epsilon(L_i) \in \mathcal{H}_{C}^{\eta-1}(B^d(0,\tau),\mathbb{R}^p)$, so we get from the Faa di Bruno formula that  $H_i\in \mathcal{H}^{\eta-1}_{C}(A,\mathbb{R})$. Let $i_1,...,i_{\lfloor \eta \rfloor} \in \{1,...,p\}$ and write $P(\{i_1,...,i_{\lfloor \eta \rfloor}\})$ the set of subsets of $\{i_1,...,i_{\lfloor \eta \rfloor}\}$. For $x\in A$, we have

\begin{align*}
    \partial_{i_1,...,i_{\lfloor \eta \rfloor}}^{\lfloor \eta \rfloor} H_i(x) & = \sum \limits_{S\in P(\{i_1,...,i_{\lfloor \eta \rfloor}\})}\Big\langle \partial^{|S|}_{i\in S} (\tilde{\Gamma}^{-\eta+1,-2}_\epsilon(L_i)\circ F_i)(x)\ ,\ \partial^{\lfloor \eta \rfloor-|S|}_{i\notin S}((\text{Id}-T)\sigma_i)(x)\Big\rangle.
\end{align*}
For all $S\in P(\{i_1,...,i_{\lfloor \eta \rfloor}\})$ with $|S|<\lfloor \eta \rfloor$, we have that $$x\mapsto \Big\langle \partial^{|S|}_{i\in S} (\tilde{\Gamma}^{-\eta+1,-2}_\epsilon(L_i)\circ F_i)(x)\ ,\ \partial^{\lfloor \eta \rfloor-|S|}_{i\notin S}(\text{Id}-T)\sigma_i(x)\Big\rangle \in \mathcal{H}^{\eta-\lfloor \eta \rfloor}_{C}(A,\mathbb{R}). 
$$
From the previous calculations, we know that 
\begin{align*}
|  \partial_{i_1,...,i_{\lfloor \eta \rfloor}}^{\lfloor \eta \rfloor} \tilde{\Gamma}^{-\eta+1,-2}_\epsilon(L_i)(F_i(x))| & \leq CK_X\delta_{F_i(x)}^{i-(\lfloor \eta \rfloor+1-\eta)}\\
 & =CK_X\| T_{|\mathcal{M}}^{-1} \circ T(x)-T( T_{|\mathcal{M}}^{-1} \circ T(x))\|^{-(\lfloor \eta \rfloor+1-\eta)}\\
 & = CK_X\| T_{|\mathcal{M}}^{-1} \circ T(x)-T(x)\|^{-(\lfloor \eta \rfloor+1-\eta)}\\
  & \leq  C\|x-T(x)\|^{-(\lfloor \eta \rfloor+1-\eta)},
\end{align*}
so 
\begin{equation}\label{eq:disttotheimage}
|\Big\langle\partial_{i_1,...,i_{\lfloor \eta \rfloor}}^{\lfloor \eta \rfloor} (\tilde{\Gamma}^{-\eta+1,-2}_\epsilon(L_i)\circ F_i)(x)\ ,(\text{Id}-T)\sigma_i(x)\Big\rangle|\leq C\|x-T(x)\|^{\eta-\lfloor \eta \rfloor}\sigma_i(x).
\end{equation}    

Let $x,y\in A$ with $T_{|\mathcal{M}}^{-1} \circ T(x),T_{|\mathcal{M}}^{-1} \circ T(y)\in \phi_{i}(B^d(0,\tau))$ and $\|x-T(x)\|\leq \|y-T(y)\|$, from \eqref{Hölder} we get
\begin{align*}
    |\Big\langle&\partial_{i_1,...,i_{\lfloor \eta \rfloor}}^{\lfloor \eta \rfloor}  (\tilde{\Gamma}^{-\eta+1,-2}_\epsilon(L_i)  \circ  F_i)(x)\ ,(\text{Id}-T)\sigma_i(x)\Big\rangle\\
    & -\Big\langle\partial_{i_1,...,i_{\lfloor \eta \rfloor}}^{\lfloor \eta \rfloor} (\tilde{\Gamma}^{-\eta+1,-2}_\epsilon(L_i)\circ F_i)(y)\ ,(\text{Id}-T)\sigma_i(y)\Big\rangle|\\
     \leq  &  |\Big\langle\partial_{i_1,...,i_{\lfloor \eta \rfloor}}^{\lfloor \eta \rfloor} (\tilde{\Gamma}^{-\eta+1,-2}_\epsilon(L_i)  \circ F_i)(x)-\partial_{i_1,...,i_{\lfloor \eta \rfloor}}^{\lfloor \eta \rfloor} (\tilde{\Gamma}^{-\eta+1,-2}_\epsilon(L_i)\circ F_i)(y)\ ,(\text{Id}-T)\sigma_i(x)\Big\rangle|\\
    & + |\Big\langle\partial_{i_1,...,i_{\lfloor \eta \rfloor}}^{\lfloor \eta \rfloor} (\tilde{\Gamma}^{-\eta+1,-2}_\epsilon(L_i)\circ F_i)(y)\ , (\text{Id}-T)\sigma_i(x)-(\text{Id}-T)\sigma_i(y)\Big\rangle|\\
     \leq & C\min(\|x-T(x)\|,\|y-T(y)\|)^{-1}\|x-y\|^{\eta-\lfloor \eta \rfloor}\|x-T(x)\| \\
    & +C\|T_{|\mathcal{M}}^{-1} \circ T(y)-T(y)\|^{-(1-(\eta-\lfloor \eta \rfloor))}\|(\text{Id}-T)\sigma_i(x)-(\text{Id}-T)\sigma_i(y)\|\\
     \leq & C\|x-y\|^{\eta-\lfloor \eta \rfloor}+C\|T_{|\mathcal{M}}^{-1} \circ T(y)-T(y)\|^{-(1-(\eta-\lfloor \eta \rfloor))}\|(\text{Id}-T)\sigma_i(x)-(\text{Id}-T)\sigma_i(y)\|.
\end{align*}
Furthermore,
\begin{align*}
    \|&(\text{Id}-T)\sigma_i(x)-(\text{Id}-T)\sigma_i(y)\| \\
    &=\|(\text{Id}-T)\sigma_i(x)-(\text{Id}-T)\sigma_i(y)\|^{1-(\eta-\lfloor \eta \rfloor)}\|(\text{Id}-T)\sigma_i(x)-(\text{Id}-T)\sigma_i(y)\|^{\eta-\lfloor \eta \rfloor}\\
    & \leq (\|(\text{Id}-T)(x)\| +\|(\text{Id}-T)(y)\|)^{1-(\eta-\lfloor \eta \rfloor)}\|x-y\|^{\eta-\lfloor \eta \rfloor}\\
    & \leq 2\|(\text{Id}-T)(y)\|^{1-(\eta-\lfloor \eta \rfloor)}\|x-y\|^{\eta-\lfloor \eta \rfloor}\\
    & \leq 2\|T_{|\mathcal{M}}^{-1} \circ T(y)-T(y)\|^{1-(\eta-\lfloor \eta \rfloor)}\|x-y\|^{\eta-\lfloor \eta \rfloor}.
\end{align*}
Now if $T_{|\mathcal{M}}^{-1} \circ T(x)\in \phi_{i}(B^d(0,\tau))$ and $T_{|\mathcal{M}}^{-1} \circ T(y)\notin \phi_{i}(B^d(0,\tau))$, we have from \eqref{eq:disttotheimage},
\begin{align*}
    |\Big\langle\partial_{i_1,...,i_{\lfloor \eta \rfloor}}^{\lfloor \eta \rfloor} (\tilde{\Gamma}^{-\eta+1,-2}_\epsilon(L_i)\circ F_i)(x)\ ,(\text{Id}-T)\sigma_i(x)\Big\rangle| & \leq C\|x-T(x)\|^{\eta-\lfloor \eta \rfloor}\sigma_i(x)\\
    &\leq C\sigma_i(x)\\
    & = C(\sigma_i(x)-\sigma_i(y))\\
    & \leq C\|x-y\|.
\end{align*}
Therefore we have finally that $H_i\in \mathcal{H}^{\eta}_{C}(A,\mathbb{R})$.
\end{proof}
Let us now prove that we can extend $H$ to the whole $\mathbb{R}^p$ space using our extension result for maps defined on the neighborhood of a submanifold (Proposition \ref{prop:extensionH}). 
 Define $D:\mathcal{M}^{\star t}\rightarrow A$ by 
\begin{equation*}
D(x) = \left\{
\begin{array}{ll}
  x & \text{if } x \in A  \medskip\\
    \frac{x-T(x)}{\|x-T(x)\|}\|T_{|\mathcal{M}}^{-1} \circ T(x)-T(x)\|+T(x)& \text{if } x \notin A.
\end{array}
\right.
\end{equation*} 
For $x\notin A$ we have 
\begin{align*}
    \nabla D(x) = &\frac{\|T_{|\mathcal{M}}^{-1} \circ T(x)-T(x)\|}{\|x-T(x)\|}\left(\nabla (\text{Id}-T)(x)-\frac{x-T(x)}{\|x-T(x)\|} \frac{(x-T(x))^\top}{\|x-T(x)\|} \nabla(\text{Id}-T)(x)\right)\\
    & + \frac{x-T(x)}{\|x-T(x)\|} \frac{(x-T(x))^\top}{\|x-T(x)\|} \nabla(T_{|\mathcal{M}}^{-1}\circ T -T)(x)+\nabla T(x),
\end{align*}
so $\|\nabla D(x)\|\leq C$. Then using Proposition \ref{prop:extensionH}, for all $i\in \{1,...,m\}$, $H_i$ can be extended to a map in $\mathcal{H}^{\eta}_{C}(\mathbb{R}^p,\mathbb{R})$.

We can now conclude the proof of Lemma \ref{lemma:firstterm} using Proposition \ref{prop:keyinterpmani} as 
\begin{align*}
\int_\mathcal{M} H_i(x)d\mu(x)
    \leq C \sup \limits_{f \in \mathcal{H}^{\eta}_1,\ f_{|\mathcal{M}^\star}=0}\int_\mathcal{M} f(x)d\mu(x).
\end{align*}

\subsection{Proof of Lemma \ref {lemma:secondterm}}\label{sec:lemma:secondterm}

\begin{proof}
Let $T$ be the  $(\mathcal{M},\mathcal{M}^\star)$-compatible map given by Theorem \ref{theo:existencecompmap}. As \\$T_{|_{\mathcal{M}}}^{-1}\in \mathcal{H}^{\beta+1}_{C} (\mathcal{M}^\star,\mathcal{M}),$ we have that $T_{\# \mu}$ admits a density $f_T\in \mathcal{H}^{\beta}_{C}(\mathcal{M}^\star,\mathbb{R})$ with respect to the volume measure on $\mathcal{M}^\star$. For $t\geq C^{-1}$ the radius of compatibility of the map $T$ (see Definition \ref{defi:compatibility}) define the $t$-envelope of $f_T$ and $f_{\mu^\star}$ as: 
\begin{equation*}
f_T^{t}(x) = 
  \frac{f_T(T(x))\Theta(4\|x-T(x)\|^2/t^2)|\text{ap}_d(\nabla T(x))|}{\int_{\mathbb{R}^{p-d}} \Theta(4\|y\|^2/t^2)d\lambda^{p-d}(y)} \end{equation*}
  and
  \begin{equation*}
  f_{\mu^\star}^{t}(x) = 
  \frac{f_{\mu^\star}(T(x))\Theta(4\|x-T(x)\|^2/t^2)|\text{ap}_d(\nabla T(x))|}{\int_{\mathbb{R}^{p-d}} \Theta(4\|y\|^2/t^2)d\lambda^{p-d}(y)}
\end{equation*} 
for $\text{ap}_d$ the approximate jacobian defined in Definition \ref{def:approx} and $\Theta\in\mathcal{H}_C^{\beta}(\mathbb{R},\mathbb{R}_+)$ such that $\Theta(x)=\Theta(-x)$, $\Theta(0)=1$ and  $\Theta_{|(1,\infty)}=0$.

As $\nabla T \in \mathcal{H}_K^\beta(\mathbb{R}^p,L(\mathbb{R}^p,\mathbb{R}^p))$ and $\text{ap}_d(\nabla T(x))$ is bounded below by $K_T^{-1}$ ($T$ verifes point (iv) of Definition \ref{defi:compatibility}), then from the Faa di Bruno formula, we deduce that $\text{ap}_d(\nabla T)$ belongs to $\mathcal{H}_{C}^\beta(\mathcal{M}^{\star t},\mathbb{R})$. Let $h\in \mathcal{H}_K^\gamma(\mathbb{R}^p,\mathbb{R})$ and define $\overline{h}:\mathcal{M}^{\star t}\rightarrow \mathbb{R} $ by 
$$
\overline{h}(x)=h(T(x))\kappa(2\|x-T(x)\|^2/t^2),
$$
for $\kappa\in \mathcal{H}^{\gamma}_C(\mathbb{R},\mathbb{R})$ such that $\kappa_{|(-\infty,1/2)}=1$ and $\kappa_{|(1,\infty)}=0$. Then $\overline{h} \in \mathcal{H}^{\gamma}_{C}(\mathbb{R}^p,\mathbb{R})$ and from Proposition \ref{approx} we have,

\begin{align*}
 \int_{\mathcal{M}}& h(T(x))d\mu(x)-\int_{\mathcal{M}^\star} h(x)d\mu^\star(x) = \int_{\mathcal{M}^{\star}} h(x)(f_T(x)-f_{\mu^\star}(x))d\lambda_{\mathcal{M}^{\star}}(x)\\
    & = \int_{\mathcal{M}^{\star}} \frac{h(x)(f_T(x)-f_{\mu^\star}(x))}{\int_{\mathbb{R}^{p-d}} \Theta(4\|y\|^2/t^2)d\lambda^{p-d}(y)}\int_{T^{-1}(\{x \})}\Theta(4\|z-T(z)\|^2/t^2)d\lambda^{p-d}(z)d\lambda_{\mathcal{M}^{\star}}(x)\\
        & = \int_{\mathcal{M}^{\star}} \int_{T^{-1}(\{x \})}\frac{h(T(z))(f_T(T(z))-f_{\mu^\star}(T(z)))\Theta(4\|z-T(z)\|^2/t^2)}{\int_{\mathbb{R}^{p-d}} \Theta(4\|y\|^2/t^2)d\lambda^{p-d}(y)}d\lambda^{p-d}(z)d\lambda_{\mathcal{M}^{\star}}(x)\\
    & = \int_{\mathbb{R}^p} h(T(x))(f_T^{t}(T(x))-f_{\mu^\star}^{t}(T(x)))d\lambda^p(x)\\
    & = \int_{\mathbb{R}^p} \overline{h}(x)(f_T^{t}(x)-f_{\mu^\star}^{t}(x))d\lambda^p(x).
\end{align*}

Let $$\tilde{\Gamma}^{0,-2}_\epsilon(\overline{h})(x) =  \sum \limits_{j=0}^{\log(\lfloor \epsilon^{-1}\rfloor)} \sum \limits_{l=1}^{2^d} \sum \limits_{w\in \mathbb{Z}^p}(1+j)^{-2}S(j,l,w)\alpha_{\overline{h}}(j,l,w)\psi_{jlw}(x)$$
for $S(j,l,w)=\frac{\alpha_{\overline{h}}(j,l,w)(\alpha_{f_T^{t}}(j,l,w)-\alpha_{f_{\mu^\star}^{t}}(j,l,w))}{\alpha_{\overline{h}}(j,l,w)(\alpha_{f_T^{t}}(j,l,w)-\alpha_{f_{\mu^\star}^{t}}(j,l,w))|}$ . Then using Proposition \ref{prop:logforweakregularity}, we have
$$\int_{\mathbb{R}^p} \overline{h}(x)(f_T^{t}(x)-f_{\mu^\star}^{t}(x))d\lambda^p(x)\leq C\log(\epsilon^{-1})^2\int_{\mathbb{R}^p} \tilde{\Gamma}^{0,-2}_\epsilon(\overline{h})(f_T^{t}(x)-f_{\mu^\star}^{t}(x))d\lambda^p(x)+C\epsilon.$$

Applying Proposition \ref{prop:Hölder} for $h_1=\tilde{\Gamma}^{0,-2}_\epsilon(\overline{h})$, $h_2=f_T^{t}-f_{\mu^\star}^{t}$, $s_1=\gamma$, $b_1=2$, $s_2=\beta$, $b_2=0$, $\tau=\gamma-\eta$, $t=1$, $r=\frac{\gamma}{\beta+\gamma}$ and $q=\frac{\beta+\eta}{\beta+\gamma}$ we get
\begin{align*}
\int_{\mathbb{R}^p} &\overline{h}^\epsilon(x)(f_T^{t}(x)-f_{\mu^\star}^{t}(x))d\lambda^p(x)\\
&
    \leq \Big\langle \tilde{\Gamma}^{\gamma-\eta,-2}_\epsilon(\overline{h}),f_T^{t}-f_{\mu^\star}^{t}\Big\rangle_{L^2(\mathbb{R}^p)}^{\frac{\beta+\gamma}{\beta+\eta}}\Big\langle \tilde{\Gamma}^{\gamma,-2}_\epsilon(\overline{h}),\Gamma^{\beta}(f_T^{t}-f_{\mu^\star}^{t})\Big\rangle_{L^2(\mathbb{R}^p)}^{\frac{\eta -\gamma}{\beta+\eta}}
\end{align*}
and 
\begin{align*}
    \sum \limits_{j=0}^\infty \sum \limits_{l=1}^{2^p}  \sum \limits_{w \in \mathbb{Z}^p} & 2^{j(\beta+\gamma)}|\alpha_{\tilde{\Gamma}^{0,-2}_\epsilon(\overline{h})}(j,l,w)| |\alpha_{f_T}(j,l,w)-\alpha_{f_{\mu^\star}}(j,l,w)|\\
    & \leq  \sum \limits_{j=0}^\infty \sum \limits_{l=1}^{2^d} \sum \limits_{w \in \mathbb{Z}^p}(1+j)^{-2}C2^{-jp}\mathds{1}_{\{supp(\psi_{jlw}\cap B^p(0,K))\neq \varnothing\}}\\
 & \leq C  \sum \limits_{j=0}^\infty (1+j)^{-2}\leq C.
\end{align*}
Define $F_h:\mathcal{M}^{\star}\rightarrow \mathbb{R}$ by 
$$
F_h(x)=\int_{T^{-1}(\{x \})}\tilde{\Gamma}^{\gamma-\eta,-2}_\epsilon(\overline{h})(z)\ \frac{\Theta(4\|z-T(z)\|^2/t^2)}{\int_{\mathbb{R}^{p-d}} \Theta(4\|y\|^2/t^2)d\lambda^{p-d}(y)}d\lambda_{T^{-1}(\{x \})}(z).
$$
Then 
\begin{align*}
    &\int_{\mathbb{R}^p}   \tilde{\Gamma}^{\gamma-\eta,-2}_\epsilon(\overline{h})(x)(f_T^{t}(x)-f_{\mu^\star}^{t}(x))d\lambda^p(x) \\
& = \int_{\mathcal{M}^{\star}} \int_{T^{-1}(\{x \})} \tilde{\Gamma}^{\gamma-\eta,-2}_\epsilon(\overline{h})(z) (f_T^{t}(z)-f_{\mu^\star}^{t}(z))|\text{ap}_d(\nabla T(z))|^{-1}d\lambda^{p-d}(z)d\lambda_{\mathcal{M}^{\star}}(x)\\
& = \int_{\mathcal{M}^{\star}} \int_{T^{-1}(\{x \})} \tilde{\Gamma}^{\gamma-\eta,-2}_\epsilon(\overline{h})(z) \frac{(f_T(T(z))-f_{\mu^\star}(T(z)))\Theta(4\frac{\|z-T(z)\|^2}{t^2})}{\int_{\mathbb{R}^{p-d}} \Theta(4\|y\|^2/t^2)d\lambda^{p-d}(y)}d\lambda^{p-d}(z)d\lambda_{\mathcal{M}^{\star}}(x)\\
& = \int_{\mathcal{M}^{\star}} F_h(x)(f_T(x)-f_{\mu^\star}(x))d\lambda_{\mathcal{M}^{\star}}(x).
\end{align*}

As $T$ verifies point iii) of Definition \ref{defi:compatibility}, we have $F_h \in  \mathcal{H}^{\eta}_{C}(\mathcal{M}^\star,\mathbb{R})$ so it can be  extended using Proposition \ref{prop:extensionH} to a map in $\mathcal{H}^{\eta}_{C}(\mathbb{R}^p,\mathbb{R})$. Then, we can conclude that 
\begin{align*}
\int_{\mathcal{M}^{\star}} h(x)(f_T(x)-f_{\mu^\star}(x))d\lambda_{\mathcal{M}^{\star}}(x) \leq C \sup \limits_{h \in \mathcal{H}^{\eta}_1}\left(\int_{\mathcal{M}^{\star}} h(x)(f_T(x)-f_{\mu^\star}(x))d\lambda_{\mathcal{M}^{\star}}(x)\right)^{\frac{\beta+\gamma}{\beta+\eta}}.
\end{align*}
\end{proof}

\subsection{Proof of Theorem \ref{theo:theineq}}\label{sec:theo:theineq}

\begin{proof}
Let $C_\eta$ the constant given by Theorem \ref{theo:existencecompmap} such that if 
        $$
    d_{\mathcal{H}^\eta_1}(\mu,\mu^\star)\leq C_{\eta}^{-1},
     $$
then there exists a map $T$ being $(\mathcal{M},\mathcal{M}^\star)$-compatible.

Suppose first that $d_{\mathcal{H}^\eta_1}(\mu,\mu^\star)> C_{\eta}^{-1}$, then
$$d_{\mathcal{H}^\gamma_1}(\mu,\mu^\star)\leq 2K \leq 2KC_{\eta}d_{\mathcal{H}^\eta_1}(\mu,\mu^\star)\leq Cd_{\mathcal{H}^\eta_1}(\mu,\mu^\star)^{\frac{\beta+\gamma}{\beta+\eta}}.$$
Suppose now that $d_{\mathcal{H}^\eta_1}(\mu,\mu^\star)\leq C_{\eta}^{-1}$. Using Corollary \ref{coro:ineq without reg} we have for all $\epsilon\in (0,1)$
\begin{equation}\label{eq:baseoftheo2}
d_{\mathcal{H}^\gamma_1}(\mu,\mu^\star)\leq C\log(\epsilon^{-1})^2 d_{\mathcal{H}^1_1}(\mu,\mu^\star)^\frac{\eta-\gamma}{\eta-1}d_{\mathcal{H}^\eta_1}(\mu,\mu^\star)^\frac{\gamma-1}{\eta-1} +\epsilon.
\end{equation}
Suppose first that $\eta\leq \beta+1$. Using lemmas \ref{lemma:firstterm} and \ref{lemma:secondterm} we have
\begin{align*}
    d_{\mathcal{H}^1_1}(\mu,\mu^\star) & \leq C\log(\epsilon^{-1})^4 \left(\sup \limits_{\substack{h \in \mathcal{H}^{\eta}_1\\ h_{|\mathcal{M}^\star}=0}}\left(\int_{\mathcal{M}}h(x)d\mu(x)\right)^{\frac{\beta+1}{\beta+\eta}} +d_{\mathcal{H}^{\eta}_1}(T_{\# \mu},\mu^{\star})^{\frac{\beta+1}{\beta+\eta}}\right)+\epsilon.
\end{align*}
On one hand we have 
$$\sup \limits_{\substack{h \in \mathcal{H}^{\eta}_1\\ h_{|\mathcal{M}^\star}=0}}\int_{\mathcal{M}}h(x)d\mu(x)\leq \sup \limits_{h \in \mathcal{H}^{\eta}_1}\int_{\mathcal{M}}h(x)d\mu(x)-\int_{\mathcal{M}^\star}h(x)d\mu^\star(x)=d_{\mathcal{H}^\eta_1}(\mu,\mu^\star),$$
on the other hand
$$d_{\mathcal{H}^{\eta}_1}(T_{\# \mu},\mu^{\star})=\sup \limits_{h \in \mathcal{H}^{\eta}_1}\int_{\mathcal{M}}h(T(x))d\mu(x)-\int_{\mathcal{M}^\star}h(T(x))d\mu^\star(x)\leq C d_{\mathcal{H}^\eta_1}(\mu,\mu^\star),$$
as any map $h\circ T \in \mathcal{H}^{\eta}_C(\mathcal{M}^{\star t/4},\mathbb{R})$ can be extended to a map in $\mathcal{H}^{\eta}_C(\mathbb{R}^p,\mathbb{R})$ using Proposition \ref{prop:extensionH}. Therefore, we have
$$d_{\mathcal{H}^1_1}(\mu,\mu^\star)\leq  C \log(\epsilon^{-1})^4 d_{\mathcal{H}^\eta_1}(\mu,\mu^\star)^{\frac{\beta+1}{\beta+\eta}}+\epsilon.
$$

Then plugging this in \eqref{eq:baseoftheo2} we obtain
\begin{align*}
    d_{\mathcal{H}^\gamma_1}(\mu,\mu^\star) & \leq C \log(\epsilon^{-1})^6 d_{\mathcal{H}^\eta_1}(\mu,\mu^\star)^{\frac{\beta+1}{\beta+\eta}\frac{\eta-\gamma}{\eta-1}+\frac{\gamma-1}{\eta-1}}+\epsilon\\
    & = C \log(\epsilon^{-1})^6 d_{\mathcal{H}^\eta_1}(\mu,\mu^\star)^{\frac{\beta+\gamma}{\beta+\eta}}+\epsilon
\end{align*}
so taking $\epsilon=d_{\mathcal{H}^\eta_1}(\mu,\mu^\star)^{\frac{\beta+\gamma}{\beta+\eta}}$, we get the result.

Now if $\eta > \beta+1$, using lemmas \ref{lemma:firstterm} and \ref{lemma:secondterm} we have again
\begin{align*}
d_{\mathcal{H}^1_1}(\mu,\mu^\star) & \leq d_{\mathcal{H}^1_1}(\mu,T_{\#\mu})+d_{\mathcal{H}^1_1}(T_{\#\mu},\mu^\star)\\
& \leq C \log(\epsilon^{-1})^4d_{\mathcal{H}^{\beta+1}_1}(\mu,\mu^\star)^\frac{\beta+1}{2\beta+1} +\epsilon.
\end{align*}
And using again Corollary  \ref{coro:ineq without reg} we have
$$d_{\mathcal{H}^{\beta+1}_1}(\mu,\mu^\star)\leq C\log(\epsilon^{-1})^2 d_{\mathcal{H}^1_1}(\mu,\mu^\star)^\frac{\eta-(\beta+1)}{\eta-1}d_{\mathcal{H}^\eta_1}(\mu,\mu^\star)^\frac{\beta}{\eta-1} +\epsilon
$$
so 
\begin{align*}
d_{\mathcal{H}^1_1}(\mu,\mu^\star)\leq C \log(\epsilon^{-1})^6d_{\mathcal{H}^1_1}(\mu,\mu^\star)^{\frac{\eta-(\beta+1)}{\eta-1}\frac{\beta+1}{2\beta+1}}d_{\mathcal{H}^\eta_1}(\mu,\mu^\star)^{\frac{\beta}{\eta-1}\frac{\beta+1}{2\beta+1}} +\epsilon.
\end{align*}
Then, taking $\epsilon =d_{\mathcal{H}^1_1}(\mu,\mu^\star)d_{\mathcal{H}^\eta_1} (\mu,\mu^\star)$,  we have
\begin{align*}
d_{\mathcal{H}^1_1}(\mu,\mu^\star)\leq C \log(d_{\mathcal{H}^\eta_1}(\mu,\mu^\star)^{-1})^6d_{\mathcal{H}^1_1}(\mu,\mu^\star)^{\frac{\eta-(\beta+1)}{\eta-1}\frac{\beta+1}{2\beta+1}}d_{\mathcal{H}^\eta_1}(\mu,\mu^\star)^{\frac{\beta}{\eta-1}\frac{\beta+1}{2\beta+1}},
\end{align*}
which gives 
\begin{align*}
d_{\mathcal{H}^1_1}(\mu,\mu^\star)^{(\eta-1)(2\beta+1)-(\eta-1-\beta)(\beta+1)}\leq & C \log(d_{\mathcal{H}^\eta_1}(\mu,\mu^\star)^{-1})^{6(\eta-1)(2\beta+1)}d_{\mathcal{H}^\eta_1}(\mu,\mu^\star)^{\beta(\beta+1)}.
\end{align*}
As $(\eta-1)(2\beta+1)-(\eta-1-\beta)(\beta+1))=\beta(\eta+\beta)$ we obtain 
\begin{align*}
d_{\mathcal{H}^1_1}(\mu,\mu^\star)\leq & C \log(d_{\mathcal{H}^\eta_1}(\mu,\mu^\star)^{-1})^{6\frac{(\eta-1)(2\beta+1)}{\beta(\eta+\beta)}}d_{\mathcal{H}^\eta_1}(\mu,\mu^\star)^{\frac{\beta+1}{\beta+\eta}}.
\end{align*}
Then plugging this result in \eqref{eq:baseoftheo2}, we obtain for $\epsilon=d_{\mathcal{H}^\eta_1}(\mu,\mu^\star)$,
\begin{align*}
    d_{\mathcal{H}^\gamma_1}(\mu,\mu^\star) & \leq C \log(d_{\mathcal{H}^\eta_1}(\mu,\mu^\star)^{-1})^{C_2} d_{\mathcal{H}^\eta_1}(\mu,\mu^\star)^{\frac{\beta+1}{\beta+\eta}\frac{\eta-\gamma}{\eta-1}+\frac{\gamma-1}{\eta-1}}\\
    & = C \log(d_{\mathcal{H}^\eta_1}(\mu,\mu^\star)^{-1})^{C_2} d_{\mathcal{H}^\eta_1}(\mu,\mu^\star)^{\frac{\beta+\gamma}{\beta+\eta}}.
\end{align*}
\end{proof}

\subsection{Proof of Proposition \ref{prop:ineqgammaleqone}}\label{sec:prop:ineqgammaleqone}
\begin{proof}
As explained at the beginning of the proof of Theorem \ref{theo:theineq}, we can suppose that the distance $d_{\mathcal{H}^\eta_1}(\mu,\mu^\star)$ is small enough so that there exists a $(\mathcal{M},\mathcal{M}^\star)$-compatible map
$T$. We have
\begin{align*}
d_{\mathcal{H}^\gamma_1}(\mu,\mu^\star) & \leq d_{\mathcal{H}^\gamma_1}(\mu,T_{\#\mu})+d_{\mathcal{H}^\gamma_1}(T_{\#\mu},\mu^\star)
\end{align*}
Using Lemma \ref{lemma:secondterm} we obtain 
\begin{align*}d_{\mathcal{H}^\gamma_1}(T_{\#\mu},\mu^\star)& \leq C\log(\epsilon^{-1})^2 d_{\mathcal{H}^\eta_1}(T_{\#\mu},\mu^\star)^{\frac{\beta+\gamma}{\beta+\eta}} + \epsilon\\
& \leq C  d_{\mathcal{H}^\eta_1}(T_{\#\mu},\mu^\star)^{\frac{\gamma}{\eta}}
\end{align*}
for $\epsilon=d_{\mathcal{H}^\eta_1}(T_{\#\mu},\mu^\star)^{\frac{\gamma}{\eta}}$. On the other hand,
\begin{align*}
    d_{\mathcal{H}^\gamma_1}(\mu,T_{\#\mu}) & =\sup \limits_{h\in \mathcal{H}^{\gamma}_1}\int_{\mathcal{M}}(h(x)-h(T(x)))f_{\mu}(x)d\lambda_{\mathcal{M}}(x)\\
    & \leq \int_{\mathcal{M}} \|x-T(x)\|^\gamma f_{\mu}(x)d\lambda_{\mathcal{M}}(x)\\
    & \leq \left(\int_{\mathcal{M}} \|x-T(x)\|^\eta f_{\mu}(x)d\lambda_{\mathcal{M}}(x)\right)^{\frac{\gamma}{\eta}}
\end{align*}
using Jensen's inequality. The function $H:\mathcal{M}^{\star t}\rightarrow \mathbb{R}$ defined by 
$$H(x)=\|x-T(x)\|^\eta,$$
belongs to $\mathcal{H}^\eta_C(\mathcal{M}^{\star t}, \mathbb{R})$ and can therefore be extend to a map in $\mathcal{H}^\eta_C(\mathbb{R}^p, \mathbb{R})$ using Proposition \ref{prop:extensionH}.
    
\end{proof}

\subsection{Proof of Proposition \ref{prop:lambdamin}}\label{sec:prop:lambdamin}

\begin{proof}
Let $s \in (0,r_{\mathcal{M}^\star})$, we are going to show that $W_1(\mu,\mu^\star)$ is bounded below which will give us that $d_{\mathcal{H}^\eta_1}(\mu,\mu^\star)$ is also bounded below by Lemma \ref{lemma:hausdorff}.
    
Suppose first that $\mathbb{H}(\mathcal{M},\mathcal{M}^\star)>s$ and
for $y\in \mathcal{M}^\star$ such that $d(y,\mathcal{M})\geq s$,    define $D^y_s:\mathbb{R}^p\rightarrow \mathbb{R}$ as
$$D_s^y(x)=(s-\|x-y\|)\vee 0.$$
As $D_s$ is $1$-Lipschitz, we have
\begin{align*}
    W_1(\mu,\mu^\star) & \geq \left(\int_{\mathcal{M}^\star}D_s^y(x)f_{\mu^\star}(x)d\lambda_{\mathcal{M}^\star}(x)-\int_\mathcal{M}D_s^y(x)f_\mu(x)d\lambda_\mathcal{M}(x)\right)\\
   &  = \int_{\mathcal{M}^\star\cap B^p(y,s)}(s-\|x-y\|)f_{\mu^\star}(x)d\lambda_{\mathcal{M}^\star}\\
   & \geq Cs^dK^{-1}
\end{align*}
which gives the result. Now if $\mathbb{H}(\mathcal{M},\mathcal{M}^\star)\leq s$, by Lemma \ref{lemma:neardiffeo} we have that for $s>0$ small enough, $\pi$ the canonical projection onto $\mathcal{M}^\star$, is a diffeomorphism from $\mathcal{M}$ to $\mathcal{M}^\star$. Le be  $z\in \mathcal{M}$ such that $ f_\mu(z)\leq s$. Define $D_s^z:\mathbb{R}^p\rightarrow \mathbb{R}$ as
$$D_s(x)=(\frac{s}{2}\wedge(s-\|x-\pi(z)\|))\vee 0.$$

As $D_s$ is $1$-Lipschitz, we have
\begin{align*}
    W_1(\mu,\mu^\star) & \geq \int_{\mathcal{M}^\star}D_s^z(x)f_{\mu^\star}(x)d\lambda_{\mathcal{M}^\star}(x)-\int_\mathcal{M}D_s^z(x)f_\mu(x)d\lambda_\mathcal{M}(x)
\end{align*}
and 
\begin{align*}
\int_{\mathcal{M}^\star}D_s^z(x)f_{\mu^\star}(x)d\lambda_{\mathcal{M}^\star}(x)& =\int_{\mathcal{M}^\star\cap B^p(\pi(z),s)}\frac{s}{2}\wedge(s-\|x-\pi(z)\|)f_{\mu^\star}(x)d\lambda_{\mathcal{M}^\star}(x)\\
 & \geq \int_{\mathcal{M}^\star\cap B^p(\pi(z),s/2)}\frac{s}{2}f_{\mu^\star}(x)d\lambda_{\mathcal{M}^\star}(x)\\
 & \geq C s^{d+1}K^{-1}.
\end{align*}
On the other hand,

\begin{align*}
\int_\mathcal{M}D_s^z(x)f_\mu(x)d\lambda_\mathcal{M}(x)& =\int_{\mathcal{M}\cap B^p(\pi(z),s)}\frac{s}{2}\wedge(s-\|x-\pi(z)\|)f_{\mu}(x)d\lambda_{\mathcal{M}}(x)\\
 & \leq \int_{\mathcal{M}\cap B^p(\pi(z),s)}sf_{\mu}(x)d\lambda_{\mathcal{M}}(x)\\
& = s\int_{\varphi_x(\mathcal{M}\cap B^p(\pi(z),s))}f_{\mu}(\varphi_x^{-1}(u))d\varphi_{x\# \lambda_{\mathcal{M}}}(u)\\
& \leq  s\int_{B^d(\varphi_x(z),2s)}f_{\mu}(\varphi_x^{-1}(u))K^{d}d\lambda^d(u)\\
& \leq C s\int_{B^d(\varphi_x(\pi(z)),2s)}(f_{\mu}(z)+K\|\varphi_x^{-1}(z)-u\|))d\lambda^d(u)\\
& \leq C s\int_{B^d(\varphi_x(\pi(z)),2s)}(1+2K)sd\lambda^d(u)\\
& \leq Cs^{d+2}.  
\end{align*}
Therefore we have
\begin{align*}
    W_1(\mu,\mu^\star) & \geq Cs^{d+1}(1-C_2s),
\end{align*}
so for $s$ small enough we get the result.
\end{proof}

\subsection{Additional proofs from Section \ref{sec:classicalineq}}
\subsubsection{Proof of Proposition \ref{propo:ineqfulldimbes}}\label{sec:addiproofs}
\begin{proof}
    Using \eqref{eq:ipmdebase}, we are looking for the smallest $q\in (0,1)$ such that $d_{\mathcal{B}^{s-\frac{q^\star}{q}\tau}_{\infty,\infty}}(f,f^\star)$ is finite. Taking $f,f^\star \in \mathcal{B}^{\beta,2}_{\infty,\infty}$ compactly supported in $B^p(0,K)$ we have 

\begin{align*}
    & d_{\mathcal{B}^{s-\frac{q^\star}{q}\tau}_{\infty,\infty}} (f,f^\star)\\
    &  = \sup \limits_{h\in \mathcal{B}^{s-\frac{q^\star}{q}\tau}_{\infty,\infty}(1)} \sum \limits_{j=0}^\infty \sum \limits_{l=1}^{2^p} \sum \limits_{w \in \mathbb{Z}^p} \alpha_{h}(j,l,w)(\alpha_{f}(j,l,w)-\alpha_{f^\star}(j,l,w))\\
     & =\sum \limits_{j=0}^\infty  \sum \limits_{l=1}^{2^p} \sum \limits_{w \in \mathbb{Z}^p}2^{-j(s-\frac{q^\star}{q}\tau+p/2)} \alpha_{h}(j,l,w)|\alpha_{f}(j,l,w)-\alpha_{f^\star}(j,l,w)|\\
     & \leq \sum \limits_{j=0}^\infty  \sum \limits_{l=1}^{2^p} \sum \limits_{w \in \mathbb{Z}^p}2^{-j(s-\frac{q^\star}{q}\tau+p/2)} 2^{-j(\beta+p/2)}(1+j)^{-2}\|f-f^\star\|_{\mathcal{B}^{\beta,2}_{\infty,\infty}}\mathds{1}_{\{supp(\psi_{jlw})\cap B^p(0,K)\neq\varnothing\}}\\
     & \leq C \|f-f^\star\|_{\mathcal{B}^{\beta,2}_{\infty,\infty}} \sum \limits_{j=0}^\infty  2^{-j(s-\frac{q^\star}{q}\tau+\beta)}(1+j)^{-2}.
\end{align*}
The last quantity is finite iff $\beta \geq \frac{q^\star}{q}\tau-s$ which gives $\frac{1}{q}\leq \frac{\beta+s}{\beta+s+\tau}$. 
\end{proof}

\subsubsection{Proof of Theorem \ref{eq:ineqinfulldim}}\label{sec:prevuedecethe}
\begin{proof}
Using lemma \ref{lemma:inclusions} we have that $f-f^\star \in \mathcal{B}^{\beta}_{\infty,\infty}(C)$ and $h\in \mathcal{B}^{\gamma}_{\infty,\infty}(C)$, then for $\epsilon\in (0,1)$ we have 
\begin{align*}
    \int_{\mathbb{R}^p}& h(x)(f(x)-f^\star(x))d\lambda^p(x)\\
    = &\sum \limits_{j=0}^\infty \sum \limits_{l=1}^{2^p} \sum \limits_{w \in \mathbb{Z}^p} \alpha_{h}(j,l,w)(\alpha_{f}(j,l,w)-\alpha_{f^\star}(j,l,w))\\
     \leq&  \sum \limits_{j=0}^{\log_2(\epsilon^{-1})} \sum \limits_{l=1}^{2^p} \sum \limits_{w \in \mathbb{Z}^p} \alpha_{h}(j,l,w)(\alpha_{f}(j,l,w)-\alpha_{f^\star}(j,l,w))\\
      & + \sum \limits_{j=\log_2(\epsilon^{-1})+1}^\infty \sum \limits_{l=1}^{2^p} \sum \limits_{w \in \mathbb{Z}^p} 2^{-j(\gamma+\beta+p)} \mathds{1}_{\{supp(\psi_{jlw})\cap B^p(0,K)\neq\varnothing\}}\\
      \leq & \log_2(\epsilon^{-1})^2\sum \limits_{j=0}^{\log_2(\epsilon^{-1})} \sum \limits_{l=1}^{2^p} \sum \limits_{w \in \mathbb{Z}^p}(1+j)^{-2} |\alpha_{h}(j,l,w)(\alpha_{f}(j,l,w)-\alpha_{f^\star}(j,l,w))|\\
      & + C\epsilon^{\gamma+\beta}\\
       \leq & C \log_2(\epsilon^{-1})^2 d_{\mathcal{B}^{\gamma}_{\infty,\infty}(1)}(\Gamma^{0,-2}(f),\Gamma^{0,-2}(f^\star))+ C\epsilon^{\gamma+\beta}.
\end{align*}
As $f,f^\star \in \mathcal{B}^{\beta,2}_{\infty,\infty}(C)$, we can now use Proposition \ref{propo:ineqfulldimbes} and obtain 
\begin{align*}
d_{\mathcal{B}^{\gamma}_{\infty,\infty}(1)}(\Gamma^{0,-2}(f),\Gamma^{0,-2}(f^\star)) & \leq C d_{\mathcal{B}^{\alpha}_{\infty,\infty}(1)}(\Gamma^{0,-2}(f),\Gamma^{0,-2}(f^\star))^{\frac{\beta+\gamma}{\beta+\alpha}}\\
    & =d_{\mathcal{B}^{\gamma,2}_{\infty,\infty}(1)}(f,f^\star)^{\frac{\beta+\gamma}{\beta+\alpha}}\\
    & \leq C d_{\mathcal{H}^{\gamma}_{C}}(f,f^\star)^{\frac{\beta+\gamma}{\beta+\alpha}},
\end{align*}
where we used Lemma \ref{lemma:inclusions} for the last inequality. Putting everything together and taking $\epsilon =d_{\mathcal{H}^{\alpha}_{1}}(f,f^\star)\wedge 1/2 $ we obtain the result.
\end{proof}

\end{document}